\documentclass[11pt]{article}
\usepackage{graphicx}
\usepackage{stmaryrd}
\usepackage{amssymb}
\usepackage{amsthm}
\usepackage{dsfont}
\usepackage{hyperref}
\usepackage{mathrsfs}
\usepackage{stmaryrd}
\usepackage[lite,initials,msc-links]{amsrefs}
\usepackage{amsmath}
\usepackage{centernot}
\usepackage{enumerate}
\usepackage{enumitem}
\usepackage{verbatim}
\usepackage[margin=1.2in]{geometry}
\usepackage{xcolor}
\usepackage{mathtools}
\usepackage{color,soul}
\usepackage{sidecap,subcaption}
\usepackage[section]{placeins}
\usepackage{wrapfig}
\usepackage[utf8]{inputenc}

\usepackage[normalem]{ulem}

\newtheorem{theorem}{Theorem}[section]
\newtheorem{lemma}[theorem]{Lemma}
\newtheorem{proposition}[theorem]{Proposition}
\newtheorem{corollary}[theorem]{Corollary}
\numberwithin{equation}{section}
\newtheorem{definition}[theorem]{Definition}

\theoremstyle{remark} \newtheorem{remark}[theorem]{Remark} \numberwithin{equation}{section}
\numberwithin{figure}{section}

\newcommand{\Ab}{\mathbb{A}}

\newcommand{\Cb}{\mathbb{C}}

\newcommand{\Eb}{\mathbb{E}}

\newcommand{\Qb}{\mathbb{Q}}
\newcommand{\Rb}{\mathbb{R}}

\newcommand{\Ub}{\mathbb{U}}

\newcommand{\Zb}{\mathbb{Z}}

\newcommand{\Ac}{\mathcal{A}}

\newcommand{\Lc}{\mathcal{L}}

\newcommand{\con}[1]{ \xleftrightarrow{#1}}

\newcommand{\IP}{\mathbf{P}}

\newcommand{\clust}{\mathscr{C}}
\newcommand{\coin}{\epsilon}

\newcommand{\ep}{\varepsilon}
\newcommand{\set}{{\bf n}}

\newcommand{\sets}{\Omega}

\newcommand{\flows}{\mathcal{F}}

\newcommand{\match}{M}

\newcommand{\cont}{C}
\newcommand{\matchs}{\mathcal{M}}

\makeatletter
\newcommand*\bigcdot{\mathpalette\bigcdot@{.5}}
\newcommand*\bigcdot@[2]{\mathbin{\vcenter{\hbox{\scalebox{#2}{$\m@th#1\bullet$}}}}}
\makeatother

\renewcommand{\i}{\mathrm{i}}
\newcommand{\sgn}{\textnormal{sgn}}
\newcommand{\cur}{\mathbf{n}}

\newcommand{\dcurr}{\textnormal{dcur}}
\newcommand{\flow}{\textnormal{flow}}

\newcommand{\weight}{\textnormal{w}}
\newcommand{\odd}{\textnormal{odd}}

\newcommand{\gask}{\mathrm{gask}}
\newcommand{\lp}{\mathcal {L}}

\newcommand{\ol}{\overline}

\newcommand{\n}{{\mathbf n}}
\newcommand{\eps}{\varepsilon}

\definecolor{slightblue}{rgb}{.8, .8, 1}
\definecolor{hair}{RGB}{100,225,190}
\definecolor{ruby}{RGB}{220,50,120}
\definecolor{grass}{RGB}{150,220,110}

\title{Conformal invariance of double random currents I: identification of the limit}

\author{Hugo Duminil-Copin\thanks{Institut des Hautes \'Etudes Scientifiques} \thanks{Universit\'e de Gen\`eve} \and Marcin Lis\thanks{Technische Universit\"{a}t Wien} \and Wei Qian\thanks{City University of Hong Kong} \thanks{CNRS and Laboratoire de Math\'{e}matiques d'Orsay, Universit\'{e} Paris-Saclay}}

\date{\today}

\begin{document}
\maketitle

\begin{abstract}
This is the first of two papers devoted to the proof of conformal invariance of the critical double random current model on the square lattice. More precisely, we show the convergence of loop ensembles obtained by taking the cluster boundaries in the sum of two independent currents with free and wired boundary conditions. The strategy is first to prove convergence of the associated height function to the continuum Gaussian free field, and then to characterize the scaling limit of the loop ensembles as certain local sets of this Gaussian Free Field. In this paper, we identify uniquely the possible subsequential limits of the loop ensembles. Combined with \cite{DumLisQia21}, this completes the proof of conformal invariance.
\end{abstract}

\section{Introduction}\label{sec:intro}

\subsection{Motivation and overview}\label{sec:motivation}

The rigorous understanding of Conformal Field Theory (CFT) and Conformally Invariant random objects via the developments of the Schramm-Loewner Evolution (SLE) and  its relations to the Gaussian Free Field (GFF)
 has progressed greatly in the last twenty-five years. It is fair to say that once a discrete lattice model is proved to be conformally invariant in the scaling limit, most of what mathematical physicists are interested in can be exactly computed using the powerful tools in the continuum.

A large class of discrete lattice models are conjectured to have interfaces that converge in the scaling limit to SLE$_\kappa$ type curves for $\kappa\in(0,8]$.
Unfortunately, such convergence results are only proved for a handful of models, including the loop-erased random walk \cite{MR1776084} and the uniform spanning tree  \cite{LSW04} (corresponding to $\kappa=2$ and $8$), the Ising model \cite{CheSmi12} and its FK representation \cite{smirnov} (corresponding to $\kappa=3$ and $16/3$), Bernoulli site percolation on the triangular lattice \cite{Smi01} (corresponding to $\kappa=6$). 
Known proofs involve a combination of exact integrability\footnote{Only approximately for site percolation on the triangular lattice.} enabling the computation of certain discrete observables, and of discrete complex analysis to imply the convergence in the scaling limit to holomorphic/harmonic functions satisfying certain boundary value problems that are naturally conformally covariant.

To upgrade the result from conformal covariance of these ``witness'' observables to the convergence of interfaces in the system, one needs an additional ingredient. 
In some cases, when properties of the discrete models are sufficiently nice (typically tightness of the family of interfaces, mixing type properties, etc), a clever martingale argument introduced by Oded Schramm enables to prove convergence of interfaces to SLEs and CLEs. This last step involves the \emph{spatial Markov properties} of the discrete model in a crucial fashion. We refer to the proofs of conformal invariance of interfaces in Bernoulli site percolation, the Ising model, the FK Ising model, or the harmonic explorer for examples. Unfortunately, the discrete properties of the model are  sometimes not sufficiently nice to implement this martingale argument and there are still many remaining examples for which the scaling limit of the interfaces cannot be easily deduced from the conformal invariance of certain observables -- most notably for the case of the double dimer model, for which an important breakthrough was performed by Kenyon in \cite{Ken14}, followed by a series of impressive papers \cite{Dub18,BasChe18}. 

In this paper we prove convergence of the nested inner and outer boundaries of clusters in the critical double random current model with free boundary conditions, as well as in its dual model with wired boundary conditions, to level loops of a GFF. In particular, the  outer boundaries of clusters in the critical double random current model with free boundary conditions converge to
CLE$_4$. 
The random current model has proved to be a very powerful tool to understand the Ising model. Its applications range from  correlation inequalities \cite{GHS}, exponential decay in the off-critical regime \cite{AizBarFer87,DumTas15,DumGosRao20}, classification of Gibbs states \cite{Rao17}, continuity of the phase transition \cite{AizDumSid15}, etc. Even in two dimensions, where a number of other tools are available, new developments have been made possible via the use of this representation \cite{DumLis,ADTW,LupWer}. 
In particular, as mentioned at the end of this Section~\ref{intro:drc}, the scaling limit of the double random current gives access to the scaling limit of spin correlations in the Ising model.
For a more exhaustive account of random currents, we refer the reader to \cite{Dum16}.

Convergence to SLE$_4$ type curves were previously proved for the harmonic explorer \cite{SchShe05}, zero contour lines of the discrete GFF \cite{MR2486487} (also in the cable-graph representation~\cite{ALS}), the zero contour lines of the Ginzburg-Landau $\nabla \phi$ interface model \cite{jason1, jason2}, and cluster boundaries of a random walk loop-soup with the critical intensity \cite{Lupu,BCL}.

As mentioned above, our proof does not follow the martingale strategy. Instead, it relies on a coupling between the double random current and a naturally associated height function, and can be decomposed into three main steps (see the next sections for more details):
\begin{itemize}
\item[(i)] Proving the joint tightness of the family of interfaces in the double random current model and the height function, as well as certain properties of the joint coupling.
\item[(ii)] Proving convergence of the height function to the GFF.
\item[(iii)] In the continuum, identifying the scaling limit of the interfaces using properties of the GFF and its local sets. 
\end{itemize}
Each of the three previous steps involves quite different branches of probability. The first one extensively uses percolation-type arguments for dependent percolation models. The second one
concerns a height function studied already by Dub\'{e}dat~\cite{Dub}, and Boutilier and de Tili\`{e}re~\cite{BoudeT}. However, unlike in \cite{Dub,BoudeT}, it
harvests a link between a percolation model (the double random current) and dimers. Moreover, it uses techniques introduced by Kenyon to prove convergence of the dimer height function, but with a new twist as the proof relies heavily on fermionic observables introduced by Chelkak and Smirnov to prove conformal invariance of the Ising model, as well as a delicate result on the double random current model (see Section~\ref{sec:input discrete}) helping identifying the boundary conditions. Finally, the last step relies on properties of the local sets of the GFF introduced by Schramm and Sheffield \cite{MR3101840}, and in particular on the two-valued local sets introduced by Aru, Sep\'ulveda and Werner \cite{MR3936643}. 
This step crucially uses the spatial Markov properties of the interfaces and the associated height function deduced from step (ii), but also establishes a certain spatial Markov property of the outer boundaries of the clusters in the continuum limit (which turn out to be CLE$_4$ of the limiting GFF) which is unknown in the discrete.

Part (i) of the proof is postponed to the second paper \cite{DumLisQia21}. In this paper, we focus on (ii) and (iii).

In the reminder of this introduction,  we state the results of the convergence of the interfaces in the double random current models with free and wired boundary conditions (Section~\ref{intro:drc}) and the convergence of the height function associated with the double random currents (Section~\ref{subsec:2cvg_gff}). 
In reality, the double random currents with free and wired boundary conditions can be coupled on the primal and dual graphs and be associated with the same height function, so that these three objects converge jointly. In particular, we have more precise descriptions on their joint limit, but we postpone these further results to Section~\ref{sec:bah} for simplicity.

\bigbreak
\paragraph{Notation}
Throughout the article we work with planar graphs embedded in the plane.
We will consider \emph{Jordan domains} $D\subsetneq \mathbb C$, i.e., simply connected domains whose boundary $\partial D$ is a Jordan curve. In certain situations we will impose a regularity condition on $\partial D$, namely that it is a $C^1$ curve.

Below, we will speak of convergence of random variables taking values in families of loops contained in $D$, and distributions (generalized functions). While the latter is classical and has a well-defined associated topology, we provide some details on the former. To this end, let $\mathfrak C=\mathfrak C(D)$ be the collection of locally finite families $\mathcal F$ of non-self-crossing loops contained in $D$ that do not intersect each other. 
Inspired by~\cite{AizBur99}, we define a metric on $\mathfrak C$, 
\begin{align} \label{eq:CNdistance}
{\bf d}(\mathcal F,\mathcal F')\le \ep \ \Longleftrightarrow\ \begin{array}{c} \exists f:\mathcal F_\ep\rightarrow \mathcal F'\text{ one-to-one s.t.~}\forall \gamma\in \mathcal F_\ep,d(\gamma,f(\gamma))\le \ep\\
\text{ and }\text{similarly when exchanging $\mathcal F'$ and $\mathcal F$}\end{array},
\end{align}
where, $\mathcal F_\ep$ is the collection of loops in $\mathcal F$ with a diameter larger than $\ep$, and for two loops $\gamma_1$ and $\gamma_2$, we set 
\begin{align} \label{eq:curvedist}
d(\gamma_1,\gamma_2):=\inf\sup_{t\in\mathbb S^1}|\gamma_1(t)-\gamma_2(t)|,
\end{align}
with the infimum running over all continuous bijective parametrizations of the loops $\gamma_1$ and~$\gamma_2$ by $\mathbb S^1$.

We will say that simply connected graphs $D^\delta \subset \delta \mathbb Z^2$ \emph{approximate} $D$ if 
$d(\partial D^\delta , \partial D)\to 0$ as $\delta \to 0 $, where $d$ is as in \eqref{eq:curvedist}.

\subsection{Convergence of interfaces in double random currents}\label{intro:drc}
Consider a finite graph $G=(V,E)$ with vertex set $V$ and edge set $E$. 
A {\em current} $\n$ on $G$ is a function defined on the undirected edges $\{v,v'\}\in E$ and taking values in the natural numbers.  
The current's set of {\em sources} is defined as the set
\begin{equation}
\partial\n \, := \, \Big\{ v\in V :\, \sum_{v'\in V:v'\sim v} \n_{\{ v,v'\}}\text{ is odd}\Big \},
\end{equation}
where $v'\sim v$ means that $\{v,v'\}\in E$.

Let $\sets^B$ be the set of currents with the set of sources equal to $B$. When $B=\emptyset$, we speak of a {\em sourceless} current.
We associate to a current $\n$ the {\em weight} 
\begin{equation} \label{eq:curweight}
\weight_{G,\beta}(\n) :=\prod_{\{v,v'\}\in E}\frac{(\beta J_{\{v,v' \}} )^{\n_{\{v,v'\}}}}{\n_{\{v,v'\}}!}\,,
\end{equation}
which comes from the associated Ising model on $G$~\cite{GHS} (that also has coupling constants $J$ and inverse temperature $\beta$).
For now we focus on the critical parameters on the square lattice 
\[
\beta=\beta_c=\tfrac 12\ln(\sqrt 2 +1),
\]
and $J_{\{v,v'\}}=1$ for every $\{v,v'\}$ which is an edge of $G$, and $0$ otherwise, and drop them from the notation. General models will be considered in Section~\ref{sec:discprem}.

We introduce the probability measure on currents with sources $B\subseteq V$ given by 
\begin{align}\label{eq:rc}
\mathbf P^B_G(\set) := \frac {\weight_{G}(\n)}{{Z}^B_G}, \qquad \text{for every }\set \in \sets^B,
\end{align}
where ${Z}^B_G$ is the partition function. The random variable $\n$ is called a {\em random current configuration on $G$ with free boundary conditions and source-set~$B$}.

We define $\mathbf P^{A,B}_{G,H}$ to be the law of $(\n_1,\n_2)$, where $\n_1$ and $\n_2$ are two independent currents with respective laws $\mathbf P^A_G$ and $\mathbf P^B_H$. 
The \emph{double random current (DRC)} (model) is the law of $\n_1+\n_2$ under $\mathbf P^{A,B}_{G,H}$.
We call a {\em cluster} of any current $\n$ a connected component of the graph with vertex set $V$ and edge set $E(\n):=\{e\in E:\n_e>0\}$. To a given cluster~$\mathcal C$ we associate a loop configuration made of the dual edges $e^*$ where $e=\{v,v'\}$ is such that $v\in \mathcal C$ and $v'\notin \mathcal C$. Note that this loop configuration is made of loops on the dual graph corresponding to the different connected components of $\mathbb Z^2\setminus \mathcal C$. The loop corresponding to the unbounded component is called the {\em outer boundary} of the cluster, and the loops corresponding to the boundaries of the bounded ones (sometimes referred to as \emph{holes}) are called the {\em inner boundaries}.
We define the {\em (nested) boundaries interface configuration} $\eta(\n)$ to be the collection of outer and inner boundaries of the clusters in $\n$.
We note that the inner and outer boundaries of different clusters may share edges (but they do not cross). We will often refer to the elements of $\eta(\n)$ as the interfaces of $\n$.

As before, we fix a simply connected Jordan domain $D\subsetneq \mathbb C$ and consider the double random current on~$D^\delta$. 
To state the following theorem, we will need the notion of two-valued sets $\Ab_{-a,b}$ introduced in \cite{MR3936643}, which is the unique thin local set of the Gaussian free field in $D$ with boundary values $-a$ and $b$. In this work, we use $\lp_{-a,b}$ to denote the collection of outer boundaries (which are SLE$_4$-type simple loops and level loops of the Gaussian free field) of the connected components of $D\setminus \Ab_{-a,b}$.  We refer to Section~\ref{sec:prem} for more details on two-valued sets and related objects.
We define 
\[
\lambda= \sqrt{\pi/8}.
\]

\begin{theorem}[Convergence of  double random current clusters with free boundary conditions]\label{thm:cvg_free_drc_v1}
Let  $D$ be a Jordan domain, and let $D^\delta$ approximate $D$. Moreover, let $\eta^\delta$ be the nested boundaries interface configuration of the critical double random current on $D^\delta$ with free boundary conditions. 
Then as $\delta\to 0$, $\eta^\delta$ converges in distribution to a limit whose law is invariant under all conformal automorphisms of $D$.
More precisely, we have that  (see Fig.~\ref{fig:DRC} Left)
\begin{itemize}
\item The outer boundaries of the outermost clusters converge to a CLE$_4$ in $D$.
\item If the outer boundary of a cluster converges to $\gamma$, then the inner boundaries of this cluster converge to $\lp_{-2\lambda, (2\sqrt{2}-2)\lambda}$ in the domain encircled by $\gamma$.
\end{itemize}
This picture repeats iteratively in each hole of every cluster. In particular,
\begin{itemize}
\item If a loop in the inner boundary of a cluster converges to 
$\gamma$, then the outer boundaries of the 
outermost clusters enclosed by $\gamma$ converge to a CLE$_4$ in the domain encircled by $\gamma$.
\end{itemize}

\end{theorem}

We will also work with the \emph{random current model with wired boundary conditions} on $G$. For the purpose of the statement below, we define it explicitly for the critical model on the square lattice without referring to the dual model. Later in Sect.~\ref{subsec:2coupling}, a version for a general planar graph will be stated.
Let $G\subset \mathbb Z^2$ be a simply connected subgraph of $\mathbb Z^2$ that is a union of square faces (in particular does not have vertices of degree one). 
Let $\partial G$ be the set of vertices of $G$ that lie on the unbounded face of $G$ and are of degree two or three. 
We define $G^+$ to be the graph with vertex set $V^+:=V\cup\{\mathfrak g\}$ where $\mathfrak g$ is an additional vertex that lies in the unbounded face of $G$, and $E^+:=E\cup\{\{v,\mathfrak g\}:v\in\partial G\}$,
where vertices of degree two contribute two edges. This condition comes from the fact that $G$ is a weak dual graph of some subgraph of the dual square lattice, and in this case $G^+$ is the full dual graph.
The coupling constants on the new edges are the same as on all other edges, and are critical.
Accordingly, we introduce the measures $\mathbf P_{G^+}^{B}$ and $\mathbf P_{G^+,H^+}^{A,B}$ as before.

\begin{figure}[t]
\centering
\includegraphics[width=0.45\textwidth]{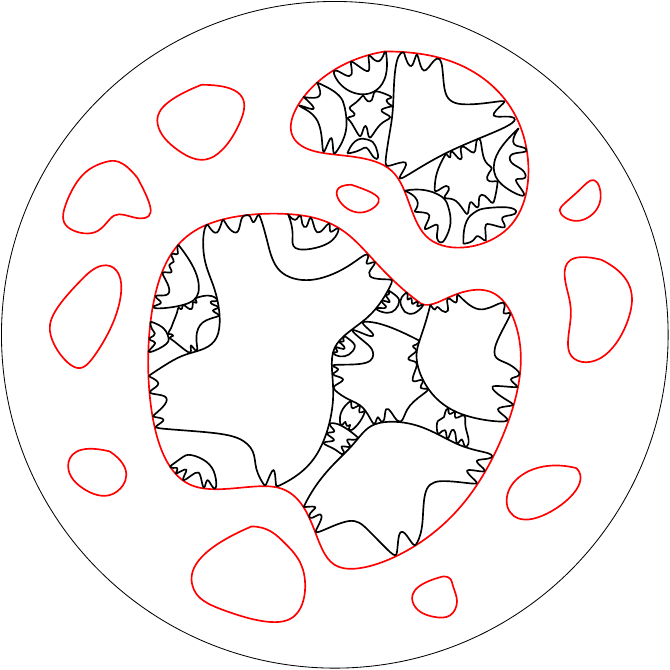}\qquad
\includegraphics[width=0.45\textwidth]{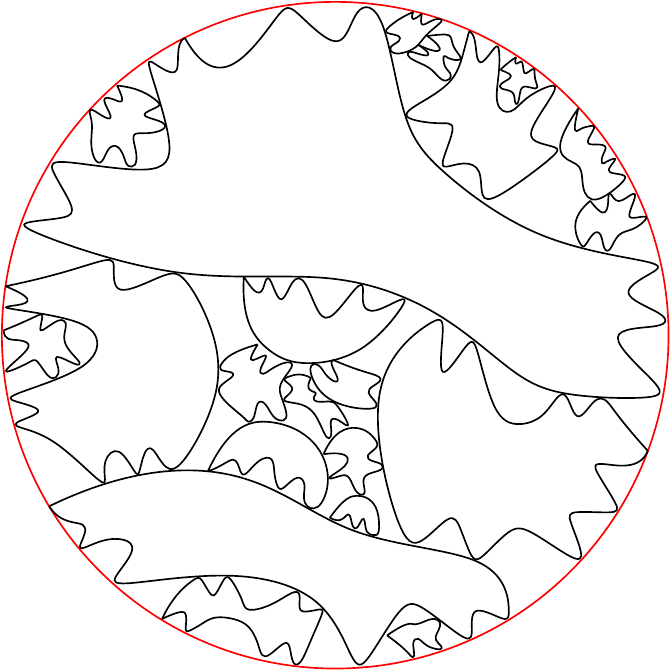}
\caption{\textbf{Left:} We depict the outermost clusters in a double random current with free boundary conditions. The outer boundaries of these clusters are in red (they form a CLE$_4$). The inner boundaries of the clusters are in black.  \textbf{Right:} We depict the unique outermost cluster in a double random current with wired boundary conditions. The inner boundaries of this cluster are in black. \textbf{For both:} In each domain encircled by an inner boundary loop, one has (the scaling limit of) an independent double random current with free boundary conditions. This allows us to iteratively sample the nested interfaces.} 
\label{fig:DRC}
\end{figure}

\begin{theorem}[Convergence of double random current clusters with wired boundary conditions]\label{thm:cvg_wired_drc_v1}
Let  $D$ be a Jordan domain, and let $D^\delta$ approximate $D$. Moreover, let $\eta^\delta$ be the nested boundaries interface configuration of the critical double random current on $D^\delta$ with wired boundary conditions. 
Then as $\delta\to 0$, $\eta^\delta$ converges in distribution to a limit whose law is invariant under all conformal automorphisms of $D$.
More precisely, we have that (see Fig.~\ref{fig:DRC} Right)
\begin{itemize}
\item The inner boundaries of the unique outermost cluster converge to $\lp_{-\sqrt{2}\lambda, \sqrt{2}\lambda}$ in $D$.

\item  If the inner boundary of a cluster converges to $\gamma$, then the outer boundaries of the outermost clusters enclosed by $\gamma$ converge to a CLE$_4$ in the domain encircled by $\gamma$.

\item If the outer boundary of a cluster converges to $\gamma$, then the inner boundaries of this cluster converge to $\lp_{-2\lambda, (2\sqrt{2}-2)\lambda}$ in the domain encircled by $\gamma$.
\end{itemize}
\end{theorem}

\begin{remark}
The values of $a$ and $b$ in $\lp_{-a,b}$ that we obtain in our results are combinations of $\sqrt{2}\lambda$ and $2\lambda$. The mechanism for the generation of each of these gaps in the scaling limit is very different, and this realisation is one of the 
main (and possibly surprising) insights of this work. The appearance of multiplies of $\sqrt{2}\lambda$ is directly related to the value of the multiplicative constant in front the Gaussian free field that arises as the scaling limit of the associated height function (see Section~\ref{subsec:2cvg_gff} and Theorem~\ref{thm:NFconv} therein). This is the same constant as the one in the scaling limit of height functions in the dimer model~\cite{Ken01}. Moreover, the inner boundaries of clusters posses a Markov property already at the discrete level as can be easily seen from the definition of the double random current. This means that the gap $\sqrt{2}\lambda$ is in some sense present already in the discrete.
On the other hand, $2\lambda$ is the height gap between the two sides of a level line in the Gaussian free field \cite{MR2486487}, which only emerges in the continuum. We have identified it using properties of two-valued sets~\cite{MR3936643, MR3827968} (see Section~\ref{sec:prem}) and properties of the scaling limit of the model (in particular how the interfaces intersect in the continuum, which we derive in our companion paper~\cite{DumLisQia21}).
Also, the outer boundaries do not have any apparent Markov property at the discrete level, and hence one can think of the value $2\lambda$ as an \emph{emergent} or \emph{effective} gap.
\end{remark}

Theorems~\ref{thm:cvg_free_drc_v1} and~\ref{thm:cvg_wired_drc_v1} have the following applications.
\begin{itemize}
\item The Hausdorff dimension of a double random current cluster in the scaling limit (for both free and wired boundary conditions) is $7/4$ \cite{SSV}. 

\item (Difference in log conformal radii) The difference of log conformal radii between two successive loops that encircle the origin in the scaling limit of double random current interfaces  is equal to $T_1+T_2$, where $T_1$ is the first time that a standard Brownian motion exits $[-\pi, (\sqrt{2}-1)\pi]$ and $T_2$ is the first time that a standard Brownian motion exits $[-\pi, \pi]$ (see \cite[Proposition 20]{MR3936643}).

\item (Number of clusters) Let $N(\ep)$ be the number of double random current clusters in the unit disk surrounding the origin such that their outer boundaries have a conformal radius w.r.t.\ the origin at least $\ep$. 
We will show in Lemma~\ref{lem:drc_radii} that almost surely,
\begin{align*}
N(\ep) / \log (\ep^{-1}) \underset{\ep\to 0}{\longrightarrow} \frac{1}{\sqrt{2} \pi^2}.
\end{align*}

\item  (Scaling limit of the magnetization in domains) With a little bit of additional work, one may derive from our results the conformal invariance of the $n$-point spin-spin correlations of the critical Ising model already obtained in \cite{CHI} as these correlations are expressed in terms of connectivity properties of $\n_1^\delta+\n_2^\delta$. The additional technicalities would consist in relating the point-to-point connectivity in $\n_1^\delta+\n_2^\delta$ to the probabilities that the $\ep$-neighborhoods of the points are connected. Such reasonings have been implemented repeatedly when proving conformal invariance, and we omit the details here as it would lengthen the paper even more. Even though the result is already known, we still wished to mention this corollary as our paper uses only the convergence of certain fermionic observables to obtain convergence of the nesting field height function to the GFF. Unlike the spinor observables used in \cite{CHI}, these are local functions of the Kadanoff--Ceva fermions.
The convergence of such fermionic observables has been obtained for the critical Ashkin--Teller model (which is a combination of two interacting Ising models) in \cite{GiuMasTon15} via renormalization arguments using the crucial fact that the observables are local. Notoriously, the spin-spin correlations are not of this kind, which makes renormalization arguments much more difficult to implement. We believe that the strategy of this paper may be of use to extend the universality results from \cite{GiuMasTon15} to non-local Grassmann observables.
\end{itemize}

Finally, we remark that Theorems~\ref{thm:cvg_free_drc_v1} and~\ref{thm:cvg_wired_drc_v1} are simplified versions of more detailed results (see Theorems~\ref{thm:ABh} and~\ref{thm:Qh}) that we will prove in Section \ref{sec:bah}. We do not include all details in the introduction in order to facilitate the reading, but let us make some comments on the additional properties that we can obtain:

\begin{itemize}
\item The proofs of Theorems~\ref{thm:cvg_free_drc_v1} and~\ref{thm:cvg_wired_drc_v1} rely on the coupling of the models with a height function that we will present in the next subsection. In fact, the primal and dual double random currents can be coupled together with the same height function (see Theorem~\ref{thm:mastercoupling}). Consequently, the limiting interfaces of the primal and dual models are also coupled with the same GFF, so that we fully understand the nesting and intersecting behavior of their limiting interfaces.

\item Theorems~\ref{thm:cvg_free_drc_v1} and~\ref{thm:cvg_wired_drc_v1} state the convergence of the boundaries of double random current clusters. 
To identify the cluster of a current, one only needs to know the edges where the current is strictly positive. 
However, apart from the shape of the clusters, we also have an additional information on whether the current is even and positive or odd on each edge. A hole of a double random current cluster is called odd if the flux of the cluster around this hole is an odd number, and otherwise it is called even. Here the flux is the total current flowing across any dual path that connects any face in the hole to the boundary of the graph. In the discrete, given the shape of the clusters, there is additional randomness to determine the parity of the holes. However, in the continuum limit, as we will show in Theorem~\ref{thm:ABh}, the parity of each hole in a double random current cluster with free b.c.\ is a \emph{deterministic function} of the shape of the cluster.
\end{itemize}

\subsection{Convergence of the nesting field of the double random current to the Gaussian free field}\label{subsec:2cvg_gff}
As mentioned above, a central piece in our strategy is a new convergence result dealing with the so-called nesting field of the double random current introduced by two of the authors in~\cite{DumLis}.
Let $G=(V,E)$ be a generic planar graph. For a current $\n$, let
\begin{itemize}
\item   $\n_{\rm odd}$ be the set of edges with an odd value in $\n$ (called the \emph{odd part} of $\n$)
\item  $\n_{\rm even}$ be the set of edges with an even and strictly positive value of $\n$ (called the \emph{even part} of $\n$).
\end{itemize}
We clearly have $\n_{\rm odd}\cup \n_{\rm even} = E(\n)$, and hope that no confusion will arise from the fact that the zero values are not included in the even part of a current.
In what follows we will often identify a current $\n$ with the pair $(\n_{\rm odd},\n_{\rm even})$ as it carries all the relevant information for our considerations.

A nontrivial connected component of the graph 
$(V,\cur_{\rm odd})$ will be called a \emph{contour}. In particular, each contour $\cont$ is contained in a unique cluster of $\cur$, and each
cluster $\clust$ is associated to a contour configuration $\clust \cap \cur_{\rm odd}$. Each contour configuration gives rise to a $\pm1$ spin configuration on the faces of $G$,
where the external unbounded face is assigned spin $+1$, and where the spin changes whenever one crosses an edge of a contour.
We call a cluster $\clust$ \emph{odd around a face} $u$ if the spin configuration associated with the contour configuration $\clust \cap \cur_{\rm odd}$ assigns spin $-1$ to $u$ (this is the same as asking for the total 
flux of the current in the cluster to be odd across any dual path connecting $u$ to infinity).

For a current $\n$, let $\mathfrak C(\n)$ be the collection of all clusters of $\n$, and let $(\coin_{\clust})_{\clust \in \mathfrak C(\n)}$ be i.i.d.\ random variables equal to $+1$ or $-1$ with probability $1/2$ indexed by $\mathfrak C(\n)$. These random variables are called the {\em labels} of the clusters. The \emph{nesting field with free boundary conditions} of a current $\n$ on $G$ evaluated at a face $u$ of $G$ is defined by
\begin{align} \label{def:nestingfield}
h_G(u) := \sum_{\clust \in \mathfrak C(\n)  }\mathbf 1\{\clust \text{ odd around } u \}  \coin_{\clust}.
\end{align}
Analogously, the \emph{nesting field with wired boundary conditions} of a current $\n$ on $G^+$ evaluated at a face $u$ of $G^+$ is defined by
\begin{equation}\label{def:nestingfield+}
h^+_{G^+}(u) :=(\mathbf 1\{\clust_{\mathfrak g} \text{ odd around } u \}-1/2)  \coin_{\clust_{\mathfrak g}}+ \sum_{\clust\neq \clust_{\mathfrak g}  }\mathbf 1\{\clust \text{ odd around } u \}  \coin_{\clust},
\end{equation}
where $\clust_{\mathfrak g}$ is the cluster containing the external vertex $\mathfrak g$, and where the sum is taken over all remaining clusters of $\n$. Here, whether $\clust_{\mathfrak g}$ is odd around a face of $G$ or not depends on the embedding of the graph $G^+$.
However, one can see that the distribution of $h^+_{G^+}(u)$ is independent of this embedding.

Note that due to the term corresponding to $\clust_{\mathfrak g}$, the nesting field with wired boundary conditions takes half-integer values, whereas the one with free boundary conditions is integer-valued. Such definition is justified by the next result, and by the joint coupling of $h_G$ and~$h^+_{G^*}$ via a dimer model described in Section~\ref{sec:DD}. We note that the global shift of $1/2$ between $h_G$ and $h^+_{G^*}$ is the same as in the work 
of Boutilier and de Tili\`{e}re~\cite{BoudeT}.

\bigbreak
The following is the main result of this part of the argument. We identify the function $h_{D^\delta}$ defined on the faces of $D^\delta$ with a distribution on $D$ in the following sense: extend $h_{D^\delta}$ to all points in $D$ by setting it to be equal to $h_{D^\delta}(u)$ at every point strictly inside the face $u$, and 0 on the complement of the faces in $D$. Then, we view $h_{D^\delta}$ as a \emph{distribution} (\emph{generalized function}) by setting
\[
h_{D^\delta}(f):=\int_Df(x)h_{D^\delta}(x)dx,
\]
where $f$ is a {\em test function}, i.e.~a smooth compactly supported function on $D$. We proceed analogously with the field $h^+_{(D^{\delta})^*}$ and extend it to all points within the faces of $(D^{\delta})^*$.
We will say that a sequence of random generalized functions $X_n$ \emph{converges weakly to} a random generalized function $X$, if $X_n(f)$ converges in distribution to $X(f)$ for every test function $f$.

The \emph{Gaussian free field (GFF)} $h_D$ with zero boundary conditions in $D$ is a random distribution such that for every smooth function $f$ with compact support in $D$, we have
\begin{align}\label{eq:GFF}
\Eb  \bigg[ \bigg(\int_D f(z) h_D(z) dz \bigg) ^2 \bigg]  = \int_D \int_D f(z_1) f(z_2) G_D(z_1, z_2) dz_1 dz_2,
\end{align}
where $G_D$ is the Green's function on $D$ with zero boundary conditions satisfying $\Delta G_D(x, \cdot) = -\delta_x(\cdot)$, where $\delta_x$ denotes the Dirac mass at $x$.
This normalization means e.g.\ that for the upper half plane $\mathbb H$, we have 
\[
G_{\mathbb H} (x,y)=\frac1{2 \pi} \log |(x-\bar y) / (x-y)|.
\]

Given a planar graph $G$, we write $G^\dagger$ for its \emph{weak dual}, i.e.\ the planar dual graph with the vertex corresponding to the outer boundary of $G$ removed.

\begin{theorem}[Convergence of the nesting field] \label{thm:NFconv}
Let $D$ be a Jordan domain, and let $D^\delta$ approximate $D$.
Denote by $h_{D^\delta}$ the nesting field of the critical double random current model on $D^{\delta}$ with free boundary conditions, and by $h^+_{(D^{\delta})^\dagger}$ the nesting field of the critical double random current model on the weak dual graph $(D^{\delta})^\dagger$ with wired boundary conditions.
Then
\[
\lim_{\delta\rightarrow 0}h_{D^\delta} =\lim_{\delta\rightarrow 0}h^+_{(D^{\delta})^\dagger}=\frac{1}{\sqrt{\pi}} h_D,
\]
where $h_D$ is the GFF in $D$ with zero boundary conditions, and where the convergence is in distribution in the space of generalized functions.
\end{theorem}

We want to mention that $h_{D^\delta}$ and $h^+_{(D^{\delta})^\dagger}$ can be coupled together as one random height function $H_{D^\delta}$ defined on the faces of a planar graph $C_{D^{\delta}}$ (whose faces correspond to both the faces of $D^{\delta}$
and $(D^{\delta})^\dagger$; see Fig.~\ref{fig:graphs}) in such a way that $\lim_{\delta\rightarrow 0}H_{D^\delta}=\frac{1}{\sqrt{\pi}} h_D$, and moreover the values of $h_{D^\delta}$ and $h^+_{(D^{\delta})^\dagger}$ differ locally by an additive constant. 
More properties of this coupling are described in Section~\ref{subsec:2coupling}.

Our proof is based on the relationship between the nesting field of double random currents on a graph~$G$ 
and the height function of a dimer model on decorated graphs $G^d$ and $C_G$ established in~\cite{DumLis}.
We will first explicitly identify the inverse Kasteleyn matrix associated with these dimer models with the correlators of real-valued Kadanoff--Ceva 
fermions in the Ising model~\cite{KadCev}. 
This is valid for arbitrary planar weighted graphs, and can also be derived from the bozonization identities of Dub\'{e}dat~\cite{Dub}. For completeness of exposition, we choose to present an 
alternative derivation that uses arguments similar to those of~\cite{DumLis}. Compared to~\cite{Dub}, rather than using the connection with the six-vertex model, we employ the double random current  model.
We then express the real-valued observables on general graph in terms of the complex-valued observables of Smirnov~\cite{smirnov}, 
Chelkak and Smirnov~\cite{CheSmi12} and Hongler and Smirnov~\cite{HonSmi}.
This is a well-known relation that can be e.g.\ found in~\cite{CCK}.
We also state the relevant scaling limit results for the critical observables on the square lattice obtained in~\cite{smirnov,CheSmi12,HonSmi}. 

All in all, we identify the scaling limit  of the inverse Kasteleyn matrix on graphs $C_{D^\delta}$ as $\delta\to 0$.
This is an important ingredient in the computation of the limit of the moments of the height function which is done by modifying an argument of Kenyon~\cite{Ken00}. 
Another crucial and new ingredient is a class of delicate estimates on the critical random current model from \cite{DumLisQia21} that allow us to do two things: 
\begin{itemize}
\item to identify the boundary conditions of the limiting GFF to be zero boundary conditions;
\item to control the behaviour of the increments of the height function
between vertices at small distances.
\end{itemize}
The first item is particularly important as handling boundary conditions directly in the dimer model is notoriously difficult. Here, the identification of the limiting boundary conditions is made possible by the connection with the double random current as well as the main result of \cite{DumLisQia21} stating that large clusters of the double random current with free boundary conditions do not come close to the boundary of the domain (see Theorem~\ref{thm:crossing free} below). We see this observation and its implication for the nesting field as one of the key innovation of our paper.

We stress the fact that Theorem~\ref{thm:NFconv} does not follow from the scaling limit results of Kenyon \cite{Ken00,Ken01}  as the boundary conditions considered in these papers are related to
Temperley's bijection between dimers and spanning trees~\cite{Temp,KPW,KenShe}, whereas 
those considered in this paper correspond to the double Ising model~\cite{DumLis,Dub,BoudeT}.
Moreover we note that the infinite volume version of~Theorem~\ref{thm:NFconv} was obtained by de Tili\`{e}re~\cite{DeT07}. Finally it can also be shown that the hedgehog domains of Russkikh~\cite{Rus}
are a special case of our framework, where the boundary of $D^\delta$ makes turns at each discrete step.

\paragraph{Organization} The paper is organized as follows. 
In Section~\ref{sec:input discrete} we state the main results from our second paper~\cite{DumLisQia21}.
In Section~\ref{sec:discprem} we recall the relationship between different discrete models and derive a connection between the inverse Kasteleyn matrix and complex-valued fermionic observables. While some (but not all) of these results are not completely new, they are scattered around the literature and we therefore review them here. In Section~\ref{sec:moments} we derive Theorem~\ref{thm:NFconv}. 
Section~\ref{sec:prem} presents more preliminaries on the continuum objects. Section~\ref{sec:bah} is devoted to the identification of the scaling limit of double random currents.

\paragraph{Acknowledgements}  We are grateful to Juhan Aru for pointing out a mistake in the previous version of the paper.  The beginning of the project involved a number of people, including Gourab Ray, Benoit Laslier, and Matan Harel. We thank them for inspiring discussions. The project would not have been possible without the numerous contributions of Aran Raoufi to whom we are very grateful.
We thank Pierre Nolin for useful comments on an earlier version of this paper. 
We thank Wendelin Werner for reading a later version of this paper, and for insightful comments which improved the paper.
Finally, we are indebted to the anonymous referees for the careful reading of the article and the valuable comments that were incorporated in the final version.
The first author was supported by the NCCR SwissMap from the FNS. This project has received funding from the European Research Council (ERC) under the European Union's Horizon 2020 research and innovation program (grant agreement No.~757296).

\section{Input from the second paper of the series}\label{sec:input discrete}

In this section we briefly recap some inputs from \cite{DumLisQia21} that are used in this paper.
 We refer to \cite{DumLisQia21} for the proofs. We only mention the main tools from \cite{DumLisQia21} that we will use and refer, later in the proof, to the precise statements of \cite{DumLisQia21} when they were not mentioned in this section.

\paragraph{Results for the double random current model}
We will need tightness results for several families of loops, notably for the outer and inner boundaries of the double random current clusters. This is done using an Aizenman--Burchard-type criterion for the double random current. Below, for a subset $A$ of vertices, an {\em $A$-cluster} is a cluster for the current configuration restricted to~$A$. A {\em domain} $D$ is a subgraph of $\mathbb Z^2$ whose boundary is a self-avoiding polygon in $\mathbb Z^2$.  Let $\Lambda_r:=[-r,r]^2$ and ${\rm Ann}(r,R):=\Lambda_R\setminus \Lambda_{r-1}$. Call an ${\rm Ann}(r,R)$-cluster (i.e., an $A$-cluster with $A={ \rm Ann}(r,R)$) {\em crossing} if it intersects both $\partial\Lambda_r$ and $\partial \Lambda_R$. 
For an integer $k\ge1$, let $A_{2k}(r,R)$ be the event\footnote{The subscript $2k$ instead of $k$ is meant to illustrate that there are $k$ ${\rm Ann}(r,R)$-clusters from inside to outside separated by $k$ dual clusters separating them.} that there are $k$ distinct  ${\rm Ann}(r,R)$-clusters crossing ${\rm Ann}(r,R)$. 

\begin{theorem}[Aizenman--Burchard criterion for the double random current model]
\label{prop:tight number crossings}
There exist sequences $(C_k)_{k\geq 1},(\lambda_k)_{k\geq 1}$ with $\lambda_k$ tending to infinity as $k\to \infty$, such that for every domain $D$, every $k\ge1$ and all $r,R$ with $1\le r\le R/2$,
\begin{align}
\label{eq:tight number XOR}{\bf P}_{D,D}^{\emptyset,\emptyset}[A_{2k}(r,R)]&\le C_k(\tfrac rR)^{\lambda_k}.\end{align}
If the domain has a $C^1$ boundary, the same holds for the model with wired boundary conditions but with the constants $C_k$ and $\lambda_k$ depending on $D$.
\end{theorem}

We will also need some a priori properties of possible subsequential scaling limits. These will be obtained using estimates in the discrete on certain four-arm type events. We list them now.
Let \begin{align*}
A_4^\square(r,R)&:=\{\text{there exist two $\Lambda_R$-clusters crossing $\mathrm{Ann}(r,R)$}\}\end{align*}
and let $A_4^\square(x,r,R)$ be the translate of $A_4^\square(r,R)$ by $x$.

\begin{theorem}\label{thm:absence closed pivotal expectation white}
There exists $C>0$ such that for all $r,R$ with $1\le r\le R$,
\begin{align}
\label{eq:closed pivotal0}{\bf P}_{\mathbb Z^2,\mathbb Z^2}^{\emptyset,\emptyset}[A_4^\square(r,R)]&\le C(r/R)^2.
\end{align}
Furthermore, for every $\ep>0$, there exists $\eta=\eta(\ep)>0$ such that for all $r,R$ with $1\le r\le \eta R$ and every domain $\Omega\supset\Lambda_{2R}$,
\begin{align}
\label{eq:closed pivotal existence0}{\bf P}_{\Omega,\Omega}^{\emptyset,\emptyset}[\exists x\in \Lambda_R:A_4^\square(x,r,R)]&\le \ep.
\end{align}
\end{theorem}

The result is coherent with the fact that the scaling limit of the outer boundaries of large clusters in the double random current model with free boundary conditions is given by CLE$_4$, which is known to be made of simple loops that do not touch each other. Interestingly, to derive the convergence to the continuum object it will be necessary to first prove this property at the discrete level.

We turn to a second result of the same type. For a current $\n$, let $\n^*$ be the set of dual edges $e^*$ with $\n_e=0$. For a dual path $\gamma=(e_1^*,e_2^*,\dots,e_k^*)$, call the $\n$-flux through $\gamma$ the sum of the $\n_{e_i}$. 
Call an {\em ${\rm Ann}(r,R)$-hole in $\n$} a connected component of $\n^*$ restricted to ${\rm Ann}(r,R)^*$ (note that it can be seen as a collection of faces). An ${\rm Ann}(r,R)$-hole is said to be {\em crossing} ${\rm Ann}(r,R)$ if it intersects $\partial\Lambda_r^*$ and $\partial\Lambda_R^*$.  Consider the event
\begin{align*}
A_4^\blacksquare(r,R)&:=\Big\{\begin{array}{c}\text{there exist two ${\rm Ann}(r,R)$-holes crossing ${\rm Ann}(r,R)$ and the}\\
\text{shortest dual path between them has even $(\n_1+\n_2)$-flux}\end{array}\Big\}.\end{align*}

Denote its translate by $x$ by $A_4^\blacksquare(x,r,R)$.

\begin{theorem}\label{thm:absence closed pivotal expectation black}
There exists $C>0$ such that for all $r,R$ with $1\le r\le R$,
\begin{align}
\label{eq:open pivotal}{\bf P}_{\mathbb Z^2,\mathbb Z^2}^{\emptyset,\emptyset}[A_4^\blacksquare(r,R)]&\le C(r/R)^2.
\end{align}
Furthermore, for every $\ep>0$, there exists $\eta=\eta(\ep)>0$ such that for all $r,R$ with $1\le r\le \eta R$ and every domain $D\supset\Lambda_{2R}$,
\begin{align}
\label{eq:open pivotal existence}{\bf P}_{\Omega,\Omega}^{\emptyset,\emptyset}[\exists x\in \Lambda_R:A_4^\blacksquare(x,r,R)]&\le \ep.
\end{align}
\end{theorem}

Let us mention that the previous results are obtained using the following key statement, which is of independent interest and is also directly used in this paper.
For a set $D$, let $\partial_rD$ be the set of vertices in~$D$ that are within a distance $r$ from $\partial D$. 

\begin{theorem}[Connection probabilities close to the boundary for double random current]\label{thm:crossing free}
There exists $c>0$ such that for all $r,R$ with $1\le r\le R$ and every domain $D$ containing $\Lambda_{2R}$ but not $\Lambda_{3R}$, 
$$\frac{c}{\log(R/r)}\le{\bf P}_{D,D}^{\emptyset,\emptyset}[\Lambda_R\stackrel{\n_1+\n_2}\longleftrightarrow\partial_r D]\le \epsilon(\tfrac rR),$$
where $x\mapsto \epsilon(x)$ is an explicit function tending to 0 as $x$ tends to 0. 
\end{theorem}

We predict that the upper bound should be true for $\epsilon(x):=C/\log(1/x)$ 
but we do not need such a precise estimate here. Again, the result is coherent with the fact that the scaling limit of the outer boundary of large clusters in the double random current with free boundary conditions is given by CLE$_4$.

The lower bound is to be compared with recent estimates \cite{DumSidTas17,DumManTas20} obtained for another dependent percolation model, namely the critical Fortuin--Kasteleyn random cluster model with cluster-weight $q\in[1,4)$. There, it was proved that the crossing probability is bounded from below by a constant $c=c(q)>0$ uniformly in $r/R$. We expect that the behaviour of the critical random cluster model with cluster weight $q=4$ on the other hand is comparable to the behaviour presented here: large clusters do not come close to the boundary of domains when the boundary conditions are free.

\section{Preliminaries on discrete models}\label{sec:discprem}
The main two goals of this section are the following. First of all we describe a coupling between double random currents (both primal and dual) and the associated nesting fields.
This is stated in Theorem~\ref{eq:KWdual}, and the properties of the coupling are crucial in the proofs of our main theorems (they exactly mimic the structure of level sets in the continuum GFF discussed in Section~\ref{sec:prem}). 
For the proof, we study three auxiliary and related to each other discrete models: random alternating flows (introduced in~\cite{LisT}), and two bipartite dimer models 
on two different modifications $G^d$ and $C_G$ of the underlying graph $G$ (introduced in~\cite{DumLis} and~\cite{FanWu,Dub} respectively). These are described in Section~\ref{sec:mapping between discrete models}. We stress the fact that the alternating flow model and the dimer model on $G^d$ are not used
outside this section, but they are a convenient tool to relate the double random current model with the dimer model on $C_G$. This is then used in Section~\ref{sec:moments} to show convergence of the nesting field to the GFF.
The main new result on the dimer model on $C_G$ contained in this section is the fact that the associated inverse Kasteleyn matrix is \emph{exactly} equal to the fermionic observable of Chelkak and Smirnov~\cite{CheSmi12}.

\subsection{A coupling between the primal and dual double random current}\label{subsec:2coupling}
Let $G=(V,E)$ be a graph as in Section~\ref{subsec:2cvg_gff}.
In this section we discuss the joint coupling of the double random current on $G$ and the double random current on the dual graph $G^*$
together with a height function that restricts to both the nesting field of the primal and the dual random current (see Fig.~\ref{fig:MC} for an illustration).
The coupling constants for the dual model satisfy the Kramers--Wannier duality relation
\begin{align} \label{eq:KWdual}
\exp(-2\beta^* J^*_{e^*}) = \tanh (\beta J_e).
\end{align}
We note that if $J_e=J^*_{e^*}=1$ for all $e$, and $\beta=\beta_c$, then $\beta^*=\beta_c$ (the critical point is self dual).
Properties of this coupling will be used in Section~\ref{sec:bah} to identify the scaling limit of the boundaries of the double random current clusters.
We will provide a proof of this result at the end of Section~\ref{sec:DD} using a relation with the dimer model. 

\begin{theorem}[Master coupling] \label{thm:mastercoupling}
One can couple the following objects:
\begin{enumerate}
\item[i)]\label{itm:obj1} a double random current $\n$ with free boundary conditions on the primal graph $G=(V,E)$, together with i.i.d.\ $\pm1$-valued spins $(\tau_{\mathcal C}:\mathcal C\in\mathfrak C(\n))$ associated to each cluster of $\n$,

\item[ii)]\label{itm:obj3} the dual double random current $\n^\dagger$ with free boundary conditions on the full dual graph $G^*=(U,E^*)$ (that we will refer to as the \emph{wired boundary conditions} on the weak dual graph $G^\dagger$) and with the dual coupling constants, together with i.i.d.\ $\pm 1$-valued spins $(\tau^\dagger_{\mathcal C}:\mathcal C\in\mathfrak C(\n^\dagger))$ associated with each cluster of $\n^\dagger$,

\item[iii)]\label{itm:obj5} a height function $H$ defined on $V\cup U$,
\end{enumerate}
in such a way that the following properties hold:
\begin{enumerate}
\item \label{itm:cp5} The configurations $\n$ and $\n^\dagger$ are disjoint in the sense that $\n_e>0$ implies $\n^{\dagger}_{e^*}=0$ and  $\n^\dagger_{e^*}>0$ implies $\n_{e}=0$,
where $e^*$ is the dual edge of $e$.
\item \label{itm:cp3} The odd part (the set of edges with odd values) of $\n$ is equal to the collection of interfaces of $\tau^\dagger$ (the set of primal edges separating dual clusters with $+1$ and $-1$ spins of $\tau^\dagger$), and the odd part of $\n^\dagger$ is equal to the collection of interfaces of $\tau$ (the set of dual edges separating primal clusters with $+1$ and $-1$ spins of $\tau$).
\item \label{itm:cp6}
For a face $u\in U$ and a vertex $v\in V$ incident on $u$, we have
\[
H(u)-H(v) =\tfrac 12 \tau^\dagger_u \tau_v.
\] 
Moreover, the height function $H$ restricted to the faces of $G$ (resp.\ $G^*$) has the law of the nesting field of $\n$ with free boundary conditions (resp.\ $\n^\dagger$ with wired boundary conditions) as denoted by $h$ (resp.~$h^\dagger$).
\item \label{itm:bc}
When exploring a cluster of $\n$ from the outside, inside each of its holes the dual current $\n^\dagger$ has wired boundary conditions, where each inner boundary of the hole (as defined in Sec.~\ref{intro:drc}) 
is identified as a single dual vertex (note that a single hole can have multiple inner boundaries since the inner boundaries by definition do not cross primal edges whose both endpoints are in the cluster of $\n$). To be more precise, let $\tilde G$ 
be a connected component (that is not the component of the boundary) of $G$ obtained after removing a cluster $\mathcal C$ of $\n$ and all its adjacent edges. Then $\n^\dagger$ restricted to $\tilde G^*$ is a double random current with wired boundary conditions 
(here we disregard the state of $\n^\dagger$ on edges dual to a primal edge that is adjacent to $\mathcal C$).
By duality, the same holds with the roles of $\n$ and $\n^\dagger$ exchaged.
\end{enumerate}
\end{theorem}

\begin{figure} 
		\begin{center}
			\includegraphics[scale=0.7]{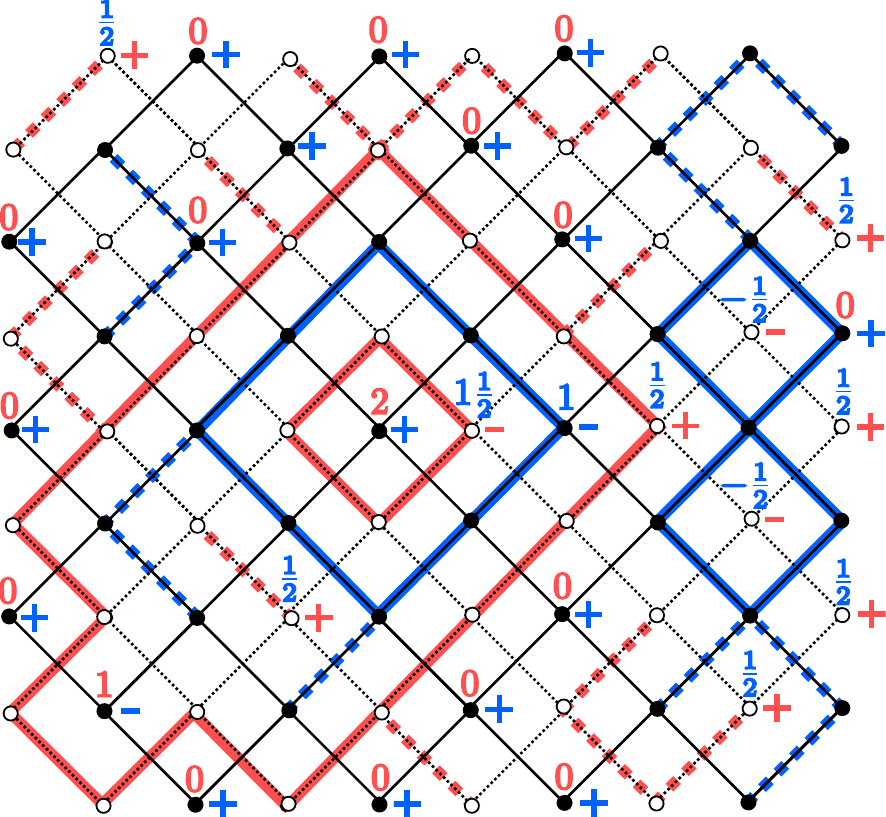}  
		\end{center}
		\caption{An illustration of the coupling from Theorem~\ref{thm:mastercoupling}. A piece of the (rotated) primal square lattice with white vertices, and its dual square lattice with black vertices is shown. The primal and dual double random current clusters are drawn in blue and red respectively. The odd parts of the current are marked with solid lines, whereas the nonzero even parts are marked with dashed lines. \\
Each vertex (primal black vertex) and a face (dual white vertex) carries both a $\pm 1$ spin ($\tau$ and $\tau^\dagger$ respectively) and the value of the height function $H$. The height function takes integer values in $\mathbb Z$ on the black vertices and in $\tfrac 12+\mathbb Z$ on the white vertices as implied by property~\ref{itm:cp6} of the master coupling. \\
Property \ref{itm:cp6} and the fact that the spins $\tau$ and $\tau^\dagger$ are constant on the primal and dual clusters respectively imply that the height function 
is also constant on both the primal and dual clusters. This is why in the figure we marked the values of the spins and height only at the rightmost vertices of the clusters (including isolated vertices). \\
 }	\label{fig:MC}
\end{figure}

We stress the fact that the interfaces of $\tau$ and $\tau^\dagger$ are disjoint in the sense of property~\ref{itm:cp5} appears already in the works of Dub\'{e}dat~\cite{Dub}, and Boutilier and de Tili\`{e}re~\cite{BoudeT}. However, property~\ref{itm:cp5} is a stronger statement as it concerns the full double random current, and not only its odd part.

We note that the laws of $\tau$ and $\tau^\dagger$ are those of a XOR Ising model and the dual XOR Ising model respectively (See Corollary~\ref{cor:XOR} below). However, we will not use this fact in the rest of the article, and our main results do not have direct implications for the scaling limit of the interfaces in the XOR Ising model. An extension of this coupling to the Ashkin--Teller model can be found in the works~\cite{LisHF, LisAT} that appeared before but were based on 
the current article. As mentioned, we will provide a different proof that uses the associated dimer model representation (see Section~\ref{sec:DD}).

The following statement identifies the labels introduced in the definition~\eqref{subsec:2cvg_gff} that correspond to the two nesting fields encoded by $H$.
\begin{corollary}
In the coupling as above, each cluster $\mathcal C$ of $\n$ (resp.\ a cluster of $\n^\dagger$ different from the cluster of the ghost vertex $\mathfrak g$) can be assigned a well-defined dual spin $\tau^{\dagger}_{\mathcal C}$ (resp.\ $\tau_{\mathcal C}$). This is the spin assigned to any face of $G$ (resp.\ $G^*$) incident on $\mathcal C$ from the outside. 
For the cluster of $\mathfrak g$ we set this spin to be $+1$.
With this definition, the independent labels associated to the clusters as in the definition~\eqref{subsec:2cvg_gff} are given by 
\begin{equation}\label{eq:expression coin}
\coin_{\mathcal C} = \tau_{\mathcal C}\tau^\dagger_{\mathcal C}.
\end{equation}
\end{corollary}

\begin{proof}
We first argue that $\tau^\dagger_{\mathcal C}$ is well defined.
By property~\ref{itm:cp3}, for each primal cluster $\mathcal C$, all the dual spins at the faces adjacent to the outer boundary of $\mathcal C$ (the innermost dual circuit surrounding $\mathcal C$) have the same value. Indeed, otherwise there would exist two consecutive dual vertices along the outer boundary of $\mathcal C$ with opposite $\tau^\dagger$ spins. However by property~\ref{itm:cp3}, the corresponding primal edge would then belong to $\mathcal{C}$, and hence the two dual vertices could not be consecutive on the outer boundary of $\mathcal C$. This justifies the definition \eqref{eq:expression coin}.

We also need to argue that given $\n$, the spins $(\coin_{\mathcal C})_{\mathcal C\in\mathfrak C(\n)}$ are independent (as in the definition of the nesting field). As mentioned, we will not use this result in the rest of the article.
This follows easily since given $\n$, $\tau^\dagger_{\mathcal C}$ is a deterministic function of $\n$ (by property~\ref{itm:cp3}), and~$\tau_{\mathcal C}$ are independent by definition.
\end{proof}

\begin{figure}
\centering
\includegraphics[width=0.8\textwidth]{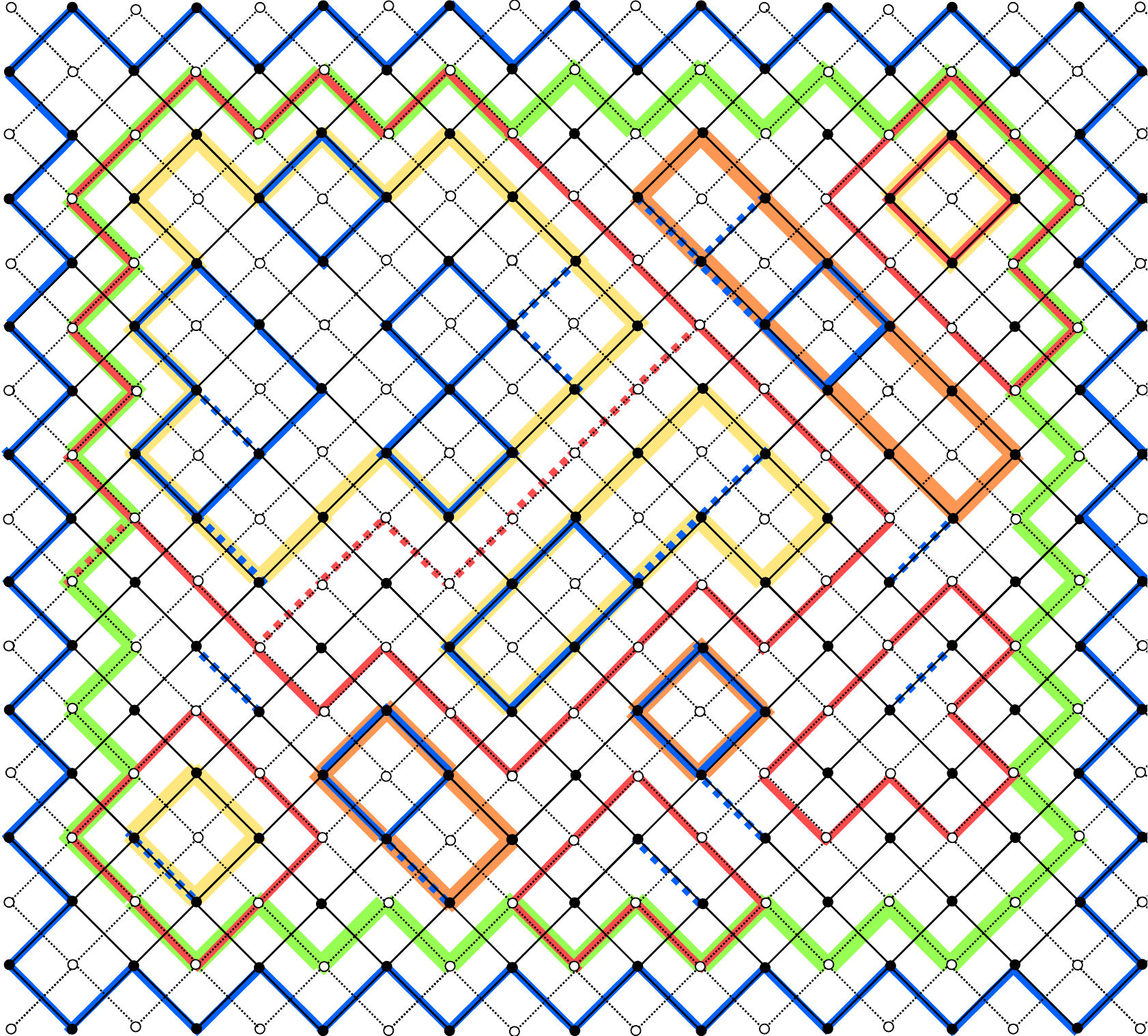}
\caption{A configuration of primal (red) and dual (blue) double random currents $\n$ and~$\n^\dagger$. The outermost blue circuit is part of a cluster of the boundary in $\n^\dagger$ whose remainder is not shown here.
The green edges denote the inner boundary loop $\ell$ of this cluster (i.e. a loop in $Q_0$ as defined in Sec.~\ref{sec:mainres}). The primal vertices on this loop are identified with each other in the exploration process described in property~\eqref{itm:bc} of the master coupling from Theorem~\ref{thm:mastercoupling}. After this identification the primal current $\n$ has wired boundary conditions. \\
The clusters of the modified current $\n_\ell$ defined in Sec.~\ref{sec:mainres} are given by the union of the green loop and the red clusters surrounded by it. Finally, $ Q_{1}(\ell)$ is defined as the collection of loops in the inner boundary of the external most cluster (touching $\ell$) of this modified current~$\n^\ell$. 
These loops come in two types, the yellow loops that are part of $ A_0(\ell)$, and the orange loops are in $ Q_1(\ell) \setminus A_0(\ell)$.
Each orange loop traces the red clusters from the outside and/or the green loop from the inside. This property is used in Lemma~\ref{lem:QH1} to obtain precompactness of the the orange loops given precompactness
of the red and green loops.
\\
Inside each yellow loop of the inner boundary of the primal clusters, the procedure is repeated and now the primal clusters surrounded by each such loop have wired boundary conditions. }  
\label{fig:nestedMarkov}
\end{figure}

Finally, for the sake of independent interest, we establish a connection with the XOR Ising model.
Recall that the XOR Ising model is just the pointwise product of two i.i.d.~Ising models.
\begin{corollary}\label{cor:XOR}
In the master coupling described above:
\begin{itemize}
\item The spins $ \tau_v$, $v\in V$, where we define $\tau_v=\tau_\mathcal C$, with $\mathcal C$ being the cluster containing~$v$, have the law of the XOR Ising model on $G$ with free boundary conditions, coupling constants $J_e$, and inverse temperature $\beta$.
\item The spins $ \tau_u$, $u\in U$, where we define $\tau^\dagger_u=\tau^\dagger_\mathcal C$, with $\mathcal C$ being the dual cluster containing~$u$, conditioned on the spin of the outer vertex $\mathfrak g$ being~$+1$, have the law of the XOR Ising model on the dual graph $G^*$ with $+$ boundary conditions and dual parameters as in~\eqref{eq:KWdual}.
\end{itemize}
\end{corollary}
\begin{proof}
We prove the first statement as the second one follows by duality. 
To this end, let $\mathbb P_{G,\beta}$ denote the master coupling probability measure and $\mathbb E_{G,\beta}$ its expectation. Moreover, let $ E_{G,\beta}^{\text{Ising}}$ be the expectation with respect to the Ising model on $G$ with free boundary conditions, coupling constants $J_e$, and inverse temperature $\beta$. 
For every $A\subseteq V$, since the spins $\tau$ are independent for all clusters, we have 
\[
\mathbb E_{G,\beta}\Big[\prod_{v\in A}\tau_v\Big] = \mathbb P_{G,\beta}[\mathcal F_A]= \mathbf P^\emptyset_{G,\beta}\otimes \mathbf P^\emptyset_{G,\beta} [\n_1+\n_2\in \mathcal F_A],
\] 
where $\n\in \mathcal F_A$ is the event that each cluster of $\n$ contains an even number (possibly zero) of vertices of $A$. 
Now, the classical switching lemma of Griffiths, Hurst and Sherman for double random currents~\cite{GHS} (see also \cite{DumLisQia21}) gives that 
\[
 \mathbf P^\emptyset_{G,\beta}\otimes \mathbf P^\emptyset_{G,\beta} [\n_1+\n_2\in \mathcal F_A]= E_{G,\beta}^{\text{Ising}}\Big[\prod_{v\in A}\sigma_v\Big]^2,
 \]
 where $\sigma$ denotes the Ising spins.
 The last expression is by definition the correlation function of XOR Ising spins at $A$.
Since the spins are $\pm 1$-valued, this implies that the law of $\tau$ under $\mathbb P_{G,\beta}$ is the law of the XOR-Ising model (e.g.~one can look at the characteristic function of the random vector $\tau$ and expand it into a finite sum of correlation functions as above).
\end{proof}

\subsection{Mappings between discrete models}\label{sec:mapping between discrete models}
In this section we recall the combinatorial equivalences between double random currents, alternating flows and bipartite dimers established in~\cite{LisT,DumLis}. 
We will later use them to derive a version of Dub\'{e}dat's bosonization identity \cite{Dub}. An additional black-white symmetry for correlators of monomer insertions is established that is not apparent in~\cite{Dub}. 
This will yield a representation of the inverse Kasteleyn matrix as the fermionic observable of Chelkak and Smirnov~\cite{CheSmi12}.

The results here are stated for general Ising models on arbitrary planar graphs $G=(V,E)$ and with arbitrary coupling constants $(J_e)_{e\in E}$. 
We focus on the free boundary conditions case and the wired boundary conditions can be treated analogously, replacing $G$ with $G^+$. We will actually mostly consider wired boundary conditions on the dual graph $G^*$ which
one can think of as $(G^{\dagger})^+$, where $G^\dagger$ is the weak dual of $G$ whose vertex set does not contain the unbounded face of $G$. 

We start by describing the relevant decorated graphs: the double random current model on a graph $G$ will be related to the 
alternating flow model on a directed graph $\vec G$, and the dimer model on two different bipartite graphs $G^d$ and $C_G$. All these graphs are planar and weighted, and their local structure together with the corresponding edge weights are shown in Fig.~\ref{fig:graphs}.
We now describe their construction in detail. Even though this is ultimately not relevant, we note that the structure of $G^d$ and $\vec G$ is determined by $G$ \emph{together with} a choice of an orientation for each edge.

\begin{figure} 
		\begin{center}
			\includegraphics[scale=0.9]{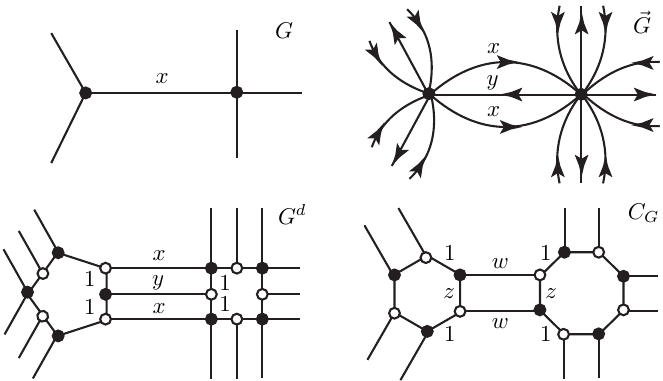}  
		\end{center}
		\caption{One can construct the graphs $\vec G$, $G^d$ and $C_G$ ($\vec G$ is formally a multigraph) locally around each vertex of $G$. The weights satisfy $y=\tfrac{2x}{1-x^2}$, $w=\tfrac{2x}{1+x^2}$, $z=\tfrac{1-x^2}{1+x^2}$.
	Here $x=x_e$ is the high-temperature weight equal to $\tanh( \beta J_e$). The edges carrying weight $1$ in $G^d$ (resp.\ in $C_G$) are called short (resp.\ roads), and the remaining edges are called long (resp.\ streets).}		\label{fig:graphs}
\end{figure}

Given $G$, $\vec G$ is obtained by replacing each edge $e$ of $G$ by three parallel directed edges $ e_{s1}$, $e_m$,  $e_{s2}$ such that the orientation of the side (or outer) edges $ e_{s1}$ and $ e_{s2}$ is opposite to the
orientation of the middle edge $ e_m$. The orientation of the middle edge can be chosen arbitrarily.  

To obtain $G^d$ from $\vec G$, we replace each vertex $v$ of $\vec G$ by a cycle of vertices of even length which is given by the number of times the orientation of edges in 
$\vec G$ incident on $v$ changes when going around $v$. We colour the new vertices black if the corresponding edges are incoming into $v$ and white otherwise. 
We then connect the white vertices in a cycle corresponding to $v$ with the appropriate black vertices in a cycle corresponding to $v'$, where $v$ and $v'$ are adjacent in~$\vec G$. 
We call \emph{long} all the edges of $G^d$ that correspond to an edge of~$\vec G$, and \emph{short} the remaining edges connecting the vertices in the cycles.

The last graph $C_G$ can be constructed directly from $G$ by replacing each edge of $G$ by a quadrangle of edges, and then connecting two quadrangles by an edge if the corresponding 
edges of $G$ share a vertex and are incident to the same face (see Fig.~\ref{fig:graphs}). Following~\cite{Dub}, we call \emph{streets}  the edges in the quadrangles and  \emph{roads}
those connecting the quadrangles (which represent cities).

We note that the set of faces $U$ (resp.\ vertices $V$) of $G$ naturally embeds into the set of faces of $\vec G$, $G^d$ and $C_G$ (resp.\ $G^d$ and $C_G$).
We therefore think of $U$ and $V$ as subsets of the set of faces of the respective decorated graphs (e.g.,\ when we talk about equality in distribution of the height function on $C_G$ and the nesting field on $G$).

In the remainder of this section we describe the mappings between the different models in the following order:
In Section~\ref{sec:DRCAF}, alternating flows on $\vec G$ are mapped under a map $\theta$ to a pair composed of the odd and even part of a double random current on $G$.
In Section~\ref{sec:AFD}, dimers on $G^d$ are mapped under a map $\pi$ to alternating flows on $\vec G$.
In Section~\ref{sec:DD}, dimers on $G^d$ are mapped to dimers on $C_G$. The corresponding statements for wired boundary conditions can be recovered by replacing $G$ with $G^+$.

The first two maps yield relations between configurations of the associated models, and the last map is described as a sequence of local transformations (urban renewals) of the graphs $C_G$ or $G^d$
that does not change the distribution of the height function on a certain subset of the faces of these two graphs.

We first describe relations on the level of distributions on configurations where no sources or disorders are imposed. Later on (in Section~\ref{sec:source insertion}) we increase the complexity by introducing sources.

\subsubsection{Double random currents on $G$ and alternating flows on $\vec G$} \label{sec:DRCAF}

A \emph{sourceless alternating flow} $F$ is a set of edges of the directed graph $\vec G$
satisfying the \emph{alternating condition}, i.e., for each vertex $v$, the edges in $F$ that are incident to $v$ alternate between being oriented towards and away from $v$ when going around $v$ (see Fig.~\ref{fig:flowcurrs}). In particular, the same number 
of edges enters and leaves~$v$. We denote the set of sourceless alternating flows on $\vec G$ by $\flows^\emptyset$, and following~\cite{LisT}, we
define a probability measure on $\flows^\emptyset$ by the formula, for every $F\in\mathcal F^\emptyset$,
\begin{align}
\label{eq:altflow}
\IP^{\emptyset}_{\flow} (F) := \frac1{Z^{\emptyset}_{\flow} }\weight_{\flow}(F), 
\end{align}
where $Z^{\emptyset}_{\flow}$ is the partition function of sourceless flows and, if $V(F)$ denotes the set of vertices in the graph $(V,F)$ that have at least one incident edge, 
\begin{align} \label{eq:floww}
\weight_{\flow}(F):=2^{|V|-|V(F)|} \prod_{ e\in F} x_{ e},
\end{align}
with the weights $x_{\vec e}$ as in Fig.~\ref{fig:graphs}.
We also define the {\em height function} of a flow $F$ to be a function $h=h_F$ defined on the faces of $\vec G$ in the following way:
\begin{itemize}
\item[(i)] $h(u_0)=0$ for the unbounded face $u_0$,
\item[(ii)] for every other face $u$, choose a path $\gamma$ connecting $u_0$ and $u$, and define $h(u)$ to be total flux of $F$ through $\gamma$, i.e., the number of edges in $F$ crossing $\gamma$ from left to right
minus the number of edges crossing $\gamma$ from right to left.
\end{itemize}
The function $h$ is well defined, i.e., independent of the choice of $\gamma$, since at each $v\in V$, the same number of edges of $F$ enters and leaves $v$ (and so the total flux of $F$ through any
closed path of faces is zero).

We are ready to state the correspondence between double random currents and alternating flows.
To this end consider a map $\theta : \flows^{\emptyset} \to \Omega^{\emptyset}$ defined as follows.
For every $F \in \flows^{\emptyset}$ and every $e \in E$, count the number of the corresponding directed edges $ e_m$, $ e_{s1}$, $ e_{s2}$ that are present in $F$.
Let $F_{\rm odd}\subseteq E$ be the set with one or three such present edges, and $F_{\rm even} \subseteq E$ the set with exactly two such edges, and set 
\[
\theta(F):=(F_{\rm odd} ,F_{\rm even}).\]
Denote by $\theta_* \IP_{\flow}^{\emptyset}$ the pushforward measure on $\Omega^{\emptyset}$. The following result was first proved in~\cite{LisT}. 

\begin{lemma}[Corollary 4.3 of \cite{LisT}] \label{thm:firstmapping}
Let $\n$ be distributed according to $\IP_{\dcurr}^{\emptyset}$, and let $h_\n$ be its nesting field. Let $F$ be distributed according to $\IP_{\flow}^{\emptyset}$.
Then
\[
(F_{\rm odd},F_{\rm even},h_F)= (\n_{\rm odd},\n_{\rm even},h_\n) \qquad \textnormal{in law.}
\]
\end{lemma}
\begin{proof}
This is a consequence of the fact that the total weight of all alternating flows corresponding to a cluster in the double random current, and whose outer 
boundary is oriented clockwise is the same as those oriented counterclockwise (see also the proof of Lemma~\ref{lem:symmetry}). 
This corresponds to the fact that
the nesting field is defined using symmetric coin flip random variables $\coin_\clust$. Moreover, the sum of these two weights is the same as the weight of the cluster in the double random current model.
More details are provided in the proof of Theorem 2.1 in~\cite{DumLis}.
\end{proof}

\subsubsection{Alternating flows on $\vec G$ and dimers on $G^d$} \label{sec:AFD}
\begin{figure}
		\begin{center}
			\includegraphics[scale=1]{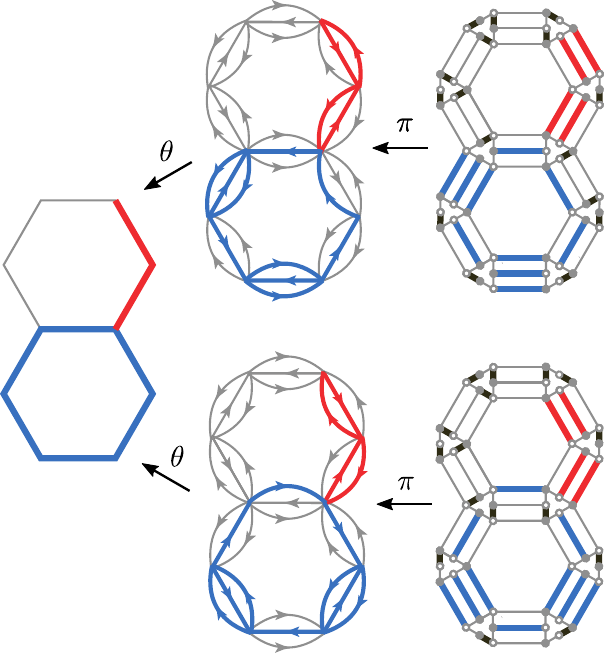}  
		\end{center}
		\caption{Left: A configuration $(\n_{\rm odd},\n_{\rm even})$ on a piece of the hexagonal lattice~$G$. The blue edges represent $\n_{\rm odd}$ and the red edges represent $\n_{\rm even}$. The blue and red edges together form one cluster~$\clust$.
		Middle: Two alternating flow configurations on $\vec G$ mapped to $(\n_{\rm odd},\n_{\rm even})$ under $\theta$. The two clusters have opposite 
		orientations of the outer boundary. Depending on this orientation the height function either increases or decreases by one when going from the outside to the inside of the 
		lower hexagon. This corresponds to two different outcomes for the label $\coin_\clust$ in the definition of the nesting field~\eqref{def:nestingfield}.
		Right: Two dimer configurations on $G^d$ that map to the corresponding alternating flows under $\pi$. Note that the parity of the height function on $G^d$ restricted to the vertices of $\clust$ 
		and shifted by $1/2$ changes whenever the sign of $\coin_\clust$ changes. This can be seen from the placement of the dimers on the short edges. This property is used in the proof of Theorem~\ref{thm:mastercoupling}. On the other hand the parity of the height function on the faces of $G$ is independent of $\coin_\clust$. \\ 
		We also note that both $\pi$ and $\theta$ are many-to-one maps.}
	\label{fig:flowcurrs}
\end{figure}

We first shortly recall the dimer model in its full generality. To this end, consider a finite weighted graph $\mathcal G$. Recall that a \emph{dimer cover} (or \emph{perfect matching}) $M$ of $\mathcal G$ is a subset of edges such that every vertex of the graph is incident to exactly one edge of $M$.
We write $\mathcal M(\mathcal G)$ for the set of all dimer covers of $\mathcal G$.
The \emph{dimer model} is a probability measure on $\mathcal M(\mathcal G)$ which assign a probability to a dimer cover that is proportional to the product of the edge-weights over the dimer cover. 

To each dimer cover $M$ on a bipartite planar finite graph $\mathcal G$ (colored in black and white in a bipartite fashion), one can associate a \emph{1-form} $f_M$ (i.e.\ a function defined on directed edges which is antisymmetric under a change of orientation) satisfying 
$f_M((v,v'))=-f_M((v',v))=1$ if $\{v,v'\} \in M$ and $v$ is white, and $f_M((v,v'))=0$ otherwise. 
For a 1-form $f$ and a vertex $v$, let $df(v)=\sum_{v'\sim v} f((v,v'))$ be the \emph{divergence} of~$f$ at $v$.
Note that for a dimer cover $M$, $df_M(v)=1$ if $v$ is white, and $df_M(v)=-1$ if $v$ is black.
Fixing a reference 1-form $f_0$ with the same divergence, we define the \emph{height function} $h=h_M$ by
\begin{itemize}
\item[(i)]  $h(u_0)=0$ for the unbounded face $u_0$, 
\item[(ii)] for every other face $u$, choose a dual path $\gamma$ connecting $u_0$ and $u$, and define $h(u)$ to be the total flux of $f_M-f_{0}$ 
through $\gamma$, i.e., the sum of values of $f_M-f_{0}$ over the edges crossing $\gamma$ from left to right.
\end{itemize}
The height function is well defined, i.e.\ independent of the choice of $\gamma$, since $f_M-f_{0}$ is a divergence-free flow, i.e.\ $d(f_M-f_0)=0$.

We now go back to the specific caseof $\mathcal G= G^d$. We will write $\IP^{\emptyset}_{G^d}$ for the dimer model measure on $G^d$ with weights as in Fig.~\ref{fig:graphs}.
We also fix a {reference 1-form} $f_0$ on $G^d$
given by
\begin{itemize}
\item $f_0((w,b))=-f_0((b,w))=1/2$ if $\{w,b\}$ is a short edge and $w$ is white,
\item  $f_0((w,b))=f_0((b,w))=0$ if $\{w,b\}$ is a long edge.
\end{itemize}

We now describe a straightforward map $\pi$ from the dimer covers on $G^d$ to alternating flows on $\vec G$ that preserves the law of the height function.
We note that one could carry out the same discussion and make a connection with double random currents directly, without introducing alternating flows.
However, we find the language of alternating flows particularly convenient to express some of the crucial steps discussed later on (especially Lemmata~\ref{lem:symmetry} and~\ref{lem:dualsurface}).
To this end, to each matching $\match \in \matchs(G^d)$,  associate a flow $\pi(\match)\in \flows^{\emptyset}$ by replacing each long edge in $\match$ by the corresponding directed edge in $\vec G$.
One can check that this always produces an alternating flow. Indeed, assuming otherwise, there would be two consecutive edges in $F(\match)$ of the same orientation,
and therefore the path of short edges connecting them in a cycle would be of odd length and therefore could not have a dimer cover, which is a contradiction. 
Let $\pi_*\IP_{G^d}^{\emptyset}$ be the pushforward measure on $\flows^{\emptyset}$ under the map $\pi$.

\begin{lemma}[\cite{DumLis}]  \label{thm:secondmapping}
We have $\pi_*\IP_{G^d}^{\emptyset}=\IP_{\flow}^{\emptyset}$. Moreover, under this identification, the restriction to $U$ of the height function of the dimer model is exactly the height function of the resulting alternating flow.
\end{lemma}
\begin{proof}
This is a consequence of the fact that the reference 1-form vanishes on the long edges, and hence its contribution to the increment of the height function across a long edge of $G^d$
is equal to zero, and the fact that the weights of the edges of $\vec G$ and the long edges of $G^d$ are the same. Moreover, if a vertex $v$ has zero flow through it, i.e, $v\in V\setminus V(F)$, then 
there are exactly $2$ dimer covers of the cycle of short edges of $G^d$ corresponding to $v$. 
Since both of these covers have total edge-weight $1$, this accounts for the factor $2^{|V|-|V(F)|}$ in~\eqref{eq:floww}.
\end{proof}

\subsubsection{Dimers on $G^d$ and on $C_G$} \label{sec:DD}

We will write $\IP^{\emptyset}_{C_G}$ for the dimer model measure on $C_G$ with weights as in Fig.~\ref{fig:graphs}.
The dimer models on $G^d$ and $(G^*)^d$ are closely related to the dimer model on $C_G$ (as was described in~\cite{DumLis}) using standard dimer model transformations called the vertex splitting and urban renewal, see Fig.~\ref{fig:urbanrenewal}. 
 The main two results of this section are Proposition~\ref{cor:CGH} below where we relate the height functions on $G^d$ and $C_G$, and the proof Theorem~\ref{thm:mastercoupling} (existence of the master coupling) that relies on Proposition~\ref{cor:CGH}.
\begin{proposition} \label{cor:CGH}
The height function on $C_G$ restricted to the faces and vertices of $G$
is distributed as the the height functions on $G^d$ and $(G^*)^d$ restricted to the faces and vertices of~$G$. In particular, the height function on $C_G$ restricted to the faces of $G$ has the law of the nesting field of the double random current with free boundary conditions on $G$,
and restricted to the vertices of $G$ has the law of the nesting field of the double random current with wired boundary conditions on $G^\dagger$ (or free boundary conditions on $G^*$).
\end{proposition}

\begin{figure}
		\begin{center}
			\includegraphics[scale=1.1]{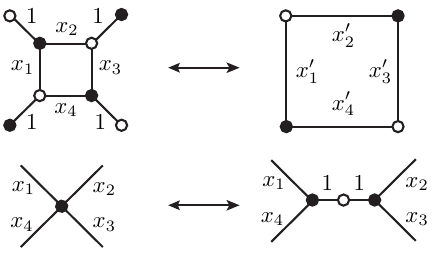}  
		\end{center}
			\caption{Urban renewal and vertex splitting  are transformations of weighted graphs preserving the distribution of dimers and the height function outside the modified region. 
		The weights in urban renewal satisfy $x_1'=\tfrac{x_3}{x_1x_3+x_2x_4}$, 
		$x_2'=\tfrac{x_4}{x_1x_3+x_2x_4}$, $x_3'=\tfrac{x_1}{x_1x_3+x_2x_4}$, $x_4'=\tfrac{x_2}{x_1x_3+x_2x_4}$.}
	\label{fig:urbanrenewal}
\end{figure}

To prove the proposition we start with a crucial lemma that first appeared in~\cite{DumLis}.
\begin{lemma}[\cite{DumLis}] \label{lem:graphtransformation}
One can transform $G^d$ and $(G^*)^d$ to $C_G$ (and the other way around) using urban renewals and vertex splittings.
\end{lemma}
\begin{proof}
We will describe how to transform $G^d$ to $C_G$ (see Fig.~\ref{fig:gdtocg} for an illustration). The second part follows since $C_G$ is symmetric with respect to $G$ and $G^*$. 

To this end, fix a bipartite black-white colouring of both $G^d$ and $C_G$. Note that for each edge $e$ in $G$, there is one quadrilateral $\mathcal Q$ in $C_G$ and two quadrilaterals $\mathcal Q_1$, $\mathcal Q_2$ in $G^d$ 
corresponding to $e$. For each such edge $e$,
choose for the internal quadrilateral of urban renewal the quadrilateral $\mathcal Q_i$ in $G^d$ with the opposite colors of vertices when compared to $\mathcal Q$. Then, split each vertex that the chosen quadrilateral shares
with a quadrilateral corresponding to a different edge of $G$. In this way we find ourselves in the situation from the upper left panel in Fig.~\ref{fig:urbanrenewal}.
After performing urban renewal and collapsing the doubled edge, we are left with one quadrilateral as desired. One can check that the weights that we obtain
match those from Fig.~\ref{fig:graphs}. We then repeat the procedure for every edge of $G$. The resulting graph is~$C_G$. 
\end{proof}

\begin{figure}
		\begin{center}
			\includegraphics[scale=0.7]{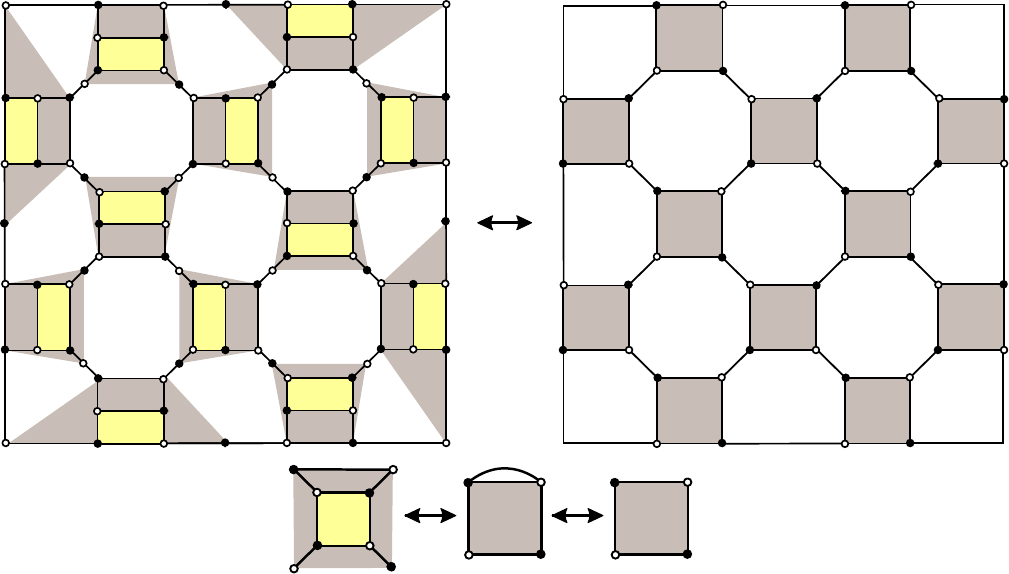}  
		\end{center}
			\caption{An example of the correspondence between dimer models on $G^d$ and $C_G$. The yellow quadrilaterals within grey quadrilaterals are transformed using urban renewal moves, and then collapsing one doubled edge to a single edge
			as shown at the bottom of the figure.
			The underlying graph $G$ is a $3\times 3$ piece of the square lattice.}
	\label{fig:gdtocg}
\end{figure}

 \begin{figure}
		\begin{center}
			\includegraphics[scale=0.7]{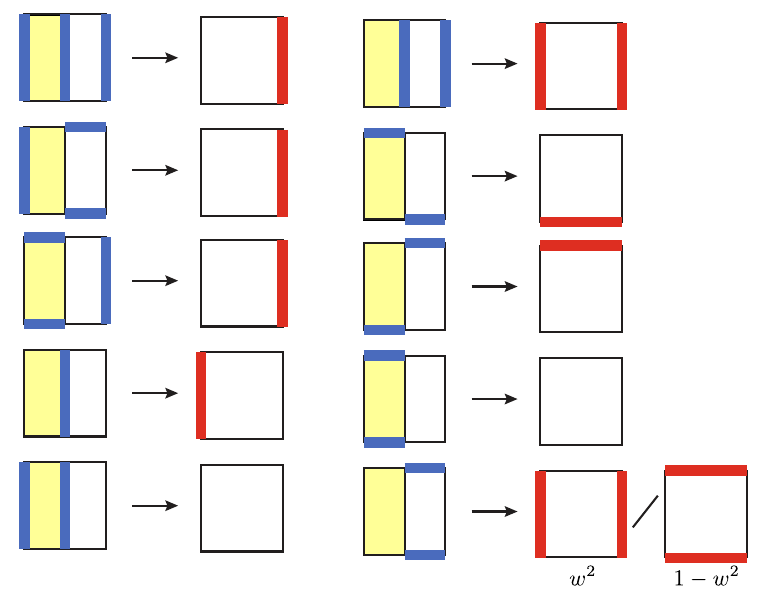}
		\end{center}
	\caption{The figure shows the measure preserving mapping of local configurations on $G^d$ (corresponding to a single edge $e$ of $G$) 
	to local configurations on the streets of $C_G$ under urban renewal performed on the left-hand-side quadrilateral in $G^d$. The last case involves additional random choice between two possible configurations. These choices are independent for local configurations corresponding to different 
	edges of $G$ and the probabilities are as in the figure with $w=2x/(1+x^2)$.}
	\label{fig:duality}
\end{figure}

A choice of quadrilaterals where urban renewals are applied for a rectangular piece of the square lattice is depicted in Fig.~\ref{fig:gdtocg}. In this way, the double random current model on the square lattice is related to a (weighted) dimer model on the square-octagon lattice. In Fig.~\ref{fig:duality}, we illustrate the behaviour of local dimer configurations under one urban renewal performed in the construction described in the lemma above.

As the reference 1-form for the dimer model on $C_G$ we choose the canonical one given by 
\begin{align} \label{eq:refflow}
f_0((w,b))=-f_0((b,w))=\IP^{\emptyset}_{C_G}(\{w,b\} \in M),
\end{align}
where $w$ is a white vertex. Note that this makes the height function centered as all its increments become centered by definition.
This is the same 1-form as used in~\cite{BoudeT} on the infinite square-octagon lattice $C_{\mathbb Z^2}$. In~\cite{KenLaplacian}, two crucial properties of $f_0$ were established 
when $G$ is an infinite isoradial graph and the Ising model on $G$ is critical. In the next lemma we show that both of these properties hold for arbitrary Ising weights on general finite planar graphs.
\begin{lemma} \label{lem:refflow}
We have
\begin{itemize}
\item $\IP^{\emptyset}_{C_G}(e \in M)=1/2$, if $e$ is a road, i.e., $e$ corresponds to a corner of $G$,
\item  $\IP^{\emptyset}_{C_G}(e \in M)=\IP^{\emptyset}_{C_G}(e' \in M)$, if $e$ and $e'$ are two parallel streets corresponding to the same edge of $G$ (or of the dual $G^*$).
\end{itemize}
\end{lemma}
In the proof, which is postponed to the end of Section~\ref{sec:source insertion}, we actually compute the probability from the second item in terms of the underlying Ising measure.
However, the exact value will not be important for our considerations. We note that the first bullet of the lemma above is the reason why the nesting field with free boundary conditions on $G$ is defined to be integer-valued
and the one with wired boundary conditions on $G^*$ to be half-integer valued.

A crucial observation now is that the height function on the faces of $G^d$ corresponding to the faces and vertices of $G$ is not modified by vertex splitting and urban renewal. 
This follows from basic properties of these transformations, and the fact that the reference 1-form on the short edges of $G^d$ is the same as the one on the roads of $C_G$ (by the first item of the lemma above).
Indeed, one can compute the height function on the faces of $G^d$ and $C^G$ corresponding to the faces and vertices of $G$ using only increments across short edges and roads respectively.
This means that the resulting height function on these faces of $C_G$ has the same distribution as the one on $G^d$. 
Since $C_G$ plays the same role with respect to $G^*$ as to $G$, we immediately conclude Proposition~\ref{cor:CGH}.

This observation is at the heart of the proof of the master coupling from Theorem~\ref{thm:mastercoupling}.
However, one has to be careful since there is loss of information between the dimer model on $G^d$ and the one on $C_G$. Indeed, we have already seen that knowing a dimer configuration on~$G^d$ allows one to fully recover the 
triple $(\n_{\rm odd}, \n_{\rm even}, h_{\n})$. 
However, a dimer configuration $M$ on $C_G$ only gives access to $(\n_{\rm odd}, h_{\n})$ (since $M$ determines the height function, and $\n_{\rm odd}$ are the edges where the height function has a nontrivial increment) 
and does not contain information about $\n_{\rm even}$. To recover it, one needs to add additional randomness in the form of independent coin flips for each edge of $G$ with an appropriate success probability.

 \begin{figure}
		\begin{center}
			\includegraphics[scale=0.7]{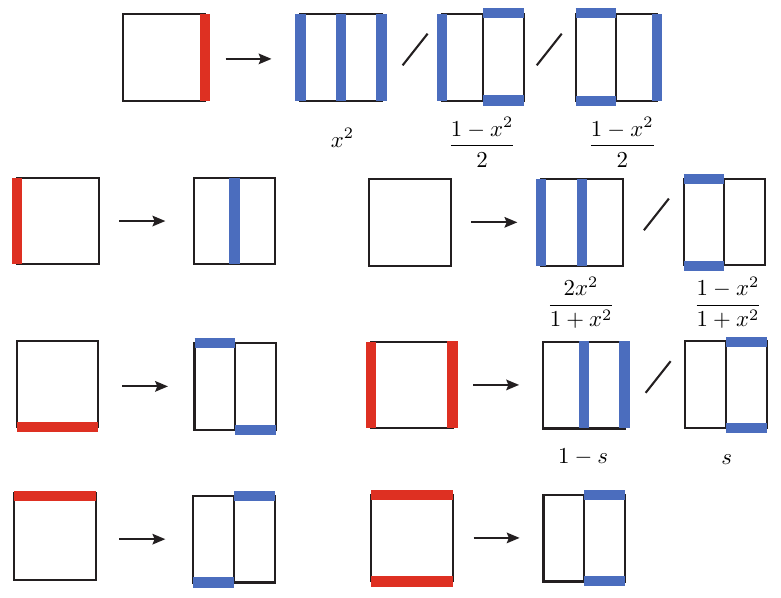}
		\end{center}
	\caption{The reverse mapping to that in Fig.~\ref{fig:duality}. Again, urban renewal is performed on the left-hand side quadrangle of the local configuration on $G^d$. Whenever there is ambiguity, we use additional randomness which is independent for each local configuration and with probabilities as in the figure
	with $s=\frac{2(1-x^2)}{3+x^4}$. These probabilities are simply obtained from Fig.~\ref{fig:duality} using the definitions of the weights in both dimer models on $C_G$ and $G^d$ and elementary conditional probability computations.}
	\label{fig:duality2}
\end{figure}

\begin{proof}[Proof of Theorem~\ref{thm:mastercoupling}]
We will use a procedure reverse to that from the proof of Lemma~\ref{lem:graphtransformation}. 
This procedure induces a measure preserving mapping between local configurations on $C_G$ and $G^d$, see Fig.~\ref{fig:duality2}, where in certain cases additional randomness 
is used to decide on the exact configuration on $G^d$.

As mentioned, the graph $C_G$ plays a symmetric role with respect to $G$ and $G^*$. Hence, taking the Kramers--Wannier dual parameters $x_e^*={(1-x_e)}/{(1+x_e)}$ and rotating the local configuration on $C_G$ by $\pi/2$, one can use the same mapping
from Fig.~\ref{fig:duality2} to generate local dimer configurations on $({G^*})^d$ that will correspond to dual random current configurations.
Recall that part of our aim is to couple the double random current on $G$ with its dual on $G^*$ so that no edge and its dual are open at the same time. 
The idea is to first sample a dimer configuration on $C_G$, and then using the rules from Fig.~\ref{fig:duality2} to choose, possibly introducing additional randomness, the dimer configurations on both $G^d$ and $(G^*)^d$. The desired property of the coupling will follow from the way we use the additional randomness
for $G^d$ and $(G^*)^d$.

We now explain this in more detail. In the coupling between double random currents and dimers on $G^d$, an edge in the current has value zero if and only if there is no long edge present in the
corresponding local dimer configuration. From Fig.~\ref{fig:duality2}, we see that the only possibility to have nonzero values of double currents for both a primal edge $e$ and its dual $e^*$ is when the quadrangle in $C_G$
that corresponds to both $e$ and $e^*$ has no dimer in the dimer cover. In that case we have a probability of ${2x_e^2}/({1+x_e^2})$ to get a non-zero (and even) value of the primal double current 
and a probability of ${2{(x_e^*)}^2}/{(1+{(x_e^*)}^2)}$ to get a non-zero (and even) value of the dual double current. However, since these choices are independent of the possible choices for other local configurations, and since
\[
\frac{2x_e^2}{1+x_e^2} + \frac{2{(x_e^*)}^2}{1+{(x_e^*)}^2} = 1- \frac{2x_e(1-x_e)}{1+x_e^2} <1
\]
we can couple the results so that the primal and dual currents are never both open (nonzero) at $e$. Together establishes Property~\ref{itm:cp5} from the statement of the theorem. 

We now focus on Property~\ref{itm:cp3}. Note that the spins $\tau^{\dagger}$ defined by the interfaces of odd current in $\n$ satisfy
\begin{align} \label{eq:tau1}
\tau^{\dagger}_u = (-1)^{H(u)}
\end{align}
for $u\in U$, where $H$ is the height function on $C_G$. By Proposition~\ref{cor:CGH} we already know that $H$ restricted to $U$ has the law of the height function on $(G^*)^d$ restricted to~$U$. 
From the relationship between the double random current $\n^\dagger$ on $G^*$ and the alternating flow model on $\vec G^*$, one can see that the parity of this height function at a face $u$ changes with the change of the orientation of the outer boundary of the cluster of $\n^\dagger$ containing $u$ (see Fig.~\ref{fig:flowcurrs} for a dual example). Therefore $(-1)^{H(u)}$ is distributed as an independent assignment of a sign to each cluster of~$\n^\dagger$.
This yields Property~\ref{itm:cp3}. A dual argument for 
\begin{align} \label{eq:tau2}
\tau_v =i (-1)^{H(v)}
\end{align}
with $v\in V$, and $i$ the imaginary unit, yields the dual correspondence. Here, the factor~$i$ appears due to the fact that the height function takes half-integer values on $V$.

Furthermore, \eqref{eq:tau1} and \eqref{eq:tau2} together imply Property~\ref{itm:cp6}. 

Finally, for Property~\ref{itm:bc} we make the following observations. First of all, when an edge is empty (has zero current) in $\n$, then
in the dimer model on $G^d$, the corresponding three long edges are not part of the dimer configuration. We can therefore remove them, and proceed similarly for all other empty edges encountered during the exploration of
a cluster of~$\n$.
This means that the unexplored part dimer configuration on $\tilde G^d$ is independent of the explored part, 
and moreover is in a measure preserving correspondence with double random currents with 
free boundary conditions on $\tilde G$. Furthermore, the (random) maps from Fig.~\ref{fig:duality} and Fig.~\ref{fig:duality2}, when composed together, map from dimers on $\tilde G^d$ to dimers on $(\tilde G^*)^d$ (and hence to double random currents with wired boundary conditions on the weak dual $(\tilde G^\dagger)^d$) are local. Therefore the distribution of $\n^\dagger$ inside $(\tilde G^*)^d$ is not affected by the explored part of the primal current $\n$ outside $\tilde G$, and is that of an independent double random current with wired boundary conditions on $(\tilde G^\dagger)^d$. 
We also note that a proof without using the dimer representation can also be given using the construction from~\cite{LisHF}.
\end{proof}

We leave it to the interested reader to check that the resulting coupling of the primal and dual double random current model is the same as the 
one described in~\cite{LisHF} (where no connection with the dimer model is used, and where all the properties above can as well be deduced).

\subsection{Disorder and source insertions}\label{sec:source insertion}

It will be important for our analysis to introduce the so-called sources in dimers, alternating flows, and double random currents, and to see how they relate to order-disorder variables in the Ising model.

A \emph{corner} $c=(u,v)$ of a planar graph $G$ is a pair composed of a face $u=u(c)$ (also seen as a vertex of the dual graph) and a vertex $v=v(c)$ bordering $u$. 
One can visualize corners as segments from the center of the face $u$ to the vertex $v$ (see Fig.~\ref{fig:insertions}). In this section we discuss correlations of disorder insertions, by which we mean  modifications of the state space of the appropriate model that are localized at the corners
of $G$, and describe their mutual relationships. In what follows, consider two corners $c_i$ and $c_j$, and a simple dual path $\gamma$ connecting $u(c_i)$ to $u(c_j)$. 
For a collection of edges $E_0$ of $G$, $\vec G$, $G^d$ or $C_G$, we define $\sgn_\gamma(E_0)=-1$ if the number of edges in $E_0$ crossed by $\gamma$ is odd and $\sgn_\gamma(E_0)=1$
otherwise. 

In the following subsections we introduce correlation functions of corner insertions in the relevant models and relate them to each other.

 \begin{figure}
		\begin{center}
			\includegraphics[scale=0.9]{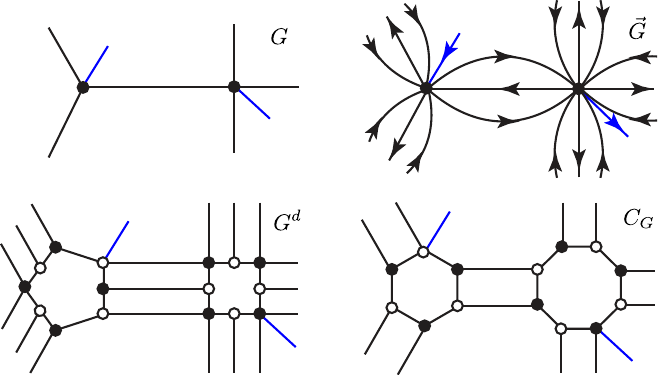}
		\end{center}
	\caption{Corner insertions in the relevant models can be realized by considering additional edges connecting a vertex and a neighbouring face.}
	\label{fig:insertions}
\end{figure}

\subsubsection{Kadanoff--Ceva fermions via double random currents} 
The two-point correlation function of \emph{Kadanoff--Ceva fermions} is defined by
\begin{align} \label{def:KCF}
\langle \chi_{c_i}\chi_{c_j} \rangle^\gamma :=\frac1{Z^{\emptyset}_{\text{hT}}} \sum_{\eta \in \mathcal E^{\{v(c_i),v(c_j)\}}} \textnormal{sgn}_\gamma(\eta)  \prod_{e\in \eta} x_e,
\end{align}
where $Z^{\emptyset}_{\text{hT}} :=\sum_{\eta \in \mathcal E^{\emptyset}} \prod_{e\in \eta} x_e$.
Here, $\mathcal E^\emptyset$ is the collection of sets of edges $\eta\subseteq E$ such that each vertex in the graph $(V,\eta)$ has even degree, and $\mathcal E^{\{v(c_i),v(c_j)\}}$ is the collection of sets of edges such that each vertex has  even 
degree except for $v(c_i)$ and $v(c_j)$ that have odd degree.
We note that the sign of this correlator depends on the choice of $\gamma$. However, its amplitude depends only on the corners $c_i$ and $c_j$.

The next lemma was proved in~\cite[Lemma 6.3]{ADTW}. It expresses Kadanoff--Ceva correlators in terms of double currents for which $u(c_i)$ is 
connected to $u(c_j)$ in the dual configuration. Below, for $\n\in \Omega^B$, let
\[
\weight_{\dcurr} (\n):=\mathop{\sum_{\n_1\in\Omega^B, \n_2\in\Omega^\emptyset}}_{\n_1+\n_2=\n}\weight(\n_1)\weight(\n_2),
\]
where $\weight=\weight_G$ is the random current weight defined in~\eqref{eq:curweight}.
For a current $\n$, let $\n^*$ be the set of dual edges $e^*$ with $\n_e=0$. For two faces $u$ and~$u'$, let $u \con{\cur^*} u'$ mean that $u$ is connected to~$u'$ in $\cur^*$, i.e., that
$u$ and~$u'$ belong to the same connected component of the graph $(U,\cur^*)$.

\begin{lemma}[Fermions via double currents~\cite{ADTW}] \label{lem:KCswitching}
We have 
\begin{align*}
\langle \chi_{c_i}\chi_{c_j} \rangle^\gamma=\frac{1}{ Z^{\emptyset}_{\dcurr}}  \sum_{\n \in \Omega^{\{v(c_i),v(c_j)\}} }\sgn_\gamma(\cur_{\odd})\weight_{\dcurr}(\n)
\mathbf{1}\{ u(c_i) \con{\cur^*} u(c_j)\}.
\end{align*}
\end{lemma}

\subsubsection{Sink and source insertions in alternating flows}
Consider the graph $\vec G$ with two additional directed edges $ c_i=(u(c_i),v(c_i))$ and  $-c_j=(v(c_j),u(c_j))$, and let $\mathcal{F}^{c_i, - c_j}$ be the set of alternating flows on this graph that contain both $ c_i$ and $-c_j$. By an alternating flow here we mean a subset of edges of the extended graph that satisfies the alternating condition at every vertex of $\vec G$.
The weights of $c_i$ and $-c_j$ are set to~$1$.
With $\gamma$ defined as above, introduce 
\[
Z_{\textnormal{flow}}^{\gamma}(c_i,-c_j) := \sum_{F \in \mathcal{F}^{c_i,  -c_j}}\textnormal{sgn}_\gamma(F) \weight_{\flow}(F).
\]
Here, $c_i$ plays the role of the \emph{source} and $-c_j$ is the \emph{sink} of the flow $F$.

Recall that $\theta:\mathcal{F}^{\emptyset}\rightarrow\Omega^{\emptyset}$ is the measure preserving map sending sourceless alternating flows on $\vec G$ to sourceless double current configurations on $G$, where as before, we identify a current $\n$ with the pair $(\n_{\rm odd}, \n_{\rm even})$.
With a slight abuse of notation, we also write $\theta$ for the analogous map from $\mathcal F^{c_i,-c_j}$ to the set $\Omega^{\{v(c_i),v(c_j)\}}$ of currents on $G$ 
with sources at $v(c_i)$ and $v(c_j)$ (for currents there is no distinction between sources and sinks). 

The next lemma is closely related to~\cite[Theorem 4.1]{LisT}.

\begin{lemma}[Symmetry between sinks and sources] \label{lem:symmetry} We have
\[
Z_{\textnormal{flow}}^{ \gamma}(c_i,- c_j)=Z_{\textnormal{flow}}^{\gamma}(c_j, -c_i).
\]
\end{lemma}
\begin{proof}
Note that the flow's weights on $\vec G$ are invariant under the reversal of direction of the flow, i.e., the weights of the three directed edges $ e_{s1},  e_m,  e_{s2}$ of $\vec G$ 
corresponding to a single edge $e$ of $G$ satisfy $x_{ e_{s1}}+x_{ e_{s2}}+x_{ e_{s1}}x_{ e_{s2}}x_{ e_{m}}=x_{ e_{m}}$ by construction.
Hence, for a fixed $(\n_{\rm odd},\n_{\rm even})\in \Omega^{\{v(c_i),v(c_j)\}}$, we have
\[
\sum_{F \in \mathcal{F}^{c_i, - c_j}: \ \theta(F)=(\n_{\rm odd},\n_{\rm even})} \weight_{\flow}(F) = \sum_{F \in \mathcal{F}^{c_j,  -c_i}:\ \theta(F)=(\n_{\rm odd},\n_{\rm even})} \weight_{\flow}(F).
\]
We finish the proof by summing both sides of this identity over $(\n_{\rm odd},\n_{\rm even})\in \Omega^{\{v(c_i),v(c_j)\}}$, and using the fact that $\textnormal{sgn}_\gamma(F)$ depends only on $\theta(F)$.
\end{proof}

The next result is a direct analog of~Lemma~\ref{lem:KCswitching} with an additional factor of $1/2$ that corresponds to the fact that the connected component of the flow that connects $c_i$ to $-c_j$ has a fixed orientation.
\begin{lemma}[Dual connection in alternating flows] \label{lem:dualsurface} 
We have
\[
\theta(\mathcal{F}^{ c_i, -c_j}) =\{\n \in \Omega^{\{v(c_i),v(c_j)\}}:  u(c_i) \con{\cur^*} u(c_j)  \},
\] 
and moreover
\[
Z_{\textnormal{flow}}^{\gamma}( c_i,  -c_j) = \tfrac 12\sum_{\n \in \Omega^{\{v(u_i),v(c_j)\}}} \sgn_\gamma(\cur_{\odd})\weight_{\textnormal{dcurr}}(\n) \mathbf{1}\{ u(c_i) \con{\cur^*} u(c_j)\}.
\]
\end{lemma}
\begin{proof}
We first argue that for each $(\n_{\rm odd},\n_{\rm even})=\theta(F)$ with $F\in\mathcal{F}^{ c_i, -c_j}$, we have that $ u(c_i) \con{\cur^*} u(c_j)$.
This follows from topological arguments and the alternating condition for flows. 
Indeed, assume by contradiction that there is a cycle of edges in $F$ separating $u(c_i)$ from $u(c_j)$, and choose the innermost such cycle surrounding $u(c_i)$. 
Consider the vertex $v$ of this cycle that is first visited on a path from $c_i$ to $-c_j$. The alternating condition implies that the edges of the cycle on both sides of $v$ 
should be oriented away from $v$. Following that orientation around the cycle, we must arrive at another vertex $v'$ of the cycle where both incident edges are oriented towards $v'$.
That is in contradiction with the alternating condition and the fact that the cycle is minimal.
The fact that the image of the map is $\{ u(c_i) \con{\cur^*} u(c_j)  \}$ follows from the same arguments as in~\cite[Lemma~5.4]{LisT}.

The second part of the statement follows from the proof of~\cite[Theorem 4.1]{LisT} or \cite[Theorem~1.7]{DumLis} (the weights of flows in~\cite{LisT} are the same as ours up to a global factor).
The multiplicative constant $1/2$ is a consequence of the fact that the orientation of the cluster containing the corners is fixed to one of the two possibilities, and in 
the double random current measure there is an additional factor of $2$ for each cluster (see \cite[Theorem~3.2]{LisT}).
\end{proof}

\begin{corollary}\label{cor:KCflow}  We have
\[
\langle \chi_{c_i}\chi_{c_j} \rangle^\gamma=2\frac{Z_{\textnormal{flow}}^{\gamma}( c_i, - c_j) }{Z_{\textnormal{flow}}^{\gamma} }=2\frac{Z_{\textnormal{flow}}^{\gamma}( c_j,  -c_i) }{Z_{\textnormal{flow}}^{\gamma} }.
\]
\end{corollary}
\begin{proof}
This follows directly from Lemmata~\ref{lem:KCswitching} and \ref{lem:dualsurface}.
\end{proof}

\subsubsection{Monomer insertions on $G^d$ and $C_G$} 

We identify the faces and vertices of the graphs $G$ and $\vec G$ with the corresponding subsets of the faces of the dimer graphs $G^d$ and $C_G$. 
We say that a vertex of $G^d$ or $C_G$ is a \emph{corner (vertex)} corresponding to $c= vu$ if it is incident both on the vertex $v$ and the face $u$ of $G$ in this identification.  
Analogously to the discussion above, for $\mathcal G\in \{ G^d, C_G\}$ and $v,v'$ two vertices of $\mathcal G$, we define $Z_{\mathcal G}^\gamma$ to be the partition function of dimer covers of the graph $\mathcal G$ with $v$ and $v'$ removed, where moreover each dimer crossed by the path $\gamma$ contributes 
an additional factor of $-1$ to the weight of the cover.

\begin{lemma}[Symmetry between white and black corners] \label{lem:ssm}
Let $b_i$ and $w_i$ (resp.~$b_j$ and $w_j$) be a black and white corner vertex of $G^d$ corresponding to the corner $c_i$ (resp.~$c_j$).
If there is no such vertex of the chosen colour, we modify $G^d$ by splitting the corner vertex of the opposite colour (using the vertex splitting operation from Figure~\ref{fig:urbanrenewal}).
Then
\[
Z_{G^d}^{\gamma}(b_i, w_j)=Z_{G^d}^{\gamma}(w_i, b_j)=Z_{\textnormal{flow}}^{\gamma}( c_i,  c_j).
\]
\end{lemma}
\begin{proof}
By the definition of the measure preserving map $F_*$ between dimers and alternating flows,
a corner monomer insertion in dimers is a source or sink insertion in alternating flows, which yields \[Z_{\textnormal{flow}}^{\gamma}( c_i,  c_j) = Z_{G^d}^{\gamma}(b_i, w_j).\]
The statement then follows immediately from Lemma~\ref{lem:symmetry}.
\end{proof}

\begin{figure}
		\begin{center}
			\includegraphics[scale=0.6]{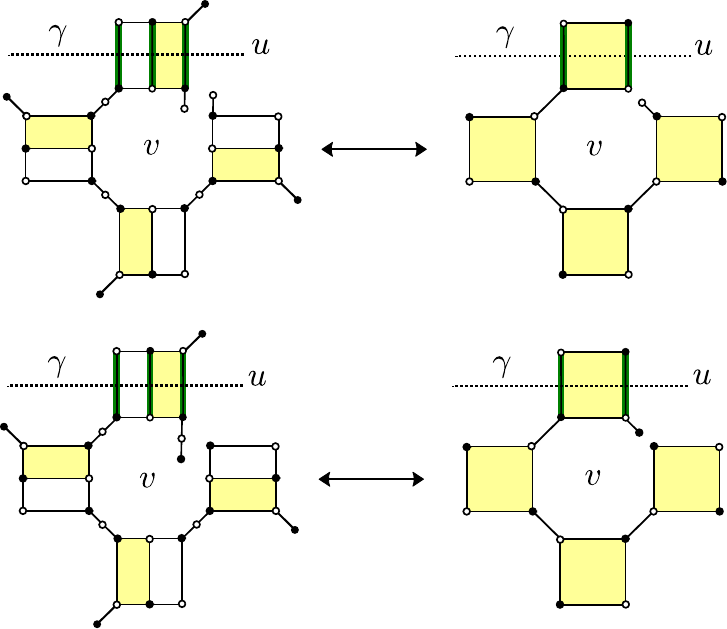}  
		\end{center}
		\caption{Behaviour of corner monomer insertions under urban renewal. Insertion of a monomer is modelled by the addition of edges with weight one into the dimer model: 
		above (resp.\ below), the insertion of a black (resp.\ white) monomer at the corner $c=uv$ with a disorder operator at $u$.
		The green edges crossing $\gamma$ are assigned negative weights. Urban renewal is applied to the yellow quadrilaterals on the left-hand side yielding the yellow
		quadrilaterals on the right-hand side. Note that the colour of the monomer insertions on the left-hand and right-hand sides agree.} 		\label{fig:urbanmonomer}
\end{figure}

\begin{lemma}[Monomer insertions in $G^d$ and $C_G$] \label{lem:gdcg}
Let $b$ and $w$ be respectively black and white corner vertices of $G^d$, and let $\tilde b$ and $\tilde w$ be the corresponding black and white vertices of~$C_G$. Then
\[
Z_{G^d}^{\gamma}(b, w)=Z_{C_G}^{\gamma}(\tilde b, \tilde w).
\]
\end{lemma}
\begin{proof}
We use urban renewal as in Fig.~\ref{fig:urbanmonomer} to transform $G^d$ with monomer insertions to $C_G$ with monomer insertions.
Note that here we use urban renewal with some of the long edges having negative weight. However, this is not a problem since  the opposite edges in a quadrilateral being transformed
by urban renewal always have the same sign, which results in a non-zero multiplicative constant for the partition functions. The resulting weights of $C_G$ are negative if and only if the edge crosses~$\gamma$. This implies the claim readily.
\end{proof}

We finally combine the previous results to obtain the following identity. We note that it can also be derived using the approach of~\cite{Dub} after taking into
account the symmetry of the underlying six-vertex model (that we do not discuss here and that is also not discussed in~\cite{Dub}).
\begin{corollary} \label{cor:main} In the setting of Lemma~\ref{lem:ssm}, we have
\[
\langle \chi_{c_i} \chi_{c_j} \rangle^\gamma = 2 \frac{Z^{\gamma}_{C_G}(w_i,b_j)}{Z_{C_G}}=2 \frac{Z^{\gamma}_{C_G}(w_j,b_i)}{Z_{C_G}}.
\]
\end{corollary}
\begin{proof}
This follows from Lemmata~\ref{lem:gdcg} and~\ref{lem:ssm}, as well as Corollary~\ref{cor:KCflow}.
\end{proof}

The final item of this section is the proof of Lemma~\ref{eq:refflow} which explicitly computes the canonical reference 1-form \eqref{lem:refflow} on $C_G$ in terms of the underlying Ising measures.
\begin{proof}[Proof of Lemma~\ref{lem:refflow}]
By the corollary above, for a street $\{w,b\}$ of $C_G$ corresponding to an edge $e=\{v,v'\}$ of $G$, we have 
\begin{align} \label{eq:dimfer}
\IP^{\emptyset}_{C_G}(\{w,b\}\in M) = \frac{2x}{1+x^2} \frac{Z^{\gamma}_{C_G}(w,b)}{Z_{C_G}} =  \frac{x}{1+x^2} \langle \chi_{c} \chi_{c'} \rangle^\gamma,
\end{align}
where $x=x_e=\tanh \beta J_e$ is the high-temperature Ising weight, $\tfrac{2x}{1+x^2}$ is the weight of the edge $\{w,b\}$ in the dimer model on $C_G$ as in Fig.~\ref{fig:graphs}, and where $c$ and $c'$ are the two corners of $G$ corresponding to the two roads of $C_G$
that are incident on $w$ and $b$ respectively. Indeed, the first identity is a consequence of the fact that in this case the path $\gamma$ can be chosen empty and therefore
the numerator $Z^{\gamma}_{C_G}(w,b)$ is actually the \emph{unsigned} partition function of dimer covers of the graph where $w$ and $b$ are removed.

We now compute $\langle \chi_{c} \chi_{c'} \rangle^\gamma$ in terms of the Ising two-point function $\mu_{G}[ \sigma_v\sigma_{v'}]$. To this end, recall that $\mathcal E^\emptyset$ is the collection of sets of edges $\eta\subseteq E$ such that each vertex in the graph $(V,\eta)$ has even degree, and $\mathcal E^{\{v,v'\}}$ is the collection of sets of edges such that each vertex has  even 
degree except for $v$ and $v'$ that have odd degree.
Let \[
Z_+:=\mathop{\sum_{\eta \in \mathcal E^{\emptyset}}}_{ e\in \eta} \prod_{ e'\in \eta} x_{e'}, \qquad \text{and} \qquad Z_-:=\mathop{\sum_{\eta \in \mathcal E^{\emptyset}}}_{ e\notin \eta} \prod_{ e'\in \eta} x_{ e'},
\]
and $Z=Z^{\emptyset}_{\text{hT}}$.
By definition~\eqref{def:KCF} of Kadanoff--Ceva fermions with $\gamma$ empty, the high-temperature expansion of spin correlations, and the fact that $\eta \mapsto \eta\triangle \{ e\}$ is a bijection between $\mathcal E^\emptyset$ and $\mathcal E^{\{v,v'\}}$, \eqref{eq:dimfer} gives
\begin{align} \label{eq:Z+-}
\IP^{\emptyset}_{C_G}(\{w,b\}\in M)= \frac{x}{1+x^2} \frac{1}{Z}(x^{-1}Z_++xZ_-)=  \frac{x}{1+x^2}  \mu_{G}[ \sigma_v\sigma_{v'}]. 
\end{align}
The same argument applied to the other street $\{w',b'\}$ corresponding to the same edge $e$ yields $\IP^{\emptyset}_{C_G}(\{w,b\}\in M)=\IP^{\emptyset}_{C_G}(\{w',b'\}\in M)$ as the last displayed expression depends only on $e$. Moreover by the Kramers--Wannier duality and the same computation for the dual Ising model on the dual graph $G^*$, we have 
\begin{align} \label{eq:Z+-1}
\IP^{\emptyset}_{C_G}(\{w,b'\}\in M)= \IP^{\emptyset}_{C_G}(\{w',b\}\in M)= \frac{x^*}{1+(x^*)^2}  \mu_{G^*}[ \sigma_u\sigma_{u'}]=\ \frac{1-x^2}{2(1+x^2)}  \mu_{G^*}[ \sigma_u\sigma_{u'}], 
\end{align}
where $x^*:=(1-x)/(1+x)$ is the dual weight, and where $\{u,u'\}$ is the dual edge of $\{v,v'\}$. This yields the second bullet of the lemma. 

To prove the first bullet of the lemma, we need to relate the dual energy correlators $ \mu_{G}[ \sigma_v\sigma_{v'}]$ and $\mu_{G^*}[ \sigma_u\sigma_{u'}]$ with each other. Interpreting the graphs in $\mathcal E^\emptyset$ as interfaces between spins of different value on the vertices of $G^*$, and using the low-temperature expansion we
get
\[
\mu_{G^*}[ \sigma_u\sigma_{u'}] =  \frac{Z_--Z_+}{Z}.
\]
This together with the second equality of \eqref{eq:Z+-}, and the fact that $Z_++Z_-=Z$, yields 
\[
2x \mu_{G}[ \sigma_v\sigma_{v'}]+(1-x^2)\mu_{G^*}[ \sigma_u\sigma_{u'}]=1+x^2.
\]
Therefore adding \eqref{eq:Z+-} and \eqref{eq:Z+-1} gives
\[
\IP^{\emptyset}_{C_G}(\{w,b\}\in M)+\IP^{\emptyset}_{C_G}(\{w,b'\}\in M)=1/2.
\]
This means that the probability of seeing the road containing $w$ in the dimer configuration is $1/2$. By symmetry this is true for all roads of $C_G$. This finishes the proof.
\end{proof}

\subsection{Kasteleyn theory and complex-valued fermionic observables} \label{sec:Kasteleyn and fermionic}

In this section, we introduce a Kasteleyn orientation which will be directly related to complex-valued observables introduced by Chelkak and Smirnov \cite{CheSmi12}.

\subsubsection{A choice of Kasteleyn's orientation}\label{sec:Kasteleyn}

A \emph{Kasteleyn weighting} of a planar bipartite graph is an assignment of complex phases $\varsigma_e\in\mathbb C$ with $|\varsigma_e|=1$ to the edges of the graph satisfying the \emph{alternating product condition} 
meaning that for each cycle $e_1,e_2,\ldots,e_{2k}$ in the graph, we have
\begin{align}\label{eq:KastW}
\prod_{i=1}^{k} \varsigma_{e_{2i-1}}\varsigma^{-1}_{e_{2i}}=(-1)^{k+1}.
\end{align}
Note that it is enough to check the condition around every bounded face of the graph.

\begin{figure}
		\begin{center}
			\includegraphics[scale=0.7]{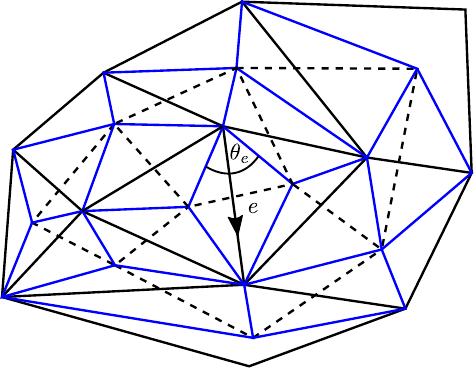}  
		\end{center}
		\caption{A piece of the primal graph $G$ and its dual $G^*$ (black solid and dashed edges respectively) and the corresponding diamond graph (blue edges) used to define the Kasteleyn weighting. 
		We assume that the edges of $G$ and $G^*$ are drawn as straight line segments.}
		Each street $e$ of $C_G$ can be identified with a directed edge of $G$ or $G^*$. Then, the angle $\theta_e$ is the angle in the diamond graph at the origin of this directed edge as depicted in the figure. By definition, these angles sum up to $2\pi$ around every vertex and face of $G$, and around every face of the diamond graph. This guarantees that the associated weighting satisfies the Kasteleyn condition.
			\label{fig:diamond}
\end{figure} 
To define an explicit Kasteleyn weighting for $C_G$, consider the \emph{diamond graph} of~$G$, i.e., the graph whose vertices are the vertices and faces of~$G$, and whose edges 
are the corners of~$G$ (see Fig.~\ref{fig:diamond}). 
Recall that the edges of $C_G$ that correspond to the corners of $G$ are called roads and the remaining edges (forming the quadrangles) are called streets. To each street there is assigned an angle $\theta_e$ between the two neighbouring corners in the diamond graph. We now define
\begin{itemize}[noitemsep]
\item $\varsigma_e=-1$ if $e$ is a road,
\item $\varsigma_e=\exp(\frac{\i}{2}\theta_e )$ if $e$ is a street that crosses a primal edge of $G$,
\item $\varsigma_e=\exp(-\frac{\i}{2}\theta_e )$ if $e$ is a street that crosses a dual edge of $G^*$.
\end{itemize}
That $\varsigma$ is a Kasteleyn orientation of $C_G$ follows from the fact that the angles sum up to $2\pi$ around every vertex and face of $G$, and around every face of the diamond graph.
Note that if $G$ is a finite subgraph of an embedded infinite graph $\Gamma$, then one can as well use the angles from the diamond graph of $\Gamma$ since, as already mentioned, 
one needs to check condition~\eqref{eq:KastW} only on the bounded faces of $C_G$.
In particular, for subgraphs of the square lattice with the standard embedding, we will take $\theta_e=\pi/2$ for all edges $e$.

Fix a bipartite coloring of $C_G$, and let $K=K_{C_G}$ be a Kasteleyn matrix for a dimer model on the bipartite graph $C_G$ with the weighting as above, i.e.,
the matrix whose rows are indexed by the black vertices and the columns by the white vertices, and whose entries are
\[
K(b,w):= \varsigma_{bw} x_{bw}
\]
if $bw$ is an edge of $C_G$ and $K(b,w)=0$ otherwise,
where $b$  and $w$ are respectively black and white vertices, and $x$ is the edge weight for $C_G$ as in Fig.~\ref{fig:graphs}.

We assume that the set of corners of $G$ comes with a prescribed order $c_1,\ldots, c_m$, and we order the rows and columns of $K$ 
according to this order (for each white and black vertex of $C_G$, there is exactly one corner of $G$ that the vertex corresponds to). We denote by $b_i$ and~$w_i$ the black and white vertex of $C_G$ 
corresponding to~$c_i$.

The following lemma is a known observation.

\begin{lemma}
We have that 
\begin{align} \label{eq:inverseKmonomer}
K^{-1}(w_i,b_j) =\i \kappa_\gamma \frac{Z^{\gamma}_{C_G}(w_i,b_j)}{Z_{C_G}},
\end{align}
where $\gamma$ is any dual path connecting a face $u_i$ adjacent to $b_i$ with a face $u_j$ adjacent to $w_j$, $\kappa_\gamma$ is a complex phase depending only on $\gamma$, $w_j$ and $b_i$
(see the proof for a concrete formula), and $Z^{\gamma}_{C_G}(b_i,w_j)$ is, as before, the 
partition function of dimers on $C_G$ with $b_i$ and $w_j$ removed, and with negative weights 
assigned to the edges crossing $\gamma$. 
\end{lemma}

The factor $\i$ is due to an arbitrary choice of $\kappa_\gamma$ which is made for later convenience. 
We will now justify~\eqref{eq:inverseKmonomer} and explicitly identify the complex phase $\kappa_\gamma$ in this expression.

\begin{proof}
To compute the inverse matrix, we use the cofactor representation as a ratio of determinants:
\begin{align*}
K^{-1}(w_i,b_j) = (-1)^{i+j} \frac{\det K^{b_j,w_i}}{\det K},
\end{align*}
where $ K^{w_i,b_j} =: \tilde K$ is the matrix $K$ with the $j$-th row and $i$-th column removed.

By definition of the determinant, we have 
\[
\det K =\sum_{\pi \in S_m} \sgn(\pi) \prod_{k=1}^m \varsigma_{b_kw_{\pi(k)}} x_{b_kw_{\pi(k)}}.
\]
In this sum, only terms where $\pi$ corresponds to a perfect matching on $C_G$ are nonzero. Moreover, by a  classical theorem of Kasteleyn~\cite{Kasteleyn}, the complex phase $\sgn(\pi) \prod_{i=1}^m \varsigma_{b_iw_{\pi(i)}}$
is constant for such $\pi$.
In particular, we can take $\pi$ to be the identity. Since $\varsigma_{b_iw_i}=-1$, we get that \[
\det K=(-1)^{N}Z_{C_G},\]
where $N$ is the number of corner edges in $C_G$.

We now want to interpret $\tilde K$ as a Kasteleyn matrix for the graph $\tilde C_G$ obtained from $C_G$ by removing the vertices $w_i$ and $b_j$.
To this end, if $w_i$ and $b_j$ are not incident on the same face $u_i=u_j$, we need to introduce a sign change to the Kasteleyn weighting along a dual path $\gamma$ which connects $u_i$ to~$u_j$.
We do it as follows. Define modified weights $\tilde \varsigma$ and $\tilde x$ by $\tilde \varsigma_e = -\varsigma_e$ (resp. $\tilde x_e=-x_e$), if $e$ is crossed by $\gamma$, and $\tilde \varsigma_e = \varsigma_e$ (resp. $\tilde x_e=x_e$) otherwise.
Then $\varsigma_ex_e=\tilde \varsigma_e \tilde x_e$, and hence $\tilde K(b,w) = \tilde \varsigma_{bw} \tilde x_{bw}$ if $bw$ is an edge of $\tilde C_G$, and  $\tilde K(b,w)=0$ otherwise.
We leave it to the reader to verify that $\tilde \varsigma$ is indeed a Kasteleyn weighting for $\tilde C_G$. 

We can therefore again apply Kasteleyn's theorem to obtain 
\[
\det \tilde K =\sum_{\pi \in S_{m-1}} \sgn(\pi) \prod_{k=1}^{m-1} \tilde \varsigma_{\tilde b_k \tilde w_{\pi(k)}} \tilde x_{b_kw_{\pi(k)}} = \tilde \kappa_{\gamma} {Z^{\gamma}_{C_G}(w_i,b_j)},
\]
where $\tilde b_1,\ldots,\tilde b_{m-1}$ (resp.\ $\tilde w_1,\ldots, \tilde w_{m-1}$) is an order preserving renumbering of the black (resp.\ white) vertices where $b_j$ (resp.\ $w_i$) is removed, and 
\[
\tilde \kappa_{\gamma}= \sgn(\pi) \prod_{k=1}^{m-1} \tilde \varsigma_{\tilde b_k \tilde w_{\pi(k)}}
\] 
is a constant complex factor independent of the permutation $\pi$ defining a perfect 
matching of $\tilde C_G$.
Setting
\begin{align} \label{eq:kappa}
\kappa_\gamma=(-1)^{i+j+1+N}  \i \tilde \kappa_\gamma
\end{align}
 justifies \eqref{eq:inverseKmonomer}. \end{proof}

We now proceed to giving $\kappa_\gamma$ a concrete representation in terms of the winding angle of~$\gamma$.
To this end, we first need to introduce some complex factors.
We follow \cite{CCK} and for each directed edge or corner $e$, we fix a square root of the corresponding direction in the complex plane
and denote by $\eta_e$ its \emph{complex conjugate}. Recall that we always assume that a corner~$c$ is oriented towards its vertex $v(c)$, and we write $-c$ whenever we consider the opposite orientation. For two directed edges or corners $e,g$ that do not point in opposite directions, we define $\angle(e,g)$ to be the \emph{turning angle} from $e$ to $g$, i.e., the number in $(-\pi,\pi)$ satisfying
\[
e^{-\i \angle(e,g)}=(\overline{\eta_e} \eta_g)^2.
\] 

\begin{lemma}\label{lem:Ksign}
Let $c_i$, $c_j$, and $\gamma$ be as above. Define $\tilde \gamma$ to be the extended path starting at $-c_j$, following $\gamma$, and ending at $c_i$. Then,
\[
\kappa_{\gamma} =  \exp(\tfrac \i 2 \textnormal{wind} (\tilde \gamma)),
\]
where $\textnormal{wind}(\tilde \gamma)$ is the total \emph{winding angle} of the path $\tilde \gamma$, i.e., the sum of all turning angles along the path.
\end{lemma}

\begin{figure} 
		\begin{center}
			\includegraphics[scale=0.9]{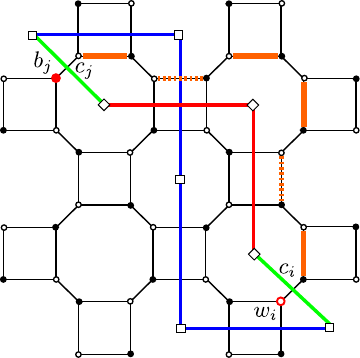}  
		\end{center}
		\caption{An illustration of the proof of Lemma~\ref{lem:Ksign} in the case where $G$ is a piece of the square lattice. The green lines represent corners $c_i$ and $c_j$, the red lines represent the primal path 
		$\rho$ from $v(c_i)$ to $v(c_j)$, and the blue lines show the dual path $\gamma$ from $u(c_j)$ to $u(c_i)$. The red vertices $w_i$ and $b_j$ are removed in the graph $C_{\tilde G}$.
		The matching $M_\rho$ corresponding to $\rho$ contains the orange streets and all remaining roads. The dashed (resp.\ solid) orange edges carry a phase 
		$\exp({\tfrac {\i \pi} 4 })$ (resp.\ $\exp(-{\tfrac {\i\pi} 4 })$) in the original Kasteleyn weighting $\varsigma$ of $C_G$. The orange edge crossed by $\gamma$ gets an 
		additional $-1$ sign in the Kasteleyn weighting $\tilde\varsigma$ of $\tilde C_G$.
		}	\label{fig:Ksign}
\end{figure}
\begin{proof}
Let $\rho$ be a simple primal path starting at $v(c_i)$ and ending at $v(c_j)$, and let $\tilde \rho$ be the extended path that starts at $c_i$, then follows $\rho$, and ends at $-c_j$.
We will define a perfect matching $M_{\rho}$ of $\tilde C_G$ that corresponds to $\rho$ in a natural way (see Fig.~\ref{fig:Ksign}).
Note that there is a unique sequence of streets $S_{\rho}$ such that the first edge contains $b_i$ and the last edge contains $w_j$, and where all the edges are directly to the right of the 
oriented path $\tilde \rho$ (the orange edges in Fig.~\ref{fig:Ksign}). We define $M_\rho$ to contain $S_{\rho}$ and all the remaining roads denoted by $R_{\rho}$.

Moreover, let $\ell$ be the loop (closed path) which is the concatenation of $\tilde \rho$ and $\tilde \gamma$.
We claim that
\begin{align} \label{eq:lrho}
 \prod_{bw\in S_{\rho}} \tilde \varsigma_{bw}= (-1)^{t(\ell)}\prod_{bw\in S_{\rho}}  \varsigma_{bw}=(-1)^{t(\ell)+1}\i\exp(-\tfrac \i2 \textnormal{wind}(\tilde \rho )),
\end{align}
where $t(\ell)$ is the number of self-crossings of $\ell$. 
Indeed, the first identity follows since the self-crossings of $\ell$ only come from a crossing between $\gamma$ and $\rho$, and each such edge gets an additional $-1$ factor in the 
Kasteleyn weighting $\tilde \sigma$.
We now argue for the second inequality by inspecting the contribution of the phases $\varsigma$ at each turn of $\tilde \rho$.
\begin{figure} 
		\begin{center}
			\includegraphics[scale=1.2]{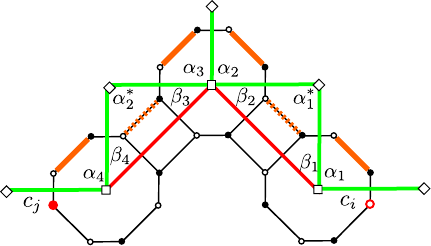}  
		\end{center}
		\caption{An illustration of the proof of~\eqref{eq:lrho}. The path $\rho$ goes from $c_i$ to $-c_j$, and is composed of the two red edges. The orange edges represent $S_{\rho}$.
		}	\label{fig:winding}
\end{figure}
To this end we consider all the corners adjacent to $\rho$. We denote by $\alpha_k$ (resp.\ $\alpha^*_k$), $k=1,2,\ldots,$ 
the unsigned angles between two consecutive corners that share a vertex (resp.\ a face) of $G$,
and by $\beta_k$ we denote the angles between the edges of $\rho$ and the corners (see Fig.~\ref{fig:winding}). Note that there is exactly $|\rho|$ angles of type $\alpha^*$, 
and $2|\rho|$ angles of type $\beta$ (there can be more angles of type $\alpha$).
Moreover, $\alpha^*_k=\pi-\beta_{2k-1}-\beta_{2k}$ for each $k\in\{1,\ldots,|\rho|\}$. 
Finally, the sum of all angles of type $\alpha$ and $\beta$ around a vertex of $G$ is by definition equal to $\pi$ plus the turning angle of $\rho$ at that vertex. 
Writing $A$ (resp.\ $B$) for the sum of all angles of type $\alpha$ (resp.\ $\beta$), and using the definition of~$\varsigma$, we find
\[
\prod_{bw\in S_{\rho}}  \varsigma_{bw}= \prod_{k}e^{-\tfrac{\i \alpha_k}2}\prod_{k}e^{\tfrac{\i \alpha^*_k}2}= e^{{-\tfrac{\i }2(A+B-|\rho|\pi)}}=
 e^{{-\tfrac \i2 (\textnormal{wind}(\tilde \rho )+\pi )}}=-\i \exp({-\tfrac \i2 \textnormal{wind}(\tilde \rho  )}),
\]
which justifies~\eqref{eq:lrho}.

On the other hand, a classical fact due to Whitney~\cite{Whi} (see also~\cite[Lemma 2.2]{CCK}) says that 
\begin{align}\label{eq:crossing}
\exp(\tfrac \i2\textnormal{wind}(\ell))=(-1)^{t(\ell)+1}.
\end{align}
Factorizing the left-hand side into the contributions coming from $\tilde \rho$ and $\tilde \gamma$, we get
\[
\exp(\tfrac \i2\textnormal{wind}(\ell)) =  \kappa_\gamma \exp(\tfrac \i2\textnormal{wind}(\tilde \rho)).
\]
Combining with~\eqref{eq:lrho} we arrive at
\begin{align*}
  \prod_{bw\in M_{\rho}} \tilde \varsigma_{bw}=\prod_{bw\in S_{\rho}} \tilde \varsigma_{bw}\prod_{bw\in R_{\rho}} \tilde \varsigma_{bw}=(-1)^{|R_{\rho}|}\i \kappa_\gamma ,
\end{align*}
where the second equality holds true since roads have complex phase $\varsigma=-1$.
On the other hand, by \eqref{eq:kappa} we have
\[
\kappa_\gamma = (-1)^{i+j+1+N} \sgn(\pi) \i \prod_{k=1}^{m-1} \tilde \varsigma_{b_kw_{\pi(k)}}=(-1)^{i+j+1+N}   \sgn(\pi) \i \prod_{bw\in M_{\rho}} \tilde \varsigma_{bw},
\] 
where $\pi\in S_{k-1}$ is the permutation defining the matching $M_\rho$, and $N$ is the number of all corner edges in $C_G$. 
Therefore to finish the proof, it is enough to show that 
\begin{align} \label{eq:signcheck}
\sgn(\pi)=(-1)^{i+j+N+|R_{\rho}|}.
\end{align} 
To this end, first note that $M_{\rho}$ naturally defines a bijection $\tilde \pi$ of the set of corners of $G$ with the two corners $c_i$ and $c_j$ identified as one corner, called from now on $\tilde c$, where 
$\tilde \pi(c)=c'$ if the black vertex corresponding to $c$ is connected by an edge in $M_\rho$ to the white vertex corresponding to $c'$. 
This bijection can be thought of as a permutation of $\{1,\ldots,k-1\}$ where the index corresponding to $\tilde c$ is $m-1$, and where the first $m-2$ indices 
respect the original order on the remaining corners of~$C_G$. Clearly $\tilde \pi$ has only one nontrivial cycle whose length is $|S_{\rho}|+1$, and hence $\sgn(\tilde \pi)=(-1)^{|S_{\rho}|}$. 
Without loss of generality, let $j>i$ and for an index $l\in \{1,\ldots,k-1\}$, let $p_l\in S_{k-1}$ be the permutation such that $p_l(l)=k-1$ and that does not change the order of the remaining indices. Note that 
$\sgn(p_l)=(-1)^{k-1-l}$ as $p_l$ is a composition of $k-1-l$ transpositions.
One can check that $\pi = p_{i}^{-1}\tilde \pi p_{j-1}$, and as a result $\sgn(\pi) = (-1)^{i+j-1+|S_{\rho}|}$. 
To show~\eqref{eq:signcheck} and finish the proof, we count the roads whose both endpoints are covered by a street in $S_{\rho}$, to get that $N= |S_{\rho}|+1+|R_{\rho}|$.
\end{proof}

All in all, from \eqref{eq:inverseKmonomer} together with Corollary~\ref{cor:main} we obtain the following statement.
\begin{corollary}\label{cor:KastKC} We have
\begin{align} \label{eq:KastKC}
K^{-1}(w_i,b_j)  = \tfrac12 \i \kappa_\gamma \langle \chi_{c_i} \chi_{c_j} \rangle^\gamma,
\end{align}
where the complex phase $\kappa_{\gamma}$ is as in Lemma~\ref{lem:Ksign}.
\end{corollary}

\subsubsection{Complex-valued fermionic observables} 
In this section we rewrite $\langle \chi_{c_i} \chi_{c_j} \rangle$, and hence the right-hand side of~\eqref{eq:KastKC},
in terms of complex-valued fermionic observables of Chelkak--Smirnov~\cite{CheSmi12}, and 
Hongler--Smirnov~\cite{HonSmi}.
This correspondence is well-known (and can be e.g.\ found in~\cite{CCK}) but we choose to present the details for completeness of exposition.
In the next section, we will use it together with the available scaling limit results 
to derive the scaling limit of $K^{-1}$ for the critical model on $C_{D^{\delta}}$.

We first define the complex version of the Kadanoff--Ceva observable for two corners $c_i$ and~$c_j$ by
\begin{align} \label{eq:cornerobs}
f(c_i,c_j) :=\frac{1}{ Z^{\emptyset}_{\text{hT}} } \sum_{\eta \in \mathcal E^{v(c_i),v(c_j)}}  \exp(-\tfrac \i2\textnormal{wind}(\rho_\eta))  \prod_{e\in \eta} x_e,
\end{align}
where $\textnormal{wind}(\rho_\eta)$ is again the total {winding angle} of the path $\rho_\eta$, i.e.~the sum of all turning angles along the path, and where $\rho_\eta$ is a simple path contained in $\eta\cup \{ c_i,c_j\}$ that starts at $ c_i$ and ends at $- c_j$, and is defined as follows: for each vertex $v$ of degree larger than two in $\eta$, one connects the edges around $v$ into pairs in a non-crossing way,
thus giving rise to a collection of non-crossing cycles $\mathcal C_\eta$ and a path from $ c_i$ to $- c_j$ that we call $\rho_\eta$.

It is a standard fact that the definition of $f(c_i,c_j)$ does not depend on the way  the connections at each vertex of $\eta$ are chosen (as long as they are noncrossing). Moreover, for all $\eta\in  \mathcal E^{v(c_i),v(c_j)}$, we have
\begin{align} \label{eq:signwind}
-\overline{ \kappa}_\gamma \exp(-\tfrac \i2\textnormal{wind}(\rho_\eta)) = \textnormal{sgn}_\gamma(\eta),
\end{align}
where as before, $\gamma$ is a fixed dual path connecting $u(c_i)$ and $u(c_j)$, and $\kappa_\gamma =  \exp(\tfrac \i2\textnormal{wind}(\tilde \gamma))$, with $\tilde \gamma$ being the path starting at $-c_j$, then following $\gamma$, and ending at $c_i$. To justify this identity, we consider the loop $\ell$ which is the concatenation of $\rho_\eta$ and the path $\tilde \gamma$, and write
\[
\exp(-\tfrac \i2\textnormal{wind}(\ell)) =\overline{\kappa}_{\gamma} \exp(-\tfrac \i2\textnormal{wind}(\rho_\eta)) .
\]
We then again use Whitney's identity~\eqref{eq:crossing} and the fact that the collection of cycles $\mathcal C_\eta$ must, by construction, cross $\gamma$ an even number of times (since $\mathcal C_\eta$ does not cross $\rho_\eta$, and $\mathcal C_\eta$ crosses $\ell$ an even number of times for topological reasons).
This justifies~\eqref{eq:signwind} and implies that 
\begin{align*}
\langle \chi_{c_i}\chi_{c_j} \rangle^\gamma =-\overline{ \kappa}_\gamma f(c_i,c_j), 
\end{align*}
which together with~Corollary~\ref{cor:KastKC} gives the following proposition.

\begin{proposition}\label{prop:Kfinal} We have
\begin{align} \label{eq:Kfinal}
K^{-1}(w_i,b_j) =- \tfrac\i2  f(c_i,c_j).
\end{align}
\end{proposition}

To make the connection with the scaling limit results of~\cite{HonSmi}, we still need to introduce an observable that is indexed by two directed edges of $G$ instead of two corners.
To this end, for each edge $e$ of $G$, let $z_e$ be its midpoint. Also, for a directed edge $e=(v_1,v_2)$, let $h(e)$ be the \emph{half-edge} 
$\{z_e, v_2\}$, let $-e=(v_2,v_1)$ be its reversal, and let $\bar e=\{v_1,v_2\}$ be its undirected version.
Moreover, for two directed edges $e=(v_1,v_2)$ and $g=(\tilde v_1,\tilde v_2)$, let 
$\mathcal E^{e,g}$ be the collections of edges $\tilde \eta\in\mathcal E^{v_2,\tilde v_1}$ that \emph{do not contain} $\bar e$ and $\bar g$.
We define
\[
f(e,g) :=\frac{1}{ Z^{\emptyset}_{\text{hT}} }\sum_{\tilde \eta \in \mathcal E^{e,g}}  \exp(-\tfrac \i2\textnormal{wind}(\rho_{\tilde \eta}))  \prod_{e\in \tilde \eta} x_e ,
\]
where $\rho_{\tilde \eta}$ is a simple path in $\tilde\eta \cup \{ h(e), h(-g)\}$ that starts at $z_e$ and ends at $z_g$, and is analogous to $\rho_\eta$ from \eqref{eq:cornerobs}.
Note that the winding of $\rho_{\tilde \eta}$ is constant (independent of $\tilde \eta$) modulo $2\pi$ and equal to $\angle (e,g)$, and therefore
\begin{align}\label{eq:conswind}
f(e,g) \in    \overline{\eta}_{e} {\eta}_{g} \mathbb R.
\end{align}

\section{Convergence of the nesting field (proof of Theorem~\ref{thm:NFconv})}\label{sec:moments}
Let $D\subset \mathbb C$ be a Jordan domain, and let $D^\delta$ approximate $D$, i.e.~$d(\partial D^\delta, \partial D)\to 0$ as $\delta \to 0$ (where $d$ is as in~\eqref{eq:curvedist}). We consider the critical double random current model with free boundary conditions on $D^{\delta}$,
and the corresponding dimer model on Dub\'{e}dat's square-octagon graph $C_{D^{\delta}}$. We call $U^{\delta}$ and $V^{\delta}$ the set of faces of $C_{D^\delta}$ that correspond to the faces and vertices of $D^{\delta}$ respectively.
In this section we show that the moments of the associated height function $h^{\delta}$ converge to the moments of $\frac{1}{\sqrt{\pi}}$ times the Dirichlet GFF.

\subsection{Scaling limit of inverse Kasteleyn matrix} \label{sec:scaling}
We start by establishing the scaling limit of the inverse Kasteleyn matrix on $C_{D^{\delta}}$. This is crucial for the computation of the moments of the height function that is done in the next section.

Our method is to use Proposition~\ref{prop:Kfinal} obtained in the previous section, as well as the existing scaling limit results for discrete s-holomorphic observables in the Ising model~\cite{CHI,HonSmi}.
It is important to note that to prove our main results, we need to work with continuum domains $D$ with an arbitrary (possibly fractal) boundary.
Therefore, we state a generalized version of the scaling limit results of Hongler and Smirnov~\cite{HonSmi} for the critical fermionic observable with two points in the bulk of the domain. 
Their result, as stated, is valid only for domains whose boundary is a rectifiable curve (see also~\cite{hongler}). 
Even though the stronger result that we need is most likely known to the experts, for the sake of completeness, we will 
outline its proof, which is a direct consequence of the robust framework of Chelkak, Hongler and Izyurov~\cite{CHI} that was used to establish scaling limits for critical spin correlations.

From now on, we assume that the observables are critical, i.e., the weight $x_e$ is constant and equal to $x_c=\sqrt{2}-1$ so that $  \prod_{e\in \eta} x_e=x_c^{|\eta|}$.
Also, we define
\begin{align}\label{eq:obssmi}
f(e,z_{g}):= x_c(f(e,g)+f(e,- g)),
\end{align}
which is the observable of Hongler and Smirnov~\cite{HonSmi} (when $e$ is a horizontal edge pointing to the right) that is indexed by a directed edge $e$ and a midpoint of an edge $z_g$.
The next lemma relates this observable to the corner observable in a linear fashion. This type of identities is well known (see e.g.~\cite{CCK}) and is closely related to the notion of \emph{s-holomorphicity} introduced by Smirnov~\cite{smirnov} for the square lattice, and generalized by Chelkak and Smirnov~\cite{CheSmi12}, and Chelkak~\cite{Che17,Che20}. We omit the proof. 
\begin{lemma} \label{lem:projections}
Let $c_i$ and $c_j$ be two corners that do not share a vertex, and let $e$ and $g$ be directed edges incident to $v(c_i)$ and $v(c_j)$ respectively. Then
\[
f(c_i,c_j) =\frac 1{\sqrt{2}}\sum_{e'\in\{e,-{e}\}}\big(1+ (\overline{\eta_{c_i}} \eta_{e'})^2\big) \big(f(e',z_{g}) -  (\overline{\eta_{e'}} {\eta_{ c_j}})^2 \overline{f(e',z_{g}) }\big).
\]
\end{lemma}

We also need to introduce the continuum counterparts of the discrete holomorphic observables. To this end,
let $D\subsetneq \mathbb C$ be a simply connected domain different from $\mathbb C$, and let $\psi_w=\psi^{D}_w$ be the unique conformal map from $D$ to the unit disk  with $\psi_w (w)=0$ and $\psi'_w (w)>0$.
For $w,z\in D$, we define
\[
f_-^D(w,z):=\frac1{2\pi} \sqrt{\psi_w'(w)\psi_w'(z)} \quad  \text{and} \quad f_+^D(w,z):=  \frac1{2\pi} \sqrt{\psi_w'(w)\psi_w'(z)}\frac{1}{\psi_w(z)}.
\]

\begin{lemma}[Conformal covariance of $f^D_{\pm}$] \label{lem:confcov}
Let $\varphi: D \to  D'$ be a conformal map. Then
\begin{align*}
f_-^D(w,z) &= {\overline{\varphi'(w)}}^{\tfrac 12}{\varphi'(z)}^{\tfrac 12} f_-^{ D'}(\varphi(w),\varphi(z)), \\
f_+^D(w,z) &={{\varphi'(w)}}^{\tfrac 12}{\varphi'(z)}^{\tfrac 12} f_+^{ D'}(\varphi(w),\varphi(z)).
\end{align*}
Moreover, for the upper half-plane $\mathbb H$, we have 
\[
f_-^{\mathbb{H}}(w,z) =  \frac{\i}{2\pi(z-\overline w)} \quad \text{and} \quad f_+^{\mathbb{H}}(w,z) =  \frac{1}{2\pi(z- w)}.
\]
\end{lemma}
\begin{proof} To prove the first part, note that $\psi^{ D'}_{\varphi(w)}(z)= \psi_w^D(\varphi^{-1}(z))\frac{\varphi'(w)}{|\varphi'(w)|}$. Indeed, the right-hand side is a 
conformal map with a positive derivative $(\psi_w^D)'(w)/|\varphi'(w)|$ and vanishing at $\varphi(w)$. Hence we have
\begin{align*}
f_-^{ D'} (\varphi(w),\varphi(z)) &= [ (\psi^{ D'}_{\varphi(w)})'(\varphi(w))(\psi^{ D'}_{\varphi(w)})'(\varphi(z))]^{\tfrac 12}\\
&= [(\psi^D_w)'(w)(\psi^D_w)'(z)]^{\frac 1 2} [{\varphi'(w)\varphi'(z)]^{-\tfrac12}} \frac{\varphi'(w)}{|\varphi'(w)|}\\
&= f_-^{ D} (w,z)  {\overline{\varphi'(w)}^{-\tfrac12}\varphi'(z)^{-\tfrac12}},
\end{align*}
and similarly for $f^D_+$. 
The second part follows from the fact that $\psi^{\mathbb{H}}_w(z) = i \frac{z-w}{z-\overline w}$ and the definition of $f_{\pm}^{\mathbb{H}}$.
\end{proof}

We now proceed to the generalization of~\cite[Theorem 8]{HonSmi} mentioned at the beginning of the section.
In the proof we will very closely follow the proof of~\cite[Theorem 2.16]{CHI} dealing with the convergence of discrete s-holomorphic spinors.

\begin{theorem}\label{lem:HonSmi}
Let $D\subset \mathbb C$ be a bounded simply connected domain, and let $D^{\delta}$ approximate $D$ as $\delta\to0.$
Fix $w,z\in D$, and let $e=e^{\delta}$ and $g=g^{\delta}$ be edges of $D^{\delta}$ whose midpoints converge to $w$ and $z$ respectively as $\delta \to 0$. Then
\[
f^{{\delta}}(e,z_{g}) = \delta\big(f_-^D(w,z)+  {\overline{\eta}_{e}^2}f_+^D(w,z) + o(1)\big)\qquad \text{as } \delta \to 0,
\]
where $f^\delta$ is the observable from~\eqref{eq:obssmi} defined on $D^\delta$.
Moreover the convergence is uniform on compact subsets of $\{(w,z) \in D^2: w\neq z\}$.
\end{theorem}
Before giving a sketch of the proof of this theorem, we state a corollary that will be convenient for us when computing moments of the height function in the next section.

\begin{corollary} \label{cor:fermscale} Consider the setting from the theorem above and let $c_i=c_i^{\delta}$ and $c_j=c_j^{\delta}$ be two corners of $D^{\delta}$ whose vertices converge to $w$ and $z$ respectively.
Then
\begin{align*}
K^{-1}(w_i,b_j)&=-\frac{1 }{\sqrt2} {\delta \i}\big(f^D_-(w,z)- \overline{\eta}_{c_i}^2\eta_{c_j}^2 \overline{f^D_-(w,z)} +  \overline{\eta}^2_{c_i}f^D_+(w,z)  - {\eta^2_{c_j}}\overline{f^D_+(w,z)}  +o(1)\big),
 \end{align*}
 where $K^{-1}$ is the inverse Kasteleyn matrix on $C_{D^{\delta}}$.
\end{corollary}
\begin{proof}
To simplify the notation, we drop $D$ from the superscripts. We combine Lemma~\ref{lem:HonSmi} and Lemma~\ref{lem:projections} to get that $\frac {\sqrt{2}}\delta f(c_i,c_j)$ equals to
\begin{align*}
& \sum_{e'\in\{e,- e\}}\big(1+ (\overline{\eta}_{c_i} \eta_{e'})^2) (f(e',z_{g}) -  (\overline{\eta}_{e'}{\eta_{ c_j}})^2 \overline{f(e',z_{g}) }) \\
&= \sum_{e'\in\{e,- e\}}\big(1+ (\overline{\eta}_{c_i} \eta_{e'})^2) (f_-(w,z)+  \overline{\eta}_{e'}^2f_+(w,z)-  (\overline{\eta}_{e'} {\eta_{ c_j}})^2\overline{f_-(w,z)} -\eta_{ c_j}^2\overline{f_+(w,z)})+o(1) \\
&=2 \big( f_-(w,z)+ \overline{\eta}_{c_i}^2f_+(w,z) - \overline{\eta}_{c_i}^2{\eta^2_{c_j}\overline{f_-(w,z)}}- 
{\eta_{c_j}^2\overline{f_+(w,z)}} \big) +o(1),
\end{align*}
where the last equality holds due to cancellations resulting from $\eta_{e}^2=-\eta_{- e}^2$.
On the other hand, by \eqref{eq:Kfinal}, $K^{-1}(w_i,b_j)=-\tfrac\i2  f(c_i,c_j)$
which finishes the proof.
\end{proof}

\begin{proof}[Sketch of proof of Theorem~\ref{lem:HonSmi}]
Based on the scaling limit results of Hongler--Smirnov~\cite{HonSmi}, we first argue that the statement holds true for a domain $D$ with a smooth boundary.
Indeed, in~\cite{HonSmi} it is assumed that $\eta_e^2=1$ and hence, in that case, the result follows directly from \cite[Theorem~8]{HonSmi}.
Applying this to a rotated domain together with the conformal covariance properties from~Lemma~\ref{lem:confcov} yields the statement for a general direction of $e$.

We now briefly describe how to use the robust framework of Chelkak, Hongler and Izyurov to extend this to general simply connected domains. In~\cite[Theorem~2.16]{CHI}, 
a scaling limit result was established for a discrete holomorphic spinor $F^\delta$ defined on an approximation $D^\delta$ of an arbitrary bounded simply connected domain $D$.
The two observables $F^\delta$ and $f^\delta$ satisfy the same 
boundary conditions (of~\cite[Proposition~18]{HonSmi} and \cite[(2.7)]{CHI}). 
Moreover, both observables are s-holomorphic away from the diagonal. The difference however is their singular behaviour near the diagonal. 
In~\cite{CHI}, the full plane version $F_{\mathbb C}^\delta$ (the discrete analog of $1/\sqrt{z-w}$) of the observable is subtracted from $F^\delta$ in order to cancel out the discrete-holomorphic singularity on the diagonal.
The details of the proof of \cite[Theorem~2.16]{CHI} can be carried out verbatim for $f^\delta$ instead of $F^\delta$ and its full plane version $f_{\mathbb C}^\delta$ (the discrete analog of $1/(z-w)$) introduced 
in~\cite{HonSmi} instead of $F_{\mathbb C}^\delta$. Indeed, the arguments in~\cite{CHI} depend only on the fact that the observables in question are 
s-holomorphic and satisfy the correct boundary value problem. 

Since the scaling limit is conformally invariant and was uniquely identified for domains with a smooth boundary by the argument above, this finishes the proof.
\end{proof}

\subsection{Moments of $h^{\delta}$}
Throughout this section, and as before, let $D\subset \mathbb C$ be a Jordan domain, and let $D^\delta$ approximate $D$, i.e.~$d(\partial D^\delta, \partial D)\to 0$ as $\delta \to 0$ (where $d$ is as in~\eqref{eq:curvedist}).
For simplicity of exposition, we only consider the height function on $C_{D^{\delta}}$ restricted to $U^{\delta}$ which has the same distribution as the nesting field of the critical double random current on $D$ with free boundary conditions. The case of mixed moments (for the joint height function on both the faces and vertices of $D^{\delta}$) follows in the same manner as the faces and 
vertices of $D^{\delta}$ play a symmetric role in the graph $C_{D^\delta}$.
To this end, let $a_1, a_2, \ldots ,a_n$ be distinct points in $D$, and let $h^{\delta}(a_i)$ $(i=1,\ldots,n)$ be the height 
function evaluated at the face $u^\delta_i=u^\delta_i(a_i)\in U^{\delta}$ of $D^\delta$, in which the point $a_i$ lies
(we choose a face arbitrarily if $a_i$ lies on an edge of $D^\delta$).                                     

Let $G_D(z,w)$ be the Dirichlet Green's function in $D$, i.e., the Green's function of standard Brownian motion in $D$ killed upon hitting $\partial D$.
In particular for the upper-half plane $\mathbb{H}$, we have 
\[
G_{\mathbb H}(z,w)= \frac{1}{2\pi} \ln \Big| \frac {z-\overline{w}}{z-w}\Big|.
\]
This section is devoted to the proof of the following theorem. Below,  $\mathbf{P}^{\emptyset,\emptyset}_{D^{\delta}, D^{\delta}}$ denotes the probability measure of the double random current model 
with free boundary conditions together with the independent labels used to define the nesting field.
\begin{theorem} \label{thm:moments} For every even integer $n$ and any distinct points $a_1,a_2, \ldots,a_n\in D$, we have
\[
\lim_{\delta\rightarrow0}\mathbf{E}^{\emptyset,\emptyset}_{D^{\delta}, D^{\delta}}\Big [ \prod_{i=1}^n h^{\delta}(a_i) \Big ] =\sum_{\pi \text{ pairing of } \{ a_1,\dots,a_n\}} \prod_{\{z,w \} \in \pi } \tfrac1{\pi} G_D(z,w),
\]   
where a pairing is a partition into sets of size two.
\end{theorem}
Note that the field $h^{\delta}$ is symmetric, and therefore the corresponding moments for $n$ odd vanish.

Kasteleyn theory classically allows to compute all moments of the height function in terms of the inverse Kasteleyn matrix $K^{-1}$.
In the proof of the theorem, we follow the line of computation due to Kenyon~\cite{Ken00} but with several adjustments to our setting. 
In particular, we start with an algebraic manipulation to take care of the behaviour of $K^{-1}$ near the boundary of $D^{\delta}$: for $a_1^0,\ldots, a_n^0 \in D$, write
\begin{align} \label{eq:sumsplit}
\mathbf{E}^{\emptyset,\emptyset}_{D^{\delta}, D^{\delta}}\Big [ \prod_{i=1}^n h^{\delta}(a_i) \Big ]& = \mathbf{E}^{\emptyset,\emptyset}_{D^{\delta}, D^{\delta}}\Big [\prod_{i=1}^n (h^{\delta}(a_i)-h^{\delta}(a_{i}^0) )\Big] - \mathop{\sum_{t\in\{ 0,1\}^n}}_{t\neq(1,\ldots,1)}(-1)^{\sum_i (1-t_i)}\mathbf{E}^{\emptyset,\emptyset}_{D^{\delta}, D^{\delta}}\Big [ \prod_{i=1}^n h^{\delta}(a_{i}^{t_i})\Big],  
\end{align}
where $a_i^1=a_i$ for $i=1,\dots,n$.

The advantage of this formulation is that the first term on the right-hand side can be computed using Kasteleyn theory, and that the others are small when $a_1^0, \ldots, a_n^0$ are close to the boundary. This latter fact is not obvious and is relying on discrete properties of the double random current obtained in \cite{DumLisQia21} (note that it is basically saying that the field is uniformly small -- in terms of moments -- near the boundary).

We start by proving that the remaining terms are small.

\begin{proposition}\label{prop:remaining terms}
 For any $\varepsilon >0$ and $a_1,\dots,a_n\in D$, 
one may choose $a_1^0,\ldots, a_n^0\in D$ so that 
\begin{align} \label{eq:boundaryapprox}
\Big |\mathbf{E}^{\emptyset,\emptyset}_{D^{\delta}, D^{\delta}}\Big [ \prod_{i=1}^n h^{\delta}(a_i) \Big ]- \mathbf{E}^{\emptyset,\emptyset}_{D^{\delta}, D^{\delta}}\Big [\prod_{i=1}^n (h^{\delta}(a_i)-h^{\delta}(a_{i}^0) )\Big]\Big | < \varepsilon
\end{align}
uniformly in $\delta>0$.
\end{proposition}

\begin{remark}\label{rmk:boundary conditions}This proposition, which basically claims that the second term on the right-hand side of \eqref{eq:sumsplit} is close to zero provided the $a_i^0$ are close enough to the boundary, is a restatement of the fact that boundary conditions for the limiting height function are zero. It is therefore the main place where we identify boundary conditions. Note that this proposition relies heavily on the main results from \cite{DumLisQia21} (Theorem~\ref{prop:tight number crossings} and
Theorem~\ref{thm:crossing free} from Section~\ref{sec:input discrete}), and is as such non-trivial.\end{remark}

To prove this proposition, we need to introduce some auxiliary notions. We say that a cluster of 
the double random current is \emph{relevant} for $ A =\{a_1, \ldots, a_n\} \subsetneq D$ if it is odd around $u_i^\delta$ for at least two different 
$i\in \{1,\ldots,n\}$ (it is possible that $u^\delta_i=u^\delta_j$ 
even though $a_i\neq a_j$).
We denote by $\mathsf R^{\delta}({ A})$ the number of relevant clusters  for $ A$ in $D^{\delta}$, and by $\mathsf I^{\delta}( A)$ the event that all faces 
$u^{\delta}_1,\ldots,u^{\delta}_n$ are surrounded by at least one relevant  cluster for~$ A$. We start with three lemmata.
\begin{lemma} \label{lem:relevantcluster}
For every $n\geq 2$ even, there exists $P_n\in(0,\infty)$ such that for all sets of points $ A= \{ a_1,\ldots,a_n \} \subsetneq D$, we have
\begin{align*}
\mathbf{E}^{\emptyset,\emptyset}_{D^{\delta}, D^{\delta}}\Big [ \prod_{i=1}^n h^{\delta}(a_i)\Big ]   \leq P_n \sqrt{\mathbf{E}^{\emptyset,\emptyset}_{D^{\delta}, D^{\delta}} [\mathsf R^{\delta}( A )^n] \mathbf{P}^{\emptyset,\emptyset}_{D^{\delta}, D^{\delta}}[\mathsf I^{\delta}( A)]}.
\end{align*}
\end{lemma}
\begin{proof}
For a cluster $\mathcal C$ of the double random current, let 
\[
\text{Odd}(\mathcal C):=\{ a_i \in  A: \mathcal C \text{ is odd around } u^\delta_i\}.\]
We denote a partition of $ A$ by $\{ A_1,\ldots,A_k\}$. We call such a partition even if all its elements have even cardinality.
Using the correspondence with the nesting field of the critical double random current on $D^{\delta}$ with free boundary conditions defined in~\eqref{def:nestingfield}, we have
\begin{align*}
\mathbf{E}^{\emptyset,\emptyset}_{D^{\delta}, D^{\delta}}\Big [ \prod_{i=1}^n h^{\delta}(a_i)\Big ] & = \mathbf{E}^{\emptyset,\emptyset}_{D^{\delta}, D^{\delta}}\Big [\prod_{i=1}^n \Big( \sum_{\mathcal{C}_i } \coin_{\mathcal C_i} \mathbf 1_{\{\mathcal C_i \text{ odd around } u^\delta_i \}} \Big)\Big ] \\
&= \mathbf{E}^{\emptyset,\emptyset}_{D^{\delta}, D^{\delta}}\Big [\sum_{(\mathcal{C}_1,\ldots, \mathcal C_n) } \prod_{i=1}^n  \coin_{\mathcal C_i}\mathbf 1_{\{\mathcal C_i \text{ odd around } u^\delta_i \}}\Big ] \\
&= \sum_{\{ A_1,\ldots,A_k\} \textnormal{ even}}\mathbf{E}^{\emptyset,\emptyset}_{D^{\delta}, D^{\delta}}\Big [\sum_{(\mathcal{C}_1,\ldots, \mathcal C_k) }  \mathbf{1}_{\{ A_i \subseteq \text{Odd}(\mathcal C_i),\ \mathcal C_i \text{ distinct } \forall i\in\{1,\ldots,k\}\}}\Big]\\ 
&\leq \sum_{\{ A_1,\ldots,A_k\} \textnormal{ even}}\mathbf{E}^{\emptyset,\emptyset}_{D^{\delta}, D^{\delta}}\Big [\sum_{(\mathcal{C}_1,\ldots, \mathcal C_{k}) } \mathbf{1}_{\{ \mathcal C_i \text{ relevant for }  A \}} \mathbf{1}_{\mathsf I^{\delta}( A)} \Big] \\
&\leq P_n\mathbf{E}^{\emptyset,\emptyset}_{D^{\delta}, D^{\delta}}\big [\mathsf R^{\delta}( A )^{n/2}  \mathbf{1}_{\mathsf I^{\delta}( A) } \big] \\
&\leq P_n \sqrt{\mathbf{E}^{\emptyset,\emptyset}_{D^{\delta}, D^{\delta}} [\mathsf R^{\delta}( A )^n]\mathbf{P}^{\emptyset,\emptyset}_{D^{\delta}, D^{\delta}}[\mathsf I^{\delta}( A)]},
\end{align*}
where $P_n$ is the number of even partitions of a set of size $n$ (we used that $k\le n/2$), and where in the last inequality we used the Cauchy--Schwarz inequality.
\end{proof}

\begin{lemma}[Logarithmic bound on the number of clusters]\label{cor:tightness number logarithmic}
There exists $C\in(0,\infty)$ such that for every bounded domain $D$ and every $A=\{a_1,\dots,a_n\}\subsetneq D$ and $N\geq 1$,
\[
{\bf E}_{D^\delta,D^\delta}^{\emptyset,\emptyset}[\mathsf R^\delta(A)^N]\le  \frac{1}{N!}\Big [C n\log\Big(\frac{{\rm diam}(D)}{\min_{i\neq j} |a_i-a_j|}\Big)\Big]^{N},
\]
uniformly in $\delta>0$.
\end{lemma}

\begin{proof}Consider the constant $C$ given by Theorem~\ref{prop:tight number crossings}. Set $\kappa:=\tfrac12\min_{i\neq j} |a_i-a_j|$ and $d:={\rm diam}(D)$. 

Consider the family $\mathcal B=(\Lambda_{r_k}(x_k):k\in \mathcal K)$ containing the boxes $\Lambda_{\tfrac r{4C}}(x)$ with $r:=2^j\kappa$, $x\in \tfrac r{4C}\mathbb Z^2\cap{\rm  Ann}(a_i,r,2r)$ for every $1\le i\le n$ and $0\le j\le \lfloor\log_2 (d/\kappa)\rfloor$.
One may easily check that every cluster that surrounds at least two vertices in $A$ must contain, for some $k\in\mathcal K$, a crossing from $\Lambda_{r_k}(x_k)$ to $\Lambda_{2Cr_k}(x_k)$. We deduce that if $X_k$ is the number of disjoint $\Lambda_{Cr_k}(x_k)$-clusters crossing ${\rm Ann}(x_k,r_k,2Cr_k)$ from inside to outside, then
\[
\mathsf R^\delta(A)\le \sum_{k\in \mathcal K}X_k.
\]
Now, for each $k\in \mathcal K$, $\Lambda_{3Cr_k}(x_k)$ intersects at most $O(C^2)$ boxes $\Lambda_{3Cr_l}(x_l)$ for $l\in\mathcal K$. We may therefore partition $\mathcal K$ in $I=O(C^2)$ disjoint sets $\mathcal K_1,\dots,\mathcal K_I$ for which the $\Lambda_{3Cr_k}(x_k)$ with $k\in \mathcal K_i$ are all disjoint. Set $S_i:=\sum_{k\in\mathcal K_i}X_k$. H\"older's inequality implies that
\[
{\bf E}_{D^\delta,D^\delta}^{\emptyset,\emptyset}[\mathsf R^\delta(A)^N]\le {\bf E}_{D^\delta,D^\delta}^{\emptyset,\emptyset}[(S_1+\dots+S_{|I|})^N]\le |I|^{N-1}\sum_{i=1}^{|I|}{\bf E}_{D^\delta,D^\delta}^{\emptyset,\emptyset}[S_i^N].
\]
The mixing property of the double random current proved in \cite{DumLisQia21} and Theorem~\ref{prop:tight number crossings} imply the existence of $C_{\rm mix}\in(0,\infty)$ (independent of everything) such that $S_i$ is stochastically dominated by $C_{\rm mix}\widetilde S_i$, where $\widetilde S_i$ is the sum of $|\mathcal K_i|$ independent Geometric random variables $(\widetilde X_k:k\in \mathcal K_i)$ of parameter $1/2$. 
We deduce that 
\[
{\bf E}_{D^\delta,D^\delta}^{\emptyset,\emptyset}[S_i^N]\le C_{\rm mix}^N\times \frac{(C_0|\mathcal K_i|)^N}{N!}.
\]
Since $|\mathcal K_i|\le |\mathcal K|\le C_1 n\log(d/\kappa)$, we deduce that 
\[
{\bf E}_{D^\delta,D^\delta}^{\emptyset,\emptyset}[\mathsf R^\delta(A)^N]\le \frac{(C_2n\log(d/\kappa))^N}{N!}.
\]
This concludes the proof.
\end{proof}

We now turn to the third lemma that we will need. Let $\partial_\alpha \Omega$ be the set of points in~$\Omega$ that are exactly at a Euclidean distance equal to $\alpha$ away from $\partial \Omega$.
\begin{lemma}[Large double random current clusters do not come close to the boundary]\label{cor:not touching boundary}
For every $C,\alpha,\ep>0$, there exists $\beta=\beta(C,\alpha,\ep)>0$ such that for every $D\subseteq \Lambda_{C}$,
\begin{equation}
{\bf P}_{D^\delta,D^\delta}^{\emptyset,\emptyset}[\partial_{\alpha}D\stackrel{\n_1+\n_2}\longleftrightarrow \partial_\beta D]\le \ep.
\end{equation}
\end{lemma}

\begin{proof}
Assume that $\partial_\alpha D$ is not empty otherwise there is nothing to prove. Since $D\subseteq \Lambda_{C}$, one may find a collection of $k=O((C/\alpha)^2)$ vertices $x_1,\dots,x_k\in \tfrac13\alpha\mathbb Z^2$ such that 
\begin{itemize}[noitemsep]
\item $\Lambda_{2\alpha/3}(x_i)\subseteq D$  for $1\le i\le k$;
\item $\Lambda_{\alpha}(x_i)\not\subseteq D$ for $1\le i\le k$;
\item $\partial_\alpha D\subseteq \Lambda_{\alpha/3}(x_1)\cup\dots\cup \Lambda_{\alpha/3}(x_k)$.
\end{itemize}
Then, Theorem~\ref{thm:crossing free} implies that
\begin{equation}
{\bf P}_{D^\delta,D^\delta}^{\emptyset,\emptyset}[\partial_{\alpha}D\stackrel{\n_1+\n_2}\longleftrightarrow \partial_{\beta}D]\le \sum_{i=1}^k {\bf P}_{D^\delta,D^\delta}^{\emptyset,\emptyset}[\Lambda_{\alpha/3}(x_i)\stackrel{\n_1+\n_2}\longleftrightarrow \partial_{\beta}D]\le k \epsilon(\beta/\alpha).
\end{equation}
We then choose $\beta$ so that the right-hand side is smaller than $\ep$.
\end{proof}

These ingredients are enough for the proof of Proposition~\ref{prop:remaining terms}.

\begin{proof}[Proof of Proposition~\ref{prop:remaining terms}]
First, Lemma~\ref{cor:tightness number logarithmic} shows that 
for every $n\geq 2$, there exist $C_n,M_n<\infty$ such that for all sets of points $ A= \{ a_1,\ldots,a_n \} \subsetneq D$, we have
\begin{equation}\label{eq:logmoment}
\mathbf{E}^{\emptyset,\emptyset}_{D^{\delta}, D^{\delta}} [\mathsf R^{\delta}( A)^n] \leq C_n\big |\log(\min_{i\neq j} |a_i-a_j|)\wedge \log\tfrac1\delta)\big|^{M_n}.
\end{equation}
Lemma~\ref{cor:not touching boundary} implies that for every $n\geq2$ and every $\eta>0$, there exists a function $\rho:[0,\infty) \to [0,\infty)$ satisfying $\rho(0)=0$ and continuous at $0$, and such that for all $\delta$ and all sets of points $ A= \{ a_1,\ldots,a_n \} \subsetneq D$ that are pairwise at least $\eta$ away from each other, we have
\[
\mathbf{P}^{\emptyset,\emptyset}_{D^{\delta}, D^{\delta}}[\mathsf I^{\delta}( A)] \leq \rho(\min_i\textnormal{dist}(u_i,\partial D)).
\]
The proof is then a direct combination of these two inequalities with Lemma~\ref{lem:relevantcluster} and \eqref{eq:sumsplit}.\end{proof}

We now turn to the computation of the first term on the right-hand side of \eqref{eq:sumsplit} using the approach of Kenyon~\cite{Ken00}. 
The next result is an analog of \cite[Proposition 20]{Ken00}.
\begin{proposition}\label{prop:intdet} Let $a_1,a_1^0,\dots,a_n,a_n^0$ be distinct points in $D$, and let $\gamma_1,\ldots,\gamma_n$ be pairwise disjoint curves in $D$ connecting $a_i^0$ to $a_i$ for $i=1,\ldots,n$.
Then,
\[
\lim_{\delta\rightarrow0}\mathbf{E}^{\emptyset,\emptyset}_{D^{\delta}, D^{\delta}}\Big [\prod_{i=1}^n (h^{\delta}(a_i)-h^{\delta}(a_{i}^0) )\Big]= \i^n \sum_{\epsilon \in\{ \pm 1\}^n} \prod_{k=1}^n \epsilon_k\int_{\gamma_1} \cdots\int_{\gamma_n} \det \big[ f_{\epsilon_i,\epsilon_j} (z_i,z_j)\big]_{1\leq i,j \leq n}  dz_1^{(\epsilon_1)} \cdots dz_n^{(\epsilon_n)},
 \]
where $dz_i^{(1)}=dz_i$, $dz_i^{(-1)}= d\overline{z_i}$, 
and 
\begin{align*}
f_{\epsilon_i,\epsilon_j}(z_i,z_j)=\begin{cases}
0 & \text{if } i=j, \\
 f_-(z_i,z_j) & \text{if } (\epsilon_i,\epsilon_j)=(-1,1), \\
  f_+(z_i,z_j) & \text{if } (\epsilon_i,\epsilon_j)=(1,1), \\
   \overline{f_-(z_i,z_j)} & \text{if } (\epsilon_i,\epsilon_j)=(1,-1), \\
    \overline{f_+(z_i,z_j)} & \text{if } (\epsilon_i,\epsilon_j)=(-1,-1). \\
\end{cases}
\end{align*}
Moreover the limit is conformally invariant.
\end{proposition}

\begin{proof}
We start by proving a stronger version of the conformal invariance statement. Namely, if one expands the determinant under the integrals as a sum of terms over permutations $\iota$,
then each multiple integral of the term $T_{\epsilon,\iota}$ corresponding to a fixed $\epsilon$ and $\iota$ is conformally invariant. This follows from the conformal covariance of the functions $f_{\pm}(z_i,z_j)$ stated in
Lemma~\ref{lem:confcov} and an integration by substitution. Indeed, it is enough to notice that $T_{\epsilon,\iota}$ is a product of $n$ functions $f_{\pm}(z_i,z_j)$ or their conjugates 
with the property that each variable $z_i$ appears in it exactly twice and in a way that, under a conformal map $\varphi$, it contributes a factor $\varphi'(z_i)$ if $\epsilon_i=1$ and $\overline{\varphi'(z_i)}$ if $\epsilon_i=-1$.

We now turn to the convergence part.
To this end, for $i=1,\ldots,n$ and every $\delta$ small enough, we fix a dual path  $\gamma^{\delta}_i$ connecting $(u_{i}^0)^{\delta}$ with $u_{i}^{\delta}$ that converges uniformly to $\gamma_i$.
It will be convenient to choose the paths $\gamma^\delta_i$ in such a way that:
\begin{itemize}
\item the faces of $C_{D^\delta}$ visited by each $\gamma^\delta_i$ alternate with each step between $U^{\delta}$ and $V^{\delta}$ (by definition, the paths start and end in $U^{\delta}$),
\item the restriction of each $\gamma^\delta_i$ to $U^{\delta}$ is a path in the dual of $D^{\delta}$, meaning that consecutive faces share an edge in $D^{\delta}$,
\item the restriction of each $\gamma^\delta_i$ to $V^{\delta}$ is a path in $D^{\delta}$ given by the left endpoints of the edges of $D^{\delta}$ crossed by the path.
\end{itemize}
Note that such paths exist (for $\delta$ small enough), and they only cross corner edges of $C_{D^\delta}$.

We enumerate the edges crossed by $\gamma^{\delta}_i$ (there is always an even number of them) using the symbols $c_{i,1}^+,c_{i,1}^-,\ldots,c_{i,l_i}^{+},c_{i,l_i}^{-}$.
With a slight abuse of notation we will also write $c_{i,t}^{\pm}$ for the indicator functions that the edge belongs to the dimer cover, and 
 $\hat{c}_{i,t}^{\pm}:=c_{i,t}^{\pm} -\mathbf E[c_{i,t}^{\pm}]$ for the centred version.
Since the height increments are centered by the choice of the reference 1-form $f_0$~\eqref{eq:refflow} and since $|f_0|= 1/2$ on all roads, we find
\begin{align}
\mathbf{E}^{\emptyset,\emptyset}_{D^{\delta}, D^{\delta}}\Big[\prod_{i=1}^n (h^{\delta}(a_i)-h^{\delta}(a_{i}^0))\Big]  &=\mathbf{E}^{\emptyset,\emptyset}_{D^{\delta}, D^{\delta}}\Big[ \prod_{i=1}^n \sum_{t=1}^{l_i} (c_{i,t}^{+} -{c_{i,t}^{-}} )\Big]\nonumber \\ 
&= \mathbf{E}^{\emptyset,\emptyset}_{D^{\delta}, D^{\delta}}\Big[ \prod_{i=1}^n \sum_{t=1}^{l_i} (\hat{c}_{i,t}^{+} -\hat{c}_{i,t}^{-} )\Big] \nonumber \\
&=\sum_{t_i=1}^{l_1}\cdots \sum_{t_n=1}^{l_n} \sum_{s\in\{ \pm \}^n}  (-1)^{\#_-(s)} \mathbf{E}\Big[  \prod_{i=1}^n \hat{c}_{i,t_i}^{s_i}\Big], \label{eq:plugdet}
\end{align}
where $\#_-(s)$ is the number of minuses in $s$.

Fix $t_1,\ldots,t_n$ and $s\in \{\pm\}^n$, and let $\hat c_i:=\hat c_{i,t_i}^{s_i}$.
By   \cite[Lemma~21]{Ken00}, the determinant of the inverse Kasteleyn matrix gives correlations of height increments, hence
\begin{align} \label{eq:Kenyon}
\mathbf{E}^{\emptyset,\emptyset}_{D^{\delta}, D^{\delta}}\Big[  \prod_{i=1}^n \hat{c}_i\Big] &=\Big( \prod_{i=1}^n K(b_i,w_i) \Big) \det \hat C =(-1)^n \det \hat C=\det \hat C,
\end{align}
where $\hat C$ is the $n\times n$ matrix given by
\begin{align*}
\hat{C}_{i,j}=\begin{cases} K^{-1}(w_i,b_j) &\text{if } i\neq j, \\
0 &\text{otherwise}.
\end{cases}
\end{align*}
Here we used that the edges of $C^{\delta}$ (roads) corresponding to the corners in $D^{\delta}$ are assigned weight $-1$ in the Kasteleyn weighting as defined in Section \ref{sec:Kasteleyn}.

Let $e_i$ be the edge satisfying $c^{\pm}(e_i)=c_{i,t_i}^{\pm}$, and let $z_i$ be its midpoint. We write $f_{\pm} :=f^D_{\pm}$ and $f^{\delta}:=f^{D_{\delta}}$. Proposition~\ref{cor:fermscale} gives
\begin{align*}
K^{-1}(w_i,b_j)&=-\frac{\delta \i}{\sqrt2 }\big(f_-(z_i,z_j)- \overline{\eta}_{c_i}^2\eta_{c_j}^2 \overline{f_-(z_i,z_j)} +  \overline{\eta}^2_{c_i}f_+(z_i,z_j)  - {\eta^2_{c_j}}\overline{f_+(z_i,z_j)}  +o(1)\big).
\end{align*}
We now expand the determinant from \eqref{eq:Kenyon} as a sum over permutations.
Let us investigate the term in this expansion coming from a fixed permutation~$\iota$, and for simplicity of notation, let us assume that $\iota$ is the cycle $\iota(i)=i+1\ (\text{mod} \ n)$. The case of a general permutation 
will follow in a similar manner. The term under consideration reads
\begin{align}
&\textnormal{sgn}(\iota) \frac{\delta^n}{\sqrt{2}^n} \i^n\prod_{i=1}^n \Big(  f_-(z_i,z_{i+1})+  \overline{\eta}^2_{c_i}f_+(z_i,z_{i+1}) - \nonumber \\ 
&\qquad \qquad \qquad \qquad \overline{\eta}_{c_i}^2\eta_{c_{i+1}}^2 \overline{f_-(z_i,z_{i+1})}- {\eta^2_{c_{i+1}}}\overline{f_+(z_i,z_{i+1})} \Big) +o(\delta^n) \nonumber \\ 
&=\textnormal{sgn}(\iota) \frac{\delta^n}{\sqrt{2}^n}\i^n \prod_{i=1}^n \Big(  f_{-1,1}(z_i,z_{i+1})+  \eta_{c_i}^{-2}f_{1,1}(z_i,z_{i+1}) - \nonumber \\  \label{eq:expandprod} 
&\qquad \qquad \qquad \qquad  {\eta}_{c_i}^{-2}\eta_{c_{i+1}}^2 
{f_{1,-1}(z_i,z_{i+1})}-\eta_{c_{i+1}}^2{f_{-1,-1}(z_i,z_{i+1})} \Big) +o(\delta^n). 
\end{align}
We can now expand the product into a sum of $4^n$ terms. Note that for each corner~$c_i$, the factors ${\eta_{c_i}^2}$ and $\eta_{c_i}^{-2} $ appear in exactly one 
out of $n$ brackets, meaning that each final term contains a multiplicative factor of $\eta_{c_i}^{r_{c_i}}$, where $r_{c_i} \in\{-2,0,2\}$. 

The first important observation is that 
the terms for which there exists $i$ such that $r_{c_i}=0$ cancel out to $o(\delta^n)$ after summing over all sign choices $s\in\{-1,1\}^n$ in \eqref{eq:plugdet}. Indeed, for each such term,
take the smallest $i$ for which $r_{c_i}=0$ and consider the corresponding term assigned in~\eqref{eq:plugdet} to a different sign choice $s'$ which differs from $s$ only at the coordinate $i$. 
By \eqref{eq:expandprod} the two terms differ by $o(\delta^n)$, and the cancellation in \eqref{eq:plugdet} is caused by the fact that $\#_-(s)=-\#_-(s')$.

There are exactly $2^n$ remaining terms indexed by $\epsilon\in\{-1,1\}^n$ that satisfy $r_{c_i}= -2\epsilon_i$ for all~$i$. Note that in the embedding of the square lattice $\delta\mathbb Z^2$,
all corners have length $\delta\sqrt2/2$, and therefore
\[
\eta_{c_i}^{\pm 2}=\sqrt{2}\delta^{-1} dc_i^{(\mp 1)},
\] 
where 
$dc_i^{(1)}:=dc_i$ and $dc_i^{(-1)}:=\overline{dc_i}$. 
Hence, the $\sqrt2$-terms cancel out, and each such term is of the form
\begin{align} \label{eq:sumsigns}
\textnormal{sgn}(\iota) \i^n \Big(\prod_{i=1}^n \epsilon_i\Big) \Big(\prod_{i=1}^n f_{\epsilon_i,\epsilon_{i+1}}(z_i,z_{i+1})\Big) dc_i^{(\epsilon_1)} \cdots dc_n^{(\epsilon_n)} +o(\delta^n).
\end{align}
The term $\prod_{i=1}^n \epsilon_i$ arises as the product of the signs from the expansion of~\eqref{eq:expandprod}.

Since
\[
d(c_{i,t_i}^{+})^{(\epsilon_i)} - d(c_{i,t_i}^{-})^{(\epsilon_i)} = d(z_i^{\delta})^{(\epsilon_i)},
\] 
 keeping the permutation $\iota$ and the signs $\epsilon$ fixed, and summing \eqref{eq:sumsigns} over all $s\in \{-1,1\}^n$, we obtain
\[
\textnormal{sgn}(\iota)\i^n \Big(\prod_{i=1}^n \epsilon_i\Big) \Big(\prod_{i=1}^n f_{\epsilon_i,\epsilon_{i+1}}(z_i,z_{i+1})\Big) d(z_1^{\delta})^{(\epsilon_1)} \cdots d(z_n^{\delta})^{(\epsilon_n)} +o(\delta^n) .
\]
Finally, summing back over all permutations and using that $\gamma^\delta_i\to \gamma_i$ as $\delta\to 0$, we obtain that \eqref{eq:plugdet} is equal to 
\begin{align}
\i^n\sum_{t_i=1}^{l_1}\cdots \sum_{t_n=1}^{l_n} \Big(\sum_{\epsilon \in\{ \pm \}^n}\Big( \prod_{i=1}^n \epsilon_i \Big)\det \big[ f_{\epsilon_i,\epsilon_j} (z_i,z_j)\big]_{1\leq i,j \leq n} d(z_1^{\delta})^{(\epsilon_1)} \cdots d(z_n ^{\delta})^{(\epsilon_n)}+o(\delta^n)\Big) \nonumber\\
= \i^n\sum_{\epsilon \in\{ \pm \}^n}\Big( \prod_{i=1}^n \epsilon_i\Big)  \int_{\gamma_1} \cdots\int_{\gamma_n} \det \big[ f_{\epsilon_i,\epsilon_j} (z_i,z_j)\big]_{1\leq i,j \leq n}  dz_1^{(\epsilon_1)} \cdots dz_n^{(\epsilon_n)}+o(1).\label{eq:detpfaff}
\end{align}
This concludes the proof of Proposition~\ref{prop:intdet}.
\end{proof}

\begin{proof}[Proof of Theorem~\ref{thm:moments}]
We already proved in Proposition~\ref{prop:intdet} that the desired limit exists and is conformally invariant. Hence, it is enough to identify it for the upper half-plane $\mathbb{H}$.
In this case, by Lemma~\ref{lem:confcov} we have an explicit formula
\[
f_{\epsilon_i,\epsilon_j}(z_i,z_j)=  \frac{\mathrm{i}^{\tfrac{\epsilon_j-\epsilon_i}{2}}}{2\pi\big(z_j^{(\epsilon_j)} - z_i^{(\epsilon_j)}\big)},
\]
where $z_i^{(1)}= z_i$ and $z_i^{(-1)}= \overline{z_i}$. Up to conjugation by a diagonal matrix with entries~$\i^{\tfrac{\epsilon_i}2}$, this is the same matrix as in \cite[Lemma~3.1]{Ken01}, and hence  
\[
\det \big[ f_{\epsilon_i,\epsilon_j} (z_i,z_j)\big]_{1\leq i,j \leq n} = \frac{1}{(2\pi)^n}\sum_{\pi \text{ pairing of } \{1,\ldots, n\}} \prod_{\{i,j\}\in \pi} \frac{1}{\big(z_j^{(\epsilon_j)} - z_i^{(\epsilon_i)}\big)^2}.
\]
This means that, after exchanging the order of summations, integrals and products, \eqref{eq:detpfaff} is equal to
\begin{align*}
& \frac{\i^n}{(2\pi)^n } \sum_{\pi \text{ pairing of } \{1,\ldots, n\}} \prod_{\{i,j\}\in \pi}2\Re e\Big[\int_{\gamma_j} \int_{\gamma_i}  \frac{dz_idz_j}{(z_j-z_i)^2} - \frac{d\overline{z_i}dz_j}{(z_j-\overline{z_i})^2} \Big] \\
  &={{\pi}^{-\tfrac n 2}} \sum_{\pi \text{ pairing of } \{1,\ldots, n\}} \prod_{\{i,j\}\in \pi} \frac1{2\pi} \ln\Bigg|\frac{(u_j-u_i)(u_j^0-u_i^0)(u_j^0-\overline{u_i})(u_j-\overline{u_i^0})}{(u_j^0-u_i)(u_j-u_i^0)(u_j-\overline{u_i})(u_j^0-\overline{u_i^0})}\Bigg|.
\end{align*}

Note that the terms in the product above converge to $G_{\mathbb H}(u_i,u_j)$ as $u_i^0$ and $u_j^0$ get close to $\partial \mathbb H$. 
This together with \eqref{eq:boundaryapprox} implies that, up to the explicit multiplicative constant, the moments have the same scaling limit as in~\cite{Ken01}, which ends the proof.
\end{proof}

\subsection{Convergence of $h^{\delta}$ as a random distribution}\label{sec:convergence h}
Recall that for $a\in D$ we write $h^\delta(a)$ for the evaluation of the nesting field at a face $u^\delta=u^\delta(a)$ of $D^\delta$ containing $a$. (Here we talk only about the graph $D^\delta$ where the nesting field is defined, and not about $C_{D^\delta}$ which is used as a intermediate tool to prove this convergence.)
For a test function $g:D\to \mathbb R$, define
\begin{align} \label{eq:defh}
h^{\delta}(g) :=\int_D g(a) h^\delta(a) da.
\end{align}

\begin{theorem} \label{thm:fieldmoments}
Let $h_D$ be the GFF in $D$ with zero boundary conditions, and let $g_1,\ldots,g_k: D\to \mathbb R$ be continuous and bounded test functions. Then, for $l_1,\ldots,l_k \in \mathbb{N}$,
\[
\lim_{\delta\rightarrow0}\mathbf{E}^{\emptyset,\emptyset}_{D^{\delta}, D^{\delta}} \Big[ \prod_{i=1}^k h^{\delta}(g_i)^{l_i}\Big]=\mathbf{E}\Big[ \prod_{i=1}^k (\tfrac{1}{\sqrt{\pi}} h_D(g_i))^{l_i} \Big],
\]
\end{theorem}

\begin{proof}
We first note that if $\sum_{i=1}^k l_i$ is odd, then the corresponding moments of $h^\delta$ and $h$ vanish and there is nothing to prove. Moreover,
to simplify notation, we only consider moments $\mathbf{E} [ h^{\delta}(g)^{l}]$ of one test function $g$ for $l$ even. The general case follows in a similar way. 
To this end, we fix $l\geq2$, and define 
\[
D^l_\delta := \{(a_1,\ldots,a_l)\in D^l: |a_i-a_j|<\delta \textnormal{ for some } i\neq j \}.
\]
Then by Lemma~\ref{lem:relevantcluster} and \eqref{eq:logmoment} we have 
\begin{align*}
\int_{D} \cdots \int_D \mathbf{E}^{\emptyset,\emptyset}_{D^{\delta}, D^{\delta}} \Big[\prod_{i=1}^l g(a_i)h^{\delta}(a_i)\Big] \mathbf 1_{(a_1,\ldots,a_l)\in D^l_\delta}da_1\cdots da_l &\leq C \|g \|^l_{\infty} (\log\tfrac1\delta)^{lM} \pmb\lambda^{2l}(D^l_\delta) \\
&\leq C' \|g \|^l_{\infty}
\pmb\lambda^{2}(D)^{l-1}(\log\tfrac1\delta)^{l M}\delta^2
\end{align*}
for some constants $C,C'$ and $M$ that depend on $l$, where $\pmb\lambda^{2l}$ is the $2l$-dimensional Lebesgue measure. 
Note that the right-hand side tends to zero as $\delta\to 0$.
The function 
\[
(a_1,\ldots,a_l) \mapsto |\log(\min_{i\neq j} |a_i-a_j|)|^{lM}
\] is integrable over $D^l$, and hence by dominated convergence, Lemma~\ref{lem:relevantcluster} and~\eqref{eq:logmoment} again, we have
\begin{align*}
\lim_{\delta \to 0}\mathbf{E}^{\emptyset,\emptyset}_{D^{\delta}, D^{\delta}} \big[ h^{\delta}(g)^{l}\big] &= \lim_{\delta \to 0} \int_{D} \cdots \int_D\mathbf{E}^{\emptyset,\emptyset}_{D^{\delta}, D^{\delta}} \Big[\prod_{i=1}^l g(a_i)h^{\delta}(a_i)\Big] da_1\cdots da_l  \\
& =\lim_{\delta \to 0} \int_{D} \cdots \int_D\mathbf{E}^{\emptyset,\emptyset}_{D^{\delta}, D^{\delta}}\Big[\prod_{i=1}^l g(a_i)h^{\delta}(a_i)\Big] \mathbf 1_{(a_1,\ldots,a_l)\in D^l\setminus D^l_\delta}da_1\cdots da_l  \\
& = \int_{D} \cdots \int_D\Big(\prod_{i=1}^l g(a_i)\Big) \lim_{\delta \to 0}\mathbf{E}^{\emptyset,\emptyset}_{D^{\delta}, D^{\delta}} \Big[\prod_{i=1}^n h^{\delta}(a_i)\Big] \mathbf 1_{(a_1,\ldots,a_l)\in D^l\setminus D^l_\delta}da_1\cdots da_l  \\
&= \int_D \cdots\int_D \Big(\prod_{i=1}^l g(a_i)\Big) \sum_{\pi \text{ pairing }} \prod_{\{ i,j\} \in \pi} \frac{1}{\pi} G_D (a_i,a_j)da_1\cdots da_l \\
&= \mathbf{E} [ (\tfrac{1}{\sqrt{\pi}}h_D(g))^{l} ],
\end{align*}
where the second last equality follows from Theorem~\ref{thm:moments}.
\end{proof}

\begin{remark}
We note that the same convergence as in Theorem~\ref{thm:fieldmoments} holds if the height function is considered as a function on all faces of $C_{G^\delta}$ and not only on the faces of $G^\delta$.
\end{remark}

We are now ready to conclude the proof the main theorem of this section.
\begin{proof}[Proof of Theorem~\ref{thm:NFconv}]
By Theorem~\ref{thm:fieldmoments}, all moments of $h^{\delta}$ converge to the corresponding moments of $\tfrac{1}{\sqrt{\pi}}h_D$. Since $h_D$ is a Gaussian process, its moments identify its law uniquely. Since convergence of the second moment implies tightness, we conclude that $h^{\delta}$ tends to $ \tfrac{1}{\sqrt{\pi}}h_D$ in distribution as $\delta$ tends to $ 0$ in the space of 
generalized functions acting on continuous test functions with compact support.
\end{proof}

\section{Further preliminaries}\label{sec:prem}
In this section, we recall some background on the continuum side.

In this section, we recall some background on the continuum side, notably on the Gaussian free field, the local sets and the two-valued sets. Throughout, let $D\subsetneq\Cb$ be a simply connected domain whose boundary is a Jordan curve.

The Schramm-Loewner evolution (SLE) was introduced by Schramm in \cite{MR1776084}. It is a family of non self-crossing random curves which depend on a parameter $\kappa> 0$. 
For many discrete models,  free or wired/monochromatic boundary conditions force the interfaces to take the form of loops. 
The loop interfaces are conjectured (and sometimes proved) to converge to a conformal loop ensemble (CLE) in the continuum,  which is a random collection of loops contained in $D$ that do not cross each other. 
The family of CLE  was  introduced by Sheffield in \cite{MR2494457} and further studied by Sheffield and Werner in  \cite{MR2979861}. It depends on a parameter $\kappa\in(8/3,8)$ and can be constructed using variants of SLE$_\kappa$.

In \cite{MR2486487, MR3101840}, Schramm and Sheffield made the important discovery that level lines of the discrete Gaussian free field (GFF) converge in the scaling limit to  SLE$_4$ curves, and that the limiting SLE$_4$ curves are coupled with the continuum GFF as its \emph{local sets} (i.e., a set with a certain spatial Markov property, see Definition~\ref{def:local_set}).  
More generally, the theory of local sets developed in \cite{MR3101840} allows one to couple SLE$_\kappa$ with the GFF for all $\kappa\in(0,8)$.
The coupling between SLE$_\kappa$ and GFF 
was further developed in \cite{MR2525778,MR3477777,MR3502592,MR3548530,MR3719057} (also, see references therein). 

In this work, we are only concerned with the case $\kappa=4$. 
It was shown in \cite{MR3101840} that SLE$_4$-type curves are coupled with the GFF with a height gap $2\lambda$   in such a way that they are local sets of the GFF with boundary values respectively $a-\lambda$ and $a+\lambda$ on the left- and right-hand sides of the curve. A crucial property shown in  \cite{MR3101840} is that such SLE$_4$-type curves are deterministic functions of the GFF. We call these curves level lines, to keep the same terminology as in the discrete. The value $a\in\Rb$ is called the \emph{height} of the level line. 
The coupling between SLE$_4$ and GFF was extended to CLE$_4$ and GFF by Miller and Sheffield \cite{MS} (a more general coupling between CLE$_\kappa$ and GFF for all $\kappa\in(0,8)$ was established in \cite{MR3708206}; a proof for the case $\kappa=4$ was also provided in \cite{MR3936643}).
 
Let us fix some notation that will be used throughout this work.
For any simply connected domain $U$, we say that its boundary $\partial U$ is a \emph{contour}.
If $\gamma$ is a simple loop, then let $O(\gamma)$ denote the domain encircled by $\gamma$, which is equal to the unique bounded connected component of $\Cb\setminus \gamma$. 
Let $\ol O(\gamma)$ be the closure of $O(\gamma)$.
Every simple loop is a contour, but a contour need not be a loop or a curve. 
Let $h$ be a zero boundary GFF in $D$. 
 For every simply connected domain $U\subseteq D$, let $h|_{U}$ denote the restriction of $h$ to the domain $U$. 
 If $h|_{U}$ is equal to a GFF in $U$ with constant boundary conditions, say equal to $c$, then let $h^0|_{U}$ be the zero boundary GFF so that $h|_{U}$ is equal to $h^0|_{U}$ plus~$c$. This constant $c$ is also called the \emph{boundary value} of $U$, or the \emph{boundary value} of $\partial U$. Let $\Gamma$ denote a collection of simple loops which do not cross each other. Let $\gask(\Gamma)$ denote the \emph{gasket} of $\Gamma$, which is equal to $\ol D\setminus \cup_{\gamma\in \Gamma} O(\gamma)$.
Given a connected set $A\subseteq\ol D$ such that $\partial D\subseteq A$, let $\lp(A)$ denote the collection of outer boundaries of the connected components of $D\setminus A$.

The Miller-Sheffield coupling between the GFF and CLE$_4$ states that $h$ a.s.\ uniquely determines a random collection $\Gamma$ of simple loops which do not cross each other and satisfy the following property (see Fig.~\ref{fig:CLE4}, left): conditionally on $\gask(\Gamma)$, for each loop $\gamma \in \Gamma$, there exists $\epsilon(\gamma)\in\{-1, 1\}$ such that $h|_{O(\gamma)}$ is equal to $\epsilon(\gamma) 2\lambda$ plus a zero-boundary GFF. In addition, the fields $h|_{O(\gamma)}$ for different $\gamma$'s are (conditionally) independent of each other.
In other words, $\gask(\Gamma)$ is a {\em local set} of $h$ with boundary values in $\{-2\lambda, 2\lambda\}$.
It turns out that $\Gamma$ has the law of a CLE$_4$. In addition,  $\gask(\Gamma)$ carries no mass of the GFF: for all test function $f$ on $D$, we have
\begin{align}\label{eq:thin}
\int_D f(x) h(x) dx = \sum_{\gamma\in\Gamma} \int_{O(\gamma)} f(x) h|_{O(\gamma)}(x) dx.
\end{align}
Each loop $\gamma$ in CLE$_4$ is a level line (we also call it a level loop) of the GFF with boundary value $\epsilon(\gamma) 2\lambda$ on the inner side of the loop and $0$ on the outer side of the loop (so it is at height $\epsilon(\gamma) \lambda$).

\begin{figure}
\centering
\includegraphics[width=0.4\textwidth, page=1]{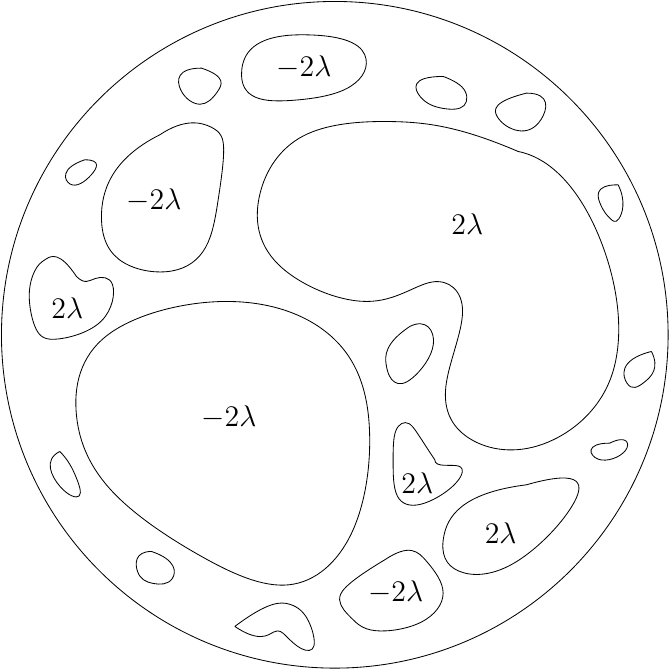}\qquad
\includegraphics[width=0.4\textwidth, page=2]{CLE4}
\caption{\textbf{Left:}  A sketch of CLE$_4$ coupled with the GFF. The loops have boundary values $-2\lambda$ or $2\lambda$. \textbf{Right:} We depict a few layers of the nested CLE$_4$ coupled with the same GFF. We mark in red the outermost loops that have boundary values $-4\lambda$ or $4\lambda$, which belong to $\lp_{-4\lambda, 4\lambda}$.}
\label{fig:CLE4}
\end{figure}

It is also natural to consider level loops of $h$ at other heights than those of CLE$_4$. For example, the previous coupling can be extended to the nested CLE$_4$ (by sampling the CLE$_4$ coupled to $h^0|_{O(\gamma)}$ for each $\gamma\in\Gamma$), so that the further layers of CLE$_4$ loops are at heights $(2k+1)\lambda$ for $k\in\Zb$.
For $a\in (-\lambda, \lambda)$, the outermost  level loops of $h$ at height $a$  are given by boundary conformal loop ensembles (BCLE)  \cite{MR3708206}, and one can then also consider nested versions of BCLE to obtain level loops of $h$ at a continuum range of heights.

The gaskets of CLEs and BCLEs belong to a particular class of local sets called \emph{two-valued sets} introduced by Aru, Sep\'ulveda and Werner in \cite{MR3936643}: a two-valued set is a \emph{thin local set} (a terminology in \cite{MR4010952} meaning that the local set carries no mass of the GFF, described by~\eqref{eq:thin}) with two boundary values in $\{-a, b\}$, denoted by $\Ab_{-a,b}$. 
For example, the gasket of CLE$_4$ is equal to $\Ab_{-2\lambda, 2\lambda}$, and the gaskets of BCLEs  correspond to $\Ab_{-a,b}$ with $a+b=2\lambda$. 
It was shown in  \cite{MR3936643}  that the sets $\Ab_{-a,b}$ exist for $a,b>0$ with $a+b\ge 2\lambda$, and are a.s.\ unique and determined by $h$. 
Let us use $\lp_{-a,b}$ to denote $\lp(\Ab_{-a,b})$. Throughout, we denote by $\lp^+_{-a,b}$ (resp.\ $\lp^-_{-a,b}$) the set of loops in $\lp_{-a,b}$ with boundary value $b$ (resp.\ $-a$). We will also use notations like CLE$_4(h)$ and $\lp_{-a,b}(h)$ to represent these sets coupled to $h$ (especially when there are different GFFs involved).

The loops in $\lp_{-a,b}$ are composed of SLE$_4$-type curves which are level lines of $h$, hence are a.s.\ simple and do not cross each other (but can intersect each other). 
The law of $\lp_{-a,b}$ is invariant under all conformal automorphisms from $D$ onto itself, since $h$ is invariant under those conformal maps.
The geometric properties of the loops in $\lp_{-a,b}$ are well understood (see e.g.~\cite{MR3936643,MR3827968,SSV} and Lemma~\ref{lem:adjacency}).

Let us now give a simple and intuitive explanation of the two-valued sets, and postpone more details to the next subsection.
As pointed out in  \cite{MR3936643}, $\Ab_{-a,b}$ is a 2D analogue for GFF of the stopping time of a 1D Brownian motion upon exiting $[-a,b]$, and is intuitively the set of points that are connected to the boundary by a path on which the values of $h$ remain in $[-a,b]$.
Let us illustrate this by the following construction of $\Ab_{-2n\lambda, 2n\lambda}$ via iterated CLE$_4$s (see Fig.~\ref{fig:CLE4}, right). 
For each point $z\in D$, the boundary values of the successive loops that encircle $z$ in the nested CLE$_4$ perform a symmetric random walk with steps  $\pm 2\lambda$. 
The loops in $\lp_{-2n\lambda, 2n\lambda}$ correspond to the first time that we obtain a nested CLE$_4$ loop with boundary value equal to $- 2n \lambda$ or $2n\lambda$. 

\medbreak

Let us give more details on Gaussian free field, local sets and two-valued sets. Here, we look at a GFF in the unit disk $\Ub$. For any other simply connected domain $D$, one can simply map $D$ conformally onto $\Ub$.
Let $\Gamma$ be the space of all closed nonempty subsets of $\overline \Ub$. We view $\Gamma$ as a metric space, endowed by the Hausdorff metric induced by the Euclidean distance. Note that $\Gamma$ is naturally equipped with the Borel $\sigma$-algebra on $\Gamma$ induced by this metric. Given $A \in \Gamma$, let $A_\delta$ denote the closure of the $\delta$-neighborhood of $A$ in $\Ub$. Let $\Ac_\delta$ be the smallest $\sigma$-algebra in which $A$ and the restriction of $h$ to the interior of $A_\delta$ are measurable. Let 
\[
\Ac:=\bigcap_{\delta\in\Qb, \delta>0} \Ac_\delta.\] Intuitively, this is the smallest $\sigma$-algebra in which $A$ and the values of $h$ in an infinitesimal neighborhood of $A$ are measurable.

\begin{definition}[Local set \cite{MR3101840}]\label{def:local_set}
Let $h$ be a GFF in $\Ub$. 
We say that a random set $A$ is a {\em local set} of $h$ if $A$ is a closed subset of $\overline\Ub$ and one can write $h=h_A + h^A$, where 
\begin{itemize}
\item $h_A$ is an $\Ac$-measurable random distribution which is a.s.\ haromonic on $\Ub \setminus A$.
\item conditionally on $\Ac$, $h^A$ is a random distribution which is independent of $\Ac$. It is a.s.\ zero on $A$ and equal to an independent zero boundary GFF in each connected component of $\Ub\setminus A$.
\end{itemize}
\end{definition}

Two-valued sets  were introduced by Aru, Sep\'ulveda and Werner in \cite{MR3936643}. More precisely, they denote thin local sets with two prescribed boundary values. 
In Section~\ref{subsec2:background}, we have mentioned the examples of  CLE$_4$ (whose gasket is a thin local set of a GFF with two boundary values in $\{-2\lambda, 2\lambda\}$) and BCLE$_4(-1)$ (whose gasket is a thin local set of a GFF with two boundary values in $\{-\lambda, \lambda\}$).

In \cite{MR3936643}, the authors first defined the more general family of \emph{bounded type thin local sets} (denoted by BTLS), as follows.
\begin{definition}[Bounded type thin local sets,  \cite{MR3936643}]\label{def:btls}
Let $h$ be a GFF in $D$. Let $A$ be a relatively closed subset of $D$.  For $K>0$, we say that $A$ is a $K$-BTLS of $h$ if 
\begin{enumerate}
\item (boundedness) $A$ is a local set of $h$ such that $|h_A(x)|\le K$ for all $x\in D\setminus A$.
\item (thinness) for any smooth function $f$, we have $(h_A, f)=\int_{D\setminus A} h_A(x) f(x) dx$.
\end{enumerate}
\end{definition}

It was shown in \cite{MR3936643} that a BTLS must be connected to the boundary of the domain.

\begin{lemma}[Proposition 4, \cite{MR3936643}]\label{lem:btls_connected}
If $A$ is a BTLS, then $A \cup \partial D$ is a.s.\ connected.
\end{lemma}

A two-valued set is defined to be a BTLS $A$ such that $h_A\in\{-a,b\}$ for $a, b>0$. 
The family of two-valued sets satisfies the properties of the following lemma.

\begin{lemma}[Proposition 2 in \cite{MR3936643}]\label{lem:tvs}
Let $-a<0<b$.
\begin{itemize}
\item When $a+b<2\lambda$, it is not possible to construct a BTLS $A$ such that $h_A \in\{-a, b\}$ a.s.
\item When $a+b\ge 2\lambda$, there is a unique BTLS $A$ coupled with $h$ such that $h_A\in\{-a,b\}$  a.s. We denote this set $A$ by $\Ab_{-a,b}$. 
\item If $[a,b]\subseteq [a', b']$, then  $\Ab_{-a,b} \subseteq \Ab_{-a', b'}$ a.s.
\end{itemize}
\end{lemma}

This lemma shows that two-valued sets are deterministic functions of the GFF $h$ (when they exist), and this property will be instrumental in our proof.

When $a+b=2\lambda$, the set $\lp_{-a,b}$ is equal to BCLE$_4(\rho)$ (where $\rho=-a/\lambda$) and can be constructed using the branching  SLE$_4(\rho, -2-\rho)$ process (\cite{MR3708206,MR3936643}). The loops in $\lp^+_{-a,b}$ (resp.\ $\lp^-_{-a,b}$) correspond to the loops traced in the clockwise (resp.\ counterclockwise) direction by the branching SLE$_4(\rho, -2-\rho)$.
Properties of such SLE processes directly imply the following lemma.

\begin{lemma}[\cite{MR3708206,MR3936643}]
\label{lem:2lambda_intersect}
If $a+b=2\lambda$, every loop  in $ \lp_{-a,b}$ intersects $\partial D$. The loops in $\Lc^-_{-a,b}$ are equal to the outer boundaries of the connected components of $D \setminus \ol{\cup_{\gamma\in \Lc^+_{-a,b}} O(\gamma)}$.
\end{lemma}

For other values of $a,b$, $\Ab_{-a,b}$ is constructed by iterating the branching  SLE$_4(\rho, -2-\rho)$ processes.
Using properties of the SLE$_4(\rho, -2-\rho)$ processes, \cite{MR3827968} has deduced  the following intersecting behavior of the loops in $\lp_{-a,b}$, which will be useful for us later.

\begin{lemma}[Intersecting behavior of the loops \cite{MR3827968}]\label{lem:adjacency}

\begin{enumerate}
\item There exists a loop in $ \lp^+_{-a,b}$ (resp.\ $\lp^-_{-a,b}$) which intersects $\partial D$ if and only if $b<2\lambda$ (resp.\ $a<2\lambda$).
\item If $a+b<4\lambda$, then one can connect any two loops $\eta_1$ and $\eta_2$  in $\lp_{-a,b}$ by a finite sequence of loops $(\gamma_1, \dots, \gamma_n)$ so that $\gamma_1=\eta_1$, $\gamma_n=\eta_2$ and $\gamma_{k+1}$ intersects $\gamma_k$ for each $1\le k \le n-1$.
Only loops with different boundary values can intersect each other.
\end{enumerate}
\end{lemma}

We will also use the following lemmas to identify uniquely the law of the limiting interfaces in Sect.~\ref{sec:mainres}.

\begin{lemma}[Lemma 3.8, \cite{MR3827968}]
\label{lem:loop_intersect}
Let $a, b>0$ with $a+b> 2\lambda$. Then almost surely, a loop $\ell$ of $\Lc_{-a,b}$ labelled $-a$ touches the boundary if and only if $a<2\lambda$ and $\ell$ is a loop of $\Lc_{-a, -a+2\lambda}$ labelled $-a$.
Moreover, the loops of $\Lc_{-a,b}$ which do not touch the boundary and are surrounded by a loop $\gamma\in \Lc_{-a, -a+2\lambda}$ labelled $-a+2\lambda$ are exactly the loops of
$ \Lc_{-2\lambda, a+b-2\lambda}(h^0|_{O(\gamma)})$.
\end{lemma}

\begin{lemma} \label{lem:nestedident}
Let $A_1,A_2,\ldots$ be an increasing sequence of thin local set of a GFF $h$ in a domain $D$ with $A_1=A_{-\sqrt 2 \lambda, \sqrt 2 \lambda}$, and such that for each $k$ and each $\ell\in \mathcal L(A_k)$ 
with boundary value $m\sqrt 2 \lambda$, each loop in $\mathcal L(A_{k+1})$ encircled by $\ell$ has boundary value either $(m-1)\sqrt 2 \lambda$ or $(m+1)\sqrt 2 \lambda$.
Then, for each $k$ and each $\ell\in \mathcal L(A_k)$, the loops in $\mathcal L(A_{k+1})$ encircled by $\ell$ are exactly $\mathcal L_{-\sqrt 2 \lambda, \sqrt 2 \lambda}(h^0_{O(\ell)})$.
\end{lemma}
\begin{proof}
Suppose that $\ell\in \mathcal L(A_k)$ has boundary value $m\sqrt{2}\lambda$.
Since  $A_1,A_2,\ldots$ is an increasing sequence, every loop $\gamma\in \Lc(A_{k+1})$ is either encircled by $\ell$ or $O(\gamma) \cap O(\ell)=\emptyset$.
Since $A_{k+1}$ is a thin local set of $h$, we have  for any smooth function $f$
\begin{align*}
\int_{O(\ell)} h_{O(\ell)}(x) f(x) dx =\int_{O(\ell)} h(x) f(x) dx = \sum_{\gamma\in \Lc (A_{k+1}), O(\gamma)\subseteq O(\ell) } \int_{O(\gamma)} h_{O(\gamma)} (x) f(x) dx.
\end{align*}
Since $\ell$ has boundary value $m\sqrt{2}\lambda$ and each $\gamma\in \Lc (A_{k+1})$ encircled by $\ell$ has boundary value either $(m-1)\sqrt 2 \lambda$ or $(m+1)\sqrt 2 \lambda$, we can conclude the proof.
\end{proof}

\begin{lemma}\label{lem:CLEiteration}
Consider the following collection of loops defined iteratively. 
\begin{itemize}
\item Let $\Lc_0(h)$ be the collection of loops resulting from replacing each $\ell \in \Lc^{+}_{-\sqrt 2\lambda, \sqrt 2\lambda}(h)$ (resp.~$\ell \in \Lc^{-}_{-\sqrt 2\lambda, \sqrt 2\lambda}(h)$) by an independent (conditionally on $\ell$) $\Lc_{- \sqrt 2\lambda,  (2 -\sqrt 2)\lambda}(h|_{O(\ell)})$
(resp.~$\Lc_{-(2 -\sqrt 2)\lambda,\sqrt 2\lambda}(h|_{O(\ell)}  )$). 
\item Given $\Lc_k(h)$, define $\Lc_{k+1}(h)$ by replacing each $\ell \in \Lc_k(h)$ with boundary value $0$ by an independent (conditionally on $\ell$) copy of $\Lc_0(h|_{O(\ell)})$.
\end{itemize}
Then, $\liminf_{k\to \infty}\Lc_k(h)=\limsup_{k\to \infty}\Lc_k(h)=\mathcal L_{- 2\lambda, 2 \lambda}(h) =\textnormal{CLE}_4(h)$.
\end{lemma}
\begin{proof}

For each $k$, $\gask(\Lc_k(h))$ is clearly a thin local set of $h$ with boundary values in $\{ -2\lambda, 0,2\lambda\}$. It remains to prove that $\gask(\lim_{k\to \infty} \Lc_k(h))$ is a thin local set of $h$ with
boundary values in $\{ -2\lambda, 2\lambda\}$. For this purpose, it is enough to prove that $\lim_{k\to \infty} \Lc_k(h)$ a.s.\ does not contain any loop with boundary value $0$. 

Let $D$ be the domain on which $h$ is defined. For $z\in D$, if $z$ is encircled by a loop in $\Lc_k(h)$ with boundary value in $\{-2\lambda, 2\lambda\}$ for some $k\ge 0$, then $z$ cannot be encircled by a loop in $\lim_{k\to \infty}  \Lc_k(h)$ with boundary value $0$. Let $E(z)$ be the event that $z$ is encircled by a loop  $\ell_k\in \Lc_k(h)$ with boundary value $0$ for every $k\ge 0$.
Then $z$ is encircled by a loop $\ell \in \lim_{k\to \infty}  \Lc_k(h)$ if and only if $E(z)$ occurs. On this event, $\ell$ is a.s.\ encircled by $\ell_k$ for every $k\ge 0$.

On $E(z)$, for $k\ge 1$, let  $r_k(z)$ be the conformal radius of $\ell_{k-1}$ seen from $z$. Let $r_0(z)$ be the conformal radius of $\partial D$ seen from $z$. Then for $k\ge 1$, conditionally on $E(z)$, the random variables $r_k(z)/r_{k-1}(z)$ are i.i.d.\ and their law does not depend on $z$ (due to conformal invariance of $\Lc_0(h)$). Moreover,  $r_k(z)/r_{k-1}(z)$ is a.s.\ strictly less than $1$, since $\gask(\Lc_0(h))$ is a.s.\ non-empty. It follows that $r_k(z)\to 0$ as $k\to \infty$ a.s., hence $\ell$ a.s.\ has conformal radius $0$, which is impossible. Therefore, $z$  is a.s.\ not encircled by a loop in $\lim_{k\to \infty}  \Lc_k(h)$ with boundary value $0$. Since this is true for all $z$, we have proved that $\lim_{k\to \infty} \Lc_k(h)$ a.s.\ does not contain any loop with boundary value $0$.
\end{proof}

\section{Scaling limit of the double random current clusters}
\label{sec:bah}
In this section, we identify the scaling limit of the double random current clusters with free and wired boundary conditions. 
More precisely, we prove Theorems~\ref{thm:Qh} and~\ref{thm:ABh} which imply Theorems~\ref{thm:cvg_free_drc_v1} and~\ref{thm:cvg_wired_drc_v1}. 
As we have pointed out at the end of Section~\ref{intro:drc},  Theorems~\ref{thm:Qh} and~\ref{thm:ABh} contain more information than Theorems~\ref{thm:cvg_free_drc_v1} and~\ref{thm:cvg_wired_drc_v1}.

Our proof crucially relies on the height function as defined in the master coupling in Theorem~\ref{thm:mastercoupling} which  
satisfies a strong form of spatial Markov property at the inner boundaries of the double random current clusters, namely one has an independent height function (which converges to a GFF) inside each domain encircled by the inner boundary of a cluster. The boundary values $\sqrt{2}\lambda$ and $2\sqrt{2}\lambda$ at the inner boundaries of the clusters come from the discrete height function (in the discrete, 
the height changes by $\pm 1$ or $\pm 1/2$ between neighbouring sites and faces but the limiting field is $({2\sqrt 2\lambda})^{-1}$ times the GFF, hence the values of the continuum field on the scaling limit of such inner boundaries are multiples of $\sqrt 2\lambda$).
For example, the scaling limit of the inner boundaries of the outermost cluster in a double random current model with wired boundary conditions follow directly from this spatial Markov property and the characterization of two-valued sets (Lemma~\ref{lem:tvs}) \cite{MR3936643}.

In contrast, the discrete height function does \emph{not} have this form of spatial Markov property at the outer boundaries of the clusters. However, we establish this spatial Markov property in the continuum limit, using additional information on the geometric properties of these loops and their interaction with other interfaces of the primal and dual models coupled through Theorem~\ref{thm:mastercoupling}. 
More precisely, we show that the outer boundaries of the clusters in a free boundary double random current model converge to the CLE$_4$ coupled with the limiting GFF, so that each limiting loop has boundary value $-2\lambda$ or $2\lambda$. The value $2\lambda$ cannot be found in the height function of the discrete model, but only appears in the continuum limit. This is the same value as the height gap at the two sides of a level line, identified in \cite{MR2486487}.

\smallskip

Throughout, let $D$ be a Jordan domain. Let $U_1$ and $U_2$ be two open and simply connected sets.
We say that two contours $\partial U_1$ and $\partial U_2$ cross each other if $U_1 \not\subseteq U_2$, $U_2 \not\subseteq U_1$ and $U_1\cap U_2\not=\emptyset$. We say that a contour $\partial U_1$ encircles another contour $\partial U_2$ if $U_2 \subseteq U_1$, and we say $\partial U_1$ strictly encircles $\partial U_2$ if $U_2 \subsetneq U_1$.

\subsection{Main results} \label{sec:mainres}

In this section, we state the main Theorems~\ref{thm:Qh} and~\ref{thm:ABh}, which can be seen as enhanced versions of Theorems~\ref{thm:cvg_free_drc_v1} and~\ref{thm:cvg_wired_drc_v1} presented in the introduction.

Let $D\subsetneq \mathbb C$ be a Jordan domain. Recall that we say that simply connected graphs $D^\delta \subset \delta \mathbb Z^2$ \emph{approximate} $D$ if 
$d(\partial D^\delta , \partial D)\to 0$ as $\delta \to 0 $, where $d$ is as in \eqref{eq:curvedist}.
We consider a critical double random current $\mathbf n^\delta$ on $D^\delta$ with free boundary conditions, and the dual double random current $(\n^\dagger)^\delta$ on $(D^\delta)^\dagger$ with wired boundary conditions, 
coupled together as in Theorem~\ref{thm:mastercoupling}. Let $\mathbb P_{D^\delta}$ be this coupling that encodes also the joint height function $H^\delta$ composed of the nesting field $h^\delta$ of $\n^\delta$, and the nesting field $(h^\dagger)^\delta$ of $(\n^\dagger)^\delta$.
The following collections of loops will be relevant in our proofs.

\begin{itemize}
\item $ Q_0^\delta$ is the collection of loops in the inner boundary of the cluster of the ghost vertex of~$(\n^\dagger)^\delta$.
We proceed inductively. Having defined $ Q_k^\delta$, we define $ Q_{k+1}^\delta$ in the following way. Recall that by property~\eqref{itm:bc} of the master coupling from Theorem~\ref{thm:mastercoupling},
if $k$ is even, then in each loop $\ell$ of $ Q_k^\delta$, $\n^\delta$ restricted to the domain encircled by $\ell$ has wired boundary conditions. 
We modify the current by setting $\n^\delta_e=2$ (the only important property is that the value is nonzero and even) for every primal edge $e$ whose both endpoints are adjacent to $\ell$ from the inside.
We denote this modified current restricted to $\ell$ by $\n^\delta_\ell$.
We then define $ Q_{k+1}^\delta(\ell)$ as the union of all the loops in the inner boundary of the external most (touching $\ell$) cluster of $\n^\delta_\ell$ (see Fig.~\ref{fig:nestedMarkov} for an illustration). 
Finally we set $ Q_{k+1}^\delta=\bigcup_{\ell \in Q_k^\delta}  Q_{k+1}^\delta(\ell)$. If $k$ is odd, then we proceed analogously with $\n^\delta$ replaced by $(\n^\dagger)^\delta$, and the primal graph replaced by the dual graph. In particular, the loops in $Q_k$ are on the primal (resp.~dual) lattice for $k$ even (resp.~odd). We define $Q^\delta=\bigcup_{k=0}^{\infty} Q_k^\delta$.
\item $B_k^\delta$, for $k$ even, is the collection of outer boundaries of the clusters of $\n^\delta$ that touch a loop of $Q_k^\delta$ from the inside.
Moreover, for each loop $\ell \in B_k^\delta$, let $\mathcal C(\ell)$ be the cluster of $\n^\delta$ with outer boundary $\ell$, and let $A_k^\delta(\ell)$ be the collection of loops in the inner boundary of $\mathcal C(\ell)$, and $A_k^\delta:=\bigcup_{\ell \in B_k^\delta} A_k^\delta(\ell)$. The collection of loops $B_k^\delta$, for $k$ odd, is defined in the same way but with $\n^\delta$ exchanged for $(\n^\dagger)^\delta$.
Finally let $B^\delta=\bigcup_{k=0}^\infty B_k^\delta$ and $A^\delta=\bigcup_{k=0}^\infty A_k^\delta$. 
\end{itemize}

\begin{remark} \label{rem:looptracing}
Note that $A_k^\delta(\ell)\subset Q_{k+1}^\delta(\ell)$, and hence $A_k^\delta\subset Q_{k+1}^\delta$. 
Moreover, every loop in $ Q_{k+1}^\delta(\ell) \setminus A_k^\delta(\ell) $ traces pieces of loops in $B_k^\delta$ that touch $\ell$ and/or the loop $\ell$ itself (see Fig.~\ref{fig:nestedMarkov} for an illustration).
We also note that the outermost loops both in $B^\delta$ and $A^\delta$ can be of arbitrary level, i.e. belong to $B_k^{{\delta}}$ and $A_k^{{\delta}}$ for any $k$.  
\end{remark}

For $k$ odd (resp.~even) and each $\gamma\in Q_k^\delta(\ell)$, we say that $\gamma$ is the boundary of an {\em odd hole} if 
$\n^\delta_\ell$ (resp.~$(\n^\dagger)^\delta_\ell$) is odd around every face encircled by $\gamma$ (see definition in Section~\ref{subsec:2cvg_gff}). Otherwise we say that $\gamma$ is the boundary of an {\em even hole}.
We define $ c^\delta(\ell)=1$ (resp.\ $ c^\delta(\ell)=-1$) if $\ell$ is the boundary of an odd (resp.\ even) hole.
Note that every loop in $ Q_{k+1}^\delta(\ell) \setminus A_k^\delta(\ell) $ is the boundary of an even hole in by construction (since we modified the current by adding edges with value $2$).
Moreover, for each loop $\ell \in B^\delta$, let $\epsilon^\delta(\ell)$ be the label of the cluster $\mathcal C(\ell)$ of $\n^\delta$ with outer boundary $\ell$. The label is defined by the coupling with the nesting field $h^\delta$ as in Theorem~\ref{thm:mastercoupling}.

We will prove the following theorems which clearly imply Theorem~\ref{thm:cvg_free_drc_v1} and Theorem~\ref{thm:cvg_wired_drc_v1}.
\begin{theorem}
\label{thm:Qh}
Let $D$ and $D^\delta$ be as above, and such that $\partial D$ is $C^1$. Let $\coin^\delta_{{\mathfrak g}}$ be the label of the cluster of the boundary in $(\n^\dagger)^{\delta}$. Then, as $\delta\to 0$, the family 
$(H^\delta, Q^\delta, c^\delta, \coin^\delta_{{\mathfrak g}})$
defined above converges in distribution to a limit $(\tfrac{1}{\sqrt \pi}h, Q, c,\coin_{{\mathfrak g}})$ satisfying:
\begin{itemize}
\item $h$ is a GFF with zero boundary conditions in $D$.
\item For $k \ge 0$, let $Q_k$ be the scaling limit of the loops in $Q_k^\delta$. Then $Q_0$
is equal to $\lp_{-\sqrt 2 \lambda, \sqrt 2 \lambda}(h)$. 
Moreover, for every loop $\gamma\in Q_0$, $h$ restricted to $O(\gamma)$ has boundary value $\coin_{{\mathfrak g}} c(\gamma) \sqrt{2}\lambda$.

\item This picture repeats iteratively: if $\ell$ is a loop in $Q_k$, then all the loops in $Q_{k+1}$ directly encircled by $\ell$ form $\lp_{-\sqrt 2 \lambda, \sqrt 2 \lambda}(h^0|_{O(\ell)})$, and for each
such loop $\gamma$, $h^0|_{O(\ell)}$ restricted to $O(\gamma)$ has boundary value 
\[
(-1)^kc(\gamma) c(\ell) \sqrt{2}\lambda.
\]
\end{itemize}
\end{theorem}

\begin{remark} \label{rem:increments}
The difference in the gaps between the first layer and the remaining layers ($\coin_{{\mathfrak g}} c(\gamma) \sqrt{2}\lambda$ and $(-1)^k c(\gamma)c(\ell) \sqrt{2}\lambda$ respectively) comes from the fact that in the master coupling of Theorem~\ref{thm:mastercoupling}, the label of the external most cluster of $\n^\dagger$ is chosen uniformly at random, whereas the increment of the heights between loops in consecutive layers is given 
by Property~\eqref{itm:cp6}. Here we also use Property~\eqref{itm:cp3} to see that for a primal cluster $\mathcal C$, one has $\epsilon_{\mathcal C}=-c(\gamma)$, where $\gamma$ is the loop in $Q^\delta$ that surrounds and touches $\mathcal C$. An analogous formula holds for dual clusters.
The alternating sign $(-1)^{k}$ appears since $Q^\delta_k$ alternate between primal and dual interfaces, and the formula in Property~\eqref{itm:cp6} changes sign depending if we compute the increment from a face or from a vertex of the original graph.
\end{remark}

\begin{figure}[t]
\centering
\includegraphics[width=0.48\textwidth]{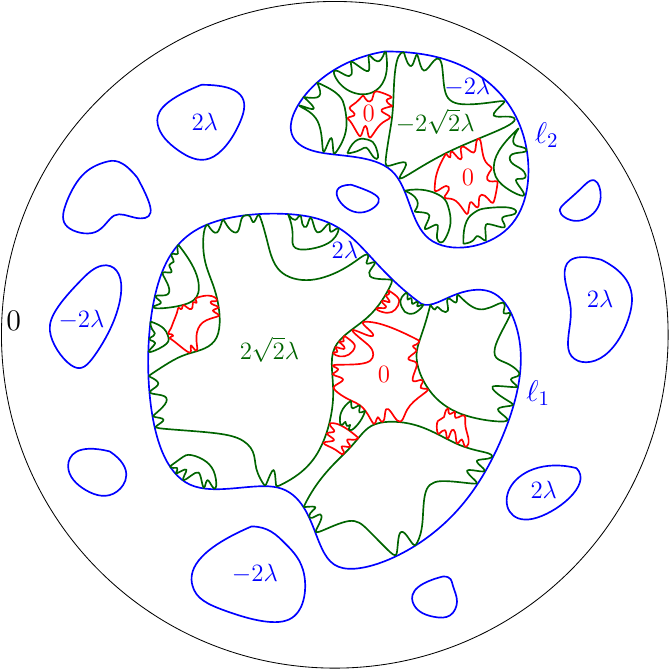}
\caption{Scaling limit of the boundaries of the outermost clusters on the primal graph. We depicted the outermost loops of $B$ in blue. For each blue loop $\ell$, the loops in $A(\ell)$ have boundary value either $0$ or $\epsilon(\ell) 2\sqrt{2}\lambda$.
For two blue loops $\ell_1$ and $ \ell_2$ with labels $\epsilon(\ell_1)=1$ and $\epsilon(\ell_2)=-1$, we depict the loops in $A(\ell_1)$ and $A(\ell_2)$. For $i=1,2$, we draw the loops in $A(\ell_i)$ with boundary value $0$ (resp.\ $\epsilon(\ell_i) 2\sqrt{2}\lambda$) in red (resp.\ green). Each green (resp.\ red) loop is the limit of the boundary of an odd (resp.\ even) hole.}  
\label{fig:RCC}
\end{figure}

\begin{theorem}
\label{thm:ABh}
Let $D$ and $D^\delta$ be as above and such that $\partial D$ is $C^1$.
As $\delta\to 0$, the family 
\[
(H^\delta, B^\delta, A^\delta,  \epsilon^\delta, c^\delta)
\] 
defined above converges in distribution to a limit $(\tfrac{1}{\sqrt \pi}h,B, A,\epsilon, c)$ satisfying (see Fig.~\ref{fig:RCC}):\begin{itemize}
\item $h$ is a GFF with zero boundary conditions in $D$.
\item The collection of outermost loops in $B$ is equal to CLE$_4(h)$. For each such loop $\ell$, $h|_{O(\ell)}$ is equal to an independent zero-boundary GFF $h^0|_{O(\ell)}$ plus the constant $\epsilon(\ell) 2\lambda$.

\item For each such outermost loop $\ell$ of $B$, let $A(\ell)$ denote the collection of loops $\gamma$ in $A$ that are directly encircled by $\ell$ (no other loop in $A$ encircles $\gamma$).
\begin{itemize}
\item If $\epsilon(\ell)=1$, then $A(\ell)$ is equal to $\lp_{-2\lambda, (2\sqrt{2}-2)\lambda}(h^0|_{O(\ell)})$. 
\item If  $\epsilon(\ell)=-1$, then $A(\ell)$ is equal to $\lp_{-(2\sqrt{2}-2)\lambda, 2\lambda}(h^0|_{O(\ell)})$.
\item Each loop $\gamma\in A(\ell)$ has boundary value $\epsilon(\ell) (c(\gamma) +1)\sqrt{2} \lambda$.
\end{itemize}
\item This picture repeats iteratively in each outermost loop $\ell$ of $A$ (with $\partial D:=\ell$, and with the loops of $B$ and $A$ encircled by $\ell$).
\end{itemize}
\end{theorem}

The relation between the loops in $Q$ an $A, B$ is illustrated in Fig.~\ref{fig:RCC_coupling}.

\begin{figure}[htbp]
\centering
\begin{subfigure}{0.8\textwidth}
\centering
\includegraphics[width=.59\textwidth]{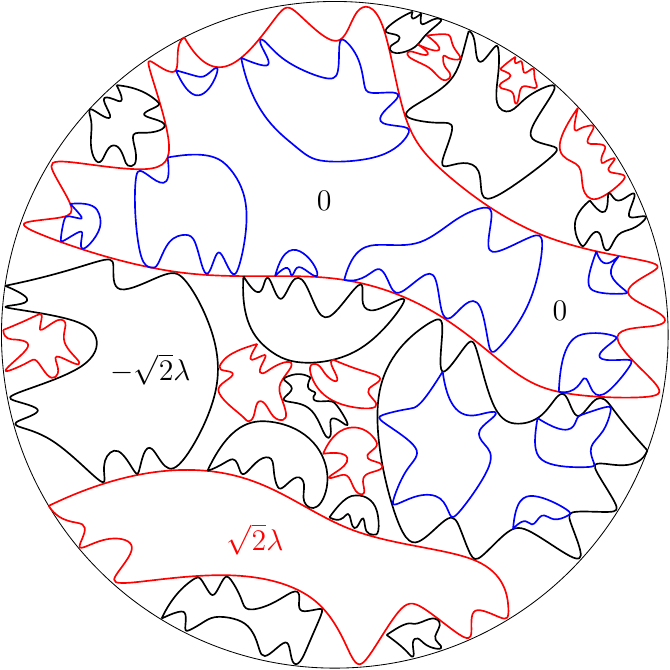}
\caption{The red and black loops represent the limit of the wired d.r.c.\ interfaces, i.e., the loops in $Q_0$. The red (resp.\ black) loops are the limit of even (resp.\ odd) holes, and are distributed as $\Lc^+_{-\sqrt{2}\lambda, \sqrt{2}\lambda}(h)$ (resp.\ $\Lc^-_{-\sqrt{2}\lambda, \sqrt{2}\lambda}(h)$). The blue loops are the loops in $B_1$ that touch a loop in $Q_0$.}
\end{subfigure}
\medbreak

\begin{subfigure}{0.8\textwidth}
\centering
\includegraphics[width=.6\textwidth]{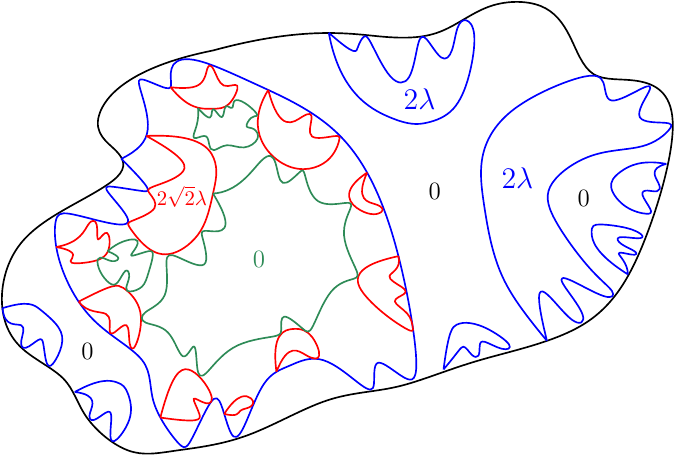}
\caption{The exterior black loop represents a loop $\ell$ in $Q_0$ with boundary value $\sqrt{2}\lambda$. The blue loops are the loops in $B_1$ that touch $\ell$, and are distributed as $\lp^+_{-\sqrt{2}\lambda, (2-\sqrt{2})\lambda} (h^0|_{O(\ell)})$ (each of the blue loops has boundary value $2\lambda$). The complement of the blue loops constitutes the set $Q_1(\ell)\setminus A_1(\ell)$, and is distributed as $\lp^-_{-\sqrt{2}\lambda, (2-\sqrt{2})\lambda} (h^0|_{O(\ell)})$ (each loop in this set has boundary value $0$). Each blue loop is further split into a two-valued set with loops of boundary values $0$ (in green) and $2\sqrt{2}\lambda$ (in red). The green and red loops constitute the set $A_1(\ell)$.}
\end{subfigure}
\caption{The nesting between the loops in $Q, A, B$ and their coupling with $h$. For each set of discrete loops at meshsize $\delta$, we take away the superscript $\delta$ to denote its scaling limit in the continuum. For example, $A_k$ and $B_k$ denote respectively the scaling limit of the loops in $A_k^\delta$ and $B_k^\delta$. We point out that an outermost loop in $B$ or $A$ can be at an arbitrary level, i.e.\ belongs to $B_k$ or $A_k$ for any $k$.}
\label{fig:RCC_coupling}
\end{figure}

\begin{remark}
We can deduce using crossing estimates from \cite{DumLisQia21} that for each loop $\ell \in B_k$, two loops in $A_k(\ell)$ of the same parity (hence of the same boundary value  and drawn in the same color in Fig.~\ref{fig:RCC}) never touch each other. Moreover, only the limit of the boundaries of odd holes can touch $\ell$. This is consistent with Theorem~\ref{thm:ABh} and the adjacency properties of the loops in a two-valued set (Lemma~\ref{lem:adjacency}). Furthermore, Theorem~\ref{thm:ABh} implies that each loop in $A_k(\ell)$ is connected to $\ell$ via a finite chain of loops of alternating parities (hence the length of this chain always has a fixed parity). In particular, the parity of the holes are determined by the shape of the clusters.
\end{remark}

\begin{remark}\label{rmk:touch}
We can deduce using crossing estimates from \cite{DumLisQia21} that two loops in $Q_0$ of the same parity (hence of the same boundary value  and drawn in the same color in Fig.~\ref{fig:RCC_coupling}) never touch each other. This is consistent with Theorem~\ref{thm:Qh} and the adjacency properties of the loops in a two-valued set (Lemma~\ref{lem:adjacency}).
\end{remark}

\subsection{Precompactness and first properties of limiting curves}
We now proceed to proving the two theorems. To this end, recall the tightness criterion \cite[{\bf H1}]{AizBur99}: a family of random variables $\mathcal F^\delta$ (with law $\mathbb P_\delta$) taking values in $\mathfrak C(\Omega)$ satisfies condition {\bf H1} if for every $k<\infty$ and every annulus $\mathrm{Ann}(x,r,R)$ with $\delta\le r\le R\le 1$, the following bound holds uniformly
in $\delta>0$:
\begin{equation} \label{eq:H1}
\mathbb P_\delta[N_{\mathcal F^\delta}(\mathrm{Ann}(x,r,R))\geq k]\le C(k)(\tfrac rR)^{\lambda(k)},
\end{equation}
with $C(k)>0$ and $\lambda(k)$ tending to infinity as $k\to \infty$, and
where
\begin{align}\label{eq:H12}
N_{\mathcal F^\delta}(\mathbf{A}) =\{k \text{ distinct pieces of curves in $\mathcal F^\delta$ cross the annulus $\mathbf{A}$}  \}.
\end{align}
Here by a \emph{piece} (in $\mathbf{A}$) of a curve we mean a connected component of the curve resulting from a restriction of the curve to the annulus $\mathbf{A}$.  
If $\mathcal F^\delta$ contains only one curve $\ell$, we will simply write $N_{\ell}$ for $N_{\{\ell\}}$.
Theorem~1.2 of \cite{AizBur99} says that if $\mathcal F^\delta$ satisfies condition {\bf H1}, then $\mathcal F^\delta$ is precompact for the topology of weak convergence 
 with respect to the distance~\eqref{eq:CNdistance}.

\begin{proposition} \label{prop:etaH1}
Let $D$, $D^\delta$ and $\mathbb P_{D^\delta}$ be as above.
Let $\eta^\delta$ (resp.~$\widehat \eta^\delta$) be the nested boundaries interface configuration of $\n^\delta$ (resp.~$(\n^\dagger)^\delta$) as defined in Section~\ref{intro:drc}. We view $\eta^\delta$ and $\widehat \eta^\delta$ as collections of loops, so that $\eta^\delta\cup \widehat \eta^\delta=A^\delta\cup B^\delta $. Then $\eta^\delta$ satisfies condition {\bf H1} under $\mathbb P_{D^\delta}$.
Moreover if $\partial D$ is $C^1$, then $\widehat \eta^\delta$ also satisfies condition {\bf H1} under~$\mathbb P_{D^\delta}$.
\end{proposition}
\begin{proof}
We apply criterion~{\bf H1} to the families $\eta^\delta$ and $\widehat \eta^\delta$.
The event that $\mathrm{Ann}(x,r,R)$ is crossed by $k$ separate pieces of interfaces in $\eta^\delta$ (resp.~$\widehat \eta^\delta$) is included in the (rescaled version of the) event $A_{2k}(r/\delta,R/\delta)$ for $\n$ 
(resp.~$(\n^\delta)^\dagger$), so that we may apply Theorem~\ref{prop:tight number crossings} and Remark~7.3 of \cite{DumLisQia21}.
 This concludes the proof.
\end{proof}

\begin{lemma} \label{lem:QH1}
Let $D$, $D^\delta$,  $\mathbb P_{D^\delta}$, and $Q_k^\delta$ be as above. Assume moreover that $\partial D$ is $C^1$. Then for each $k\geq 0$, $Q_k^\delta$ satisfies condition {\bf H1} under $\mathbb P_{D^\delta}$. 
\end{lemma}

\begin{proof}
We will say that two (pieces of) loops are \emph{adjacent} if {either} they are both subsets of the same graph (primal or dual) and moreover they intersect, 
or they are subsets of mutually dual graphs and moreover they visit {at least one} same corner (pair of vertex and face) of the primal graph. 

We will proceed inductively. By Proposition~\ref{prop:etaH1}, $Q_0^\delta$ satisfies {\bf H1} since it is a subset of $\hat \eta^\delta$.
Let us hence assume that $Q_k^\delta$ satisfies {\bf H1}. Suppose $k$ is even (the case of $k$ odd is treated analogously).
Let us show that $Q_{k+1}^\delta$ also satisfies {\bf H1} (with properly adjusted constants in~\eqref{eq:H12}).
For $\ell \in Q_k^\delta$, let $L(\ell)=Q_{k+1}^\delta(\ell) \setminus \eta^\delta$, and $L=Q_{k+1}^\delta\setminus \eta^\delta=\bigcup_{\ell \in Q_k^\delta}L(\ell)$.
Note that by Proposition~\ref{prop:etaH1}, it is enough to prove that $L$ satisfies {\bf H1}.

\begin{figure}[t]
\centering
\includegraphics[width=0.65\textwidth]{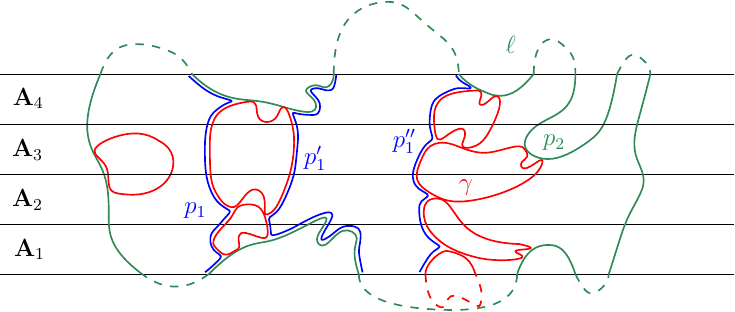}
\caption{Illustration of the proof of \eqref{eq:Nbound}. The proof is on the discrete lattice, but we depict curves in the continuum for convenience. Adjacent pieces of loops are depicted as curves that touch or trace each other.
We depict in green the loop $\ell$, and in red the loops in $\eta^\delta(\ell)$. Note that in the discrete, the loops in $\eta^\delta(\ell)$ can be adjacent to each other (i.e., the red loops can touch each other), even though Theorem~\ref{thm:Qh} (that we will prove later) implies that the scaling limit of the loops in $\eta^\delta(\ell)$ a.s.\ do not touch each other (see Remark~\ref{rmk:touch}).
The pieces in $P_2(\ell)$ are drawn in solid (green or red) curves. We depict in blue $3$ pieces $p_1, p_1', p_1''$ in $P_1(\ell)$ (among several others). We can see that $p_1$ and $p_1'$ are adjacent to the same pieces in $P_2(\ell)$, which is why the constant $2$ in \eqref{eq:Nbound} is needed. In this picture, there are $5$ pieces in $P_2(\ell)$ which are adjacent to $p_1''$, but none of them makes the required crossing across $\mathbf{A}_1, \ldots, \mathbf{A}_4$. However, in this case the pieces in $P_2(\ell)$ adjacent to $p_1''$ must contain at least one full loop $\gamma$ from $\eta^\delta(\ell)$ which is strictly contained in $\mathbf{A}_2 \cup \mathbf{A}_3$. Then $\gamma$ must be adjacent to some piece $p_2$ of the loop $\ell$ which is connected to $\partial \mathbf{A}$. Here $p_2$ crosses $\mathbf{A}_4$.  Note that $p_2$ is not directly adjacent to $p_1''$, but is adjacent to some piece adjacent to $p_1''$, so it is contained in $P_2(p_1'')$ by our definition.}  
\label{fig:tightness}
\end{figure}

To this end, we will use the fact that the loops in $L$ are constructed from (pieces of) a loop $\ell\in Q_k^\delta$ or/and pieces of the loops in $\eta^\delta$ that are adjacent to $\ell$ from the inside (see also Remark~\ref{rem:looptracing}).
Let us denote the latter collection of loops by $\eta^\delta(\ell)$.
Consider annuli $\mathbf{A}=\mathrm{Ann}(x,r,R)$, and $\mathbf{A}_i=\mathrm{Ann}(x,rs^{i-1},rs^i)$, where $i=1,2,3,4$ and $s=\sqrt[4]{R/r}$, so that $\mathbf{A}=\mathbf{A}_1\cup \mathbf{A}_2\cup \mathbf{A}_3
\cup \mathbf{A}_4$.
Since $\eta^\delta(\ell)$ and $\eta^\delta(\ell')$ are disjoint for $\ell\neq \ell'$,
by Proposition~\ref{prop:etaH1} and the induction assumption, it is enough to show that for each $\ell \in Q_k^\delta$, 
\begin{align}\label{eq:Nbound}
N_{L(\ell)}(\mathbf{A})  \leq 2 (N_{ \eta^\delta(\ell) \cup\{ \ell \}}(\mathbf{A}_1) + N_{\eta^\delta(\ell)  }(\mathbf{A}_2) +N_{\eta^\delta(\ell)  }(\mathbf{A}_3)+ N_{ \eta^\delta(\ell) \cup\{ \ell \}}(\mathbf{A}_4) ) .
\end{align}

To this end, let $P_1(\ell)$ be the set of all pieces in $\mathbf{A}$ of the loops in $L(\ell)$ that cross $\mathbf{A}$ (as defined above). Then the cardinality of $P_1(\ell)$ is equal to $N_{L(\ell)}(\mathbf{A}) $.
Moreover let $P_2(\ell)$ be the set of all pieces in $\mathbf{A}$ of the loops in $\eta^\delta(\ell) \cup \{ \ell\}$ (not necessarily crossing $\mathbf{A}$). Here if a loop is fully contained in $\mathbf{A}$, then 
there is one piece which is equal to this loop.
Furthermore, for each $p_1\in P_1(\ell)$, let $P_2(p_1)\subset P_2(\ell)$ be the set of pieces in $P_2(\ell)$ that are adjacent to $p_1$, or are adjacent to another piece adjacent to $p_1$. As we mentioned earlier, the union of the pieces in $P_2(p_1)$ should entirely cover $p_1$, by Remark~\ref{rem:looptracing}.
The pieces in $P_2(\ell)$ are of two kinds: either they come from $\ell$ or $\eta^\delta(\ell)$. 
See Fig.~\ref{fig:tightness} for an illustration.
Let $p_1\in P_1(\ell)$. Since $\ell$ encircles all loops in $\eta^\delta(\ell)$,
every piece in $P_2(p_1)$ that is a piece of a loop in $\eta^\delta(\ell)$ is either itself connected to the boundary of $\mathbf{A}$ or it is connected to it via a single piece of $\ell$ in $P_2(p_1)$.
This means that there are two possibilities: 
either there exists a piece in $P_2(p_1)$ 
that crosses $\mathbf{A}_1$ or $\mathbf{A}_4$, or there exists a piece in $P_2(p_1)$ that is a full loop and crosses either $\mathbf{A}_2$ or $\mathbf{A}_3$. Indeed, suppose that none of the two possibilities is true, then there must exist a full loop $\gamma$ in $P_2(p_1)$ that is entirely contained in $\mathbf{A}_2\cup \mathbf{A}_3$. Note that $\gamma$ must be connected to the boundary of $\mathbf{A}$ via a single piece of $\ell$ in $P_2(p_1)$. The latter piece in $P_2(p_1)$ then has to cross either $\mathbf{A}_1$ or $\mathbf{A}_4$, leading to a contradiction.
Observe moreover that by planarity each $p_2\in P_2(p_1)$ belongs to at most two sets of the form $P_2(p'_1)$ for some $p'_1\in P_1(\ell)$, since a crossing piece $p_1$ can
follow $p_2$ from at most two sides (this bound is not optimal, but sufficient for our purpose), hence the constant $2$ in \eqref{eq:Nbound}. This shows~\eqref{eq:Nbound} and finishes the proof of the lemma.
\end{proof}

Finally we will need the following intersection properties of the limiting interfaces.

\begin{lemma}  \label{lem:intersectionprop}
Let $(A,B,c)$ be any subsequantial limit of $(A^\delta,B^\delta,c^\delta)$. Then
\begin{itemize}
\item the loops in $B$ are simple and do not intersect each other,
\item the outermost loops in $B$ do not intersect the outermost loops in $A$ with $c=-1$.
\end{itemize}
\end{lemma}
\begin{proof}
The fact that the loops in $B$ do not intersect each other is a direct consequence of Theorem~\ref{thm:absence closed pivotal expectation white}. Indeed, fix $\alpha,\beta,\ep>0$. For two loops of $B^{\delta}$ of diameter at least $\alpha$ to come within distance $\beta$ of each other, there must be $x\in \Omega^{\delta}$ such that the translate by $x$ of the rescaled version of the event $A_4^\square(\beta/\delta,\alpha/\delta)$ occurs. Yet, Theorem~\ref{thm:absence closed pivotal expectation white} implies that provided that $\beta\le \beta_0(\alpha,\ep)$, this occurs with probability smaller than $\ep$. 
The fact that the loops in $A$ and $B$ are simple is also direct consequence of Theorem~\ref{thm:absence closed pivotal expectation white}. Indeed, the event that a single loop comes within distance $\beta$ of itself after going away to distance $\alpha$ also implies the same event. Letting $\beta$ tend to zero, then $\alpha$, and finally $\ep$, we obtain the result.

Moreover, for a loop of $A^{\delta}$ of diameter at least $\alpha$ and with boundary value zero (and hence $c=-1$) to come within a distance $\beta$ of an outermost loop in $B^{\delta}$ of diameter at least $\alpha$, there must be $x\in D^{\delta}$ such that the translate by $x$ of the rescaled version of the event $A_4^\blacksquare(\beta/\delta,\alpha/\delta)$ occurs. Yet, Theorem~\ref{thm:absence closed pivotal expectation black} implies that provided that $\beta\le \beta_0(\alpha,\ep)$, this occurs with probability smaller than $\ep$. Letting $\beta$ tend to zero, then $\alpha$, and finally $\ep$, we obtain the result.
\end{proof}
 
\subsection{Identification of limits}

We start with a lemma that proves the first two bullets of Theorem~\ref{thm:Qh}.
 
  \begin{lemma} \label{lem:Qfirst} Let $D$, $D^\delta$,  $\mathbb P_{D^\delta}$, and $Q_0^\delta$ be as above.  Assume moreover that $\partial D$ is $C^1$. 
 Let $\coin^\delta_{{\mathfrak g}}$ be the label of the cluster of the boundary in $(\n^\dagger)^{\delta}$.
Then the family $((h^\dagger)^\delta,Q_0^{\delta},c^\delta,\coin^\delta_{{\mathfrak g}})$ converges weakly to $(\frac{1}{\sqrt \pi}h,Q_0,c,\coin_{{\mathfrak g}} )$ as $\delta\to 0$, where 
\begin{itemize}
\item $h$ is a GFF with zero boundary conditions in $D$. 
\item $Q_0=\mathcal L_{-\sqrt 2 \lambda, \sqrt 2 \lambda}(h)$.
\item For each $\ell \in Q_0$, $h$ restricted to $O(\ell)$ has boundary value $\coin_{{\mathfrak g}}  c(\ell) \sqrt{2}\lambda$.
\end{itemize}
 \end{lemma}
\begin{proof}
By Lemma~\ref{lem:QH1}, Theorem~\ref{thm:NFconv}, and the compactness of $\{-1,1\}^\mathbb N$, $((h^\dagger)^\delta,Q_0^{\delta},c^\delta,\coin^\delta_{{\mathfrak g}})$ is precompact in the topology of week convergence.
Let $(\frac{1}{\sqrt \pi}h,Q_0,c,\coin_{{\mathfrak g}})$ be a limit along a subsequence~$\delta_n$. We also know from Theorem~\ref{thm:NFconv} that $h$ is the GFF in $D$ with zero boundary conditions. 

We will identify $\gask(Q_0)$ as the only two-valued set of $h$ with boundary values $\pm \sqrt 2 \lambda$. To this end we need to show that $\gask(Q_0)$ is thin for $h$, i.e. for any smooth bounded function~$g$,
we have
\begin{align*}
\int_D g(x) h(x) dx= \sum_{\gamma\in Q_0} \int_{O(\gamma)} g(x) h|_{O(\gamma)}(x) dx.
\end{align*}

Note that $\gask(Q^\delta_0)\subset\gask(A^\delta_0)$ by the master coupling from Theorem~\ref{thm:mastercoupling}, and moreover $h^\delta$ is zero on $\gask(A^\delta_0)$.
Furthermore $(h^\dagger)^\delta$ and $h^\delta$ have a common scaling limit $\frac{1}{\sqrt \pi}h$ by Theorem~\ref{thm:NFconv}. Therefore it is enough to show the following (here we prefer to look at 
$\gask(A^\delta_0)$ as it deals with double random currents with free boundary conditions, and these are more amenable to analysis as already mentioned)
\begin{equation}\label{eq:ah1}
\lim_{\alpha\rightarrow 0}\lim_{n\rightarrow\infty}\int_{D^{\delta_n}}g(x)h^{\delta_n}(x)\mathbf 1_{x\in E_\alpha^{\delta_n}}dx=0,
\end{equation}
where, if $\Lambda_\alpha(y):=y+[-\alpha,\alpha]^2$, 
\[
E_\alpha^{\delta_n}:=\text{ union of the $\Lambda_\alpha(y)$ for $y\in \alpha\mathbb Z^2$ such that $\Lambda_{2\alpha}(y)\text{ intersects some }\gamma\in A_0^{\delta_n}$}
\]  (note that in particular every $x$ that is within a distance $\alpha$ of some $\gamma$ in $A_0^{\delta_n}$ must be in $E_\alpha^{\delta_n}$). 

In order to prove this statement, we fix $\ep>0$ and see that
\begin{align*}
\ep\mathbb P_{\delta_n}\Big[\int_{D^{\delta_n}}g(x)h^{\delta_n}(x)\mathbf 1_{x\in E_\alpha^{\delta_n}}dx\ge \ep\Big]&\le \mathbb E_{\delta_n}\Big[\Big|\int_{D^{\delta_n}}g(x)h^{\delta_n}(x)\mathbf 1_{x\in E_\alpha^{\delta_n}}dx\Big|\Big]\\
&\le \sum_{y\in \alpha\mathbb Z^2}\mathbb E_{\delta_n}\Big[\Big|\int_{\Lambda_\alpha(y)}g(x)h^{\delta_n}(x)\mathbf 1_{x\in E_\alpha^{\delta_n}}dx \Big|\Big]\\
&=\sum_{y\in \alpha\mathbb Z^2}\mathbb E_{\delta_n}\Big[\mathbf 1_{y\in E_\alpha^{\delta_n}}\Big|\int_{\Lambda_\alpha(y)}g(x)h^{\delta_n}(x)dx \Big|\Big]\\
&\le\sum_{y\in \alpha\mathbb Z^2}\mathbb P_{\delta_n}[y\in E_\alpha^{\delta_n}]^{1/2}\mathbb E_{\delta_n}\Big[\Big(\int_{\Lambda_\alpha(y)}g(x)h^{\delta_n}(x) dx \Big)^2\Big]^{1/2}\\
&\le \sum_{y\in \alpha\mathbb Z^2}\alpha^c \times C(g)\alpha^2\log(1/\alpha)\\
&\le C(g,D)\log(1/\alpha)\alpha^{c}.\end{align*}
Above, we used Markov's inequality in the first inequality, the triangle inequality in the second, the fact that $x\in E_\alpha^{\delta_n}$ is equivalent to $y\in E_\alpha^{\delta_n}$ in the third, and Cauchy--Schwarz in the fourth. In the fifth, we combine an easy estimate on the second moment of $\int_{\Lambda_\alpha(y)}g(x)h^{\delta_n}(x)dx$ based on the definition of the nesting field and RSW type estimates from \cite{DumLisQia21}, together with the fact that for $\Lambda_\alpha(y)$ to intersect a loop $\gamma$ in $A^{\delta_n}$, there must be a primal path in $\n^\delta$ from $\Lambda_{\alpha}(x)$ to $\Lambda_\beta(x)$ or a path in $(\n^\delta)^*$ (the dual complement) from $\partial\Lambda_\beta(x)$ to $\partial\Lambda_{d(x,\partial D)}(x)$, where $\beta:=\sqrt{\alpha d(x,\partial D)}$.
This proves that $\gask(Q^\delta_0)$ is thin for $h$.

Moreover by the Markov property of the nesting field with wired boundary conditions~\eqref{def:nestingfield+} and Theorem~\ref{thm:NFconv} applied inside each loop of $Q_0^\delta$, we know that $\gask(Q_0)$ is a local set of~$h$, and that for each $\gamma \in Q_0$, the restriction of $h$ to $O(\gamma)$ has boundary value equal to $ \coin_{{\mathfrak g}}  c(\ell) \sqrt{2} \lambda\in \{-\sqrt 2\lambda, \sqrt 2\lambda\}$ 
(since in the discrete the boundary value is equal to $\pm \frac12$ and the scaling limit of $h^\delta$ is $ \frac{1}{\sqrt{\pi}}  h = \frac{1}{2\sqrt{2}\lambda}h$). By Lemma~\ref{lem:tvs} this uniquely characterizes $Q_0$ as the two-valued set $\lp_{-\sqrt 2\lambda, \sqrt 2\lambda}(h)$.
\end{proof}

\begin{proof}[Proof of Theorem~\ref{thm:Qh}]
By the lemma above, we are left with proving the third bullet from the statement. By the definition of $Q^\delta_k$ and by the Markov property~\eqref{itm:bc} of the master coupling from Theorem~\ref{thm:mastercoupling} we know that 
the loops of $Q^\delta_{k+1}$ contained in a single loop $\ell$ of $Q^\delta_k$, have the same distribution as $Q^\delta_0$ in a domain $D^\delta$ whose outer boundary is $\ell$.
However, we cannot directly apply Lemma~\ref{lem:Qfirst} since the assumption on the boundary of the domain being smooth is not satisfied by the scaling limits of the loops from $Q_0^\delta$ (as they are fractal loops by Lemma~\ref{lem:Qfirst}). Nonetheless, this assumption is only used to obtain subsequantial limits of the loops. Indeed, the proof of convergence of the height function in Theorem~\ref{thm:NFconv}
and of the fact that the gasket of the limiting collection of loops is thin in Lemma~\ref{lem:Qfirst}
works for Jordan domains with arbitrary boundaries as it goes through currents with free boundary conditions (and we have more control on them as already mentioned). The remaining ingredient of the proof
is the Markov property that is the same both for random currents with free and wired boundary conditions.

Therefore to prove the third bullet it is enough to use precompactness of $Q^\delta_k$ (which follows directly from Lemma~\ref{lem:QH1}) 
and show that every subsequential limit of $\gask(Q^\delta_k)$ is a thin local set (as in the proof of Lemma~\ref{lem:Qfirst}).
Then use Remark~\ref{rem:increments} to identify the signs of boundary values of the field on consecutive loops in the continuum, and use Lemma~\ref{lem:nestedident} to identify the limit uniquely.
\end{proof}

\begin{proof}[Proof of Theorem~\ref{thm:ABh}]
By Theorem~\ref{thm:NFconv}, Proposition~\ref{prop:etaH1} and compactness of $\{\pm 1\}^{\mathbb N}$, we know that $(A^\delta,B^\delta,h^\delta,c^\delta,\epsilon^\delta)$ is precompact in the topology of weak convergence. Let $(A,B,h,c,\epsilon)$ be any subsequential limit. 

Note that from Theorem~\ref{thm:Qh} and Remark~\ref{rem:looptracing} we already know that the loops in $A$ are a subset of all the loops in the union of nested iterations of $\mathcal L_{-\sqrt 2 \lambda, \sqrt 2 \lambda}$.
However, we need an additional argument to uniquely determine exactly which subset they are.
To be more precise, recall from Remark~\ref{rem:looptracing} that $A_k^\delta\subset Q_{k+1}^\delta$.
Theorem~\ref{thm:Qh} implies that if $\ell$ is a scaling limit of $\ell^\delta\in Q_{k}^\delta$, then the scaling limits of loops in $Q_{k+1}^\delta(\ell^\delta)$ is
$\mathcal L_{-\sqrt 2 \lambda, \sqrt2 \lambda}(h^0|_{O(\ell)})$. We claim that the scaling limit of $Q_{k+1}^\delta(\ell^\delta) \setminus A_k^\delta(\ell^\delta)$
is exactly the set of loops in $\mathcal L_{-\sqrt 2 \lambda, \sqrt2 \lambda}(h^0|_{O(\ell)})$ that have label $(-1)^{k+1}c(\ell)\sqrt 2 \lambda$ and moreover intersect $\ell$. Equivalently, by Lemma~\ref{lem:loop_intersect} 
this is exactly $\mathcal L^-_{-\sqrt 2 \lambda, (2-\sqrt2) \lambda}(h^0|_{O(\ell)})$ if $c(\ell)=(-1)^k$, and $\mathcal L^+_{ -(2-\sqrt2) \lambda, \sqrt 2 \lambda}(h^0|_{O(\ell)})$ if $c(\ell)=(-1)^{k+1}$.
Indeed by 
property~\ref{itm:cp6} of the master coupling from Theorem~\ref{thm:mastercoupling}
the increment of the nesting field between $\ell^\delta$ and $\gamma^\delta\in Q_{k+1}^\delta(\ell^\delta)$ is $(-1)^kc(\ell^\delta)c(\gamma^\delta)$.
The loops in $Q_{k+1}^\delta(\ell^\delta) \setminus A_k^\delta(\ell^\delta)$ are boundaries of even holes as mentioned below Remark~\ref{rem:looptracing}, and hence $c(\gamma^\delta)=-1$ for every such loop $\gamma^\delta$.
 Altogether this means that all loops in the scaling limit of $Q_{k+1}^\delta(\ell^\delta) \setminus A_k^\delta(\ell^\delta)$ have label $(-1)^{k+1}c(\ell)\sqrt 2 \lambda$. 
To prove the claim we still need to show that they are boundaries of exactly those even holes in $Q_{k+1}^\delta(\ell^\delta)$ whose scaling limit intersects~$\ell$.
Here is where we use the intersection properties from Lemma~\ref{lem:intersectionprop}. First of all, every loop $\gamma^\delta\in  A_k^\delta(\ell^\delta)$ is by definition encircled by a loop in $B_k^\delta$,
which in turn is encircled by $\ell^\delta$. If $\gamma^\delta$ is the boundary of an even hole, then its scaling limit cannot intersect~$\ell$, as in this case it would intersect the scaling limit of the corresponding loop in 
 $B_k^\delta$, which is forbidden by bullet two of Lemma~\ref{lem:intersectionprop}. Hence, it is enough to show that every loop in $Q_{k+1}^\delta(\ell^\delta) \setminus A_k^\delta(\ell^\delta)$ has a scaling limit that intersects~$\ell$. To this end recall from Remark~\ref{rem:looptracing} that each such loop traces pieces of loops in $B^\delta_k$ that touch $\ell^\delta$ and/or pieces of $\ell^\delta$ itself. If its scaling limit does not intersect~$\ell$ it means that it can only trace pieces of scaling limits of loops from $B^\delta_k$ that intersect $\ell$. However, that would imply that these loops either touch each other or self-touch which is forbidden by bullet one of Lemma~\ref{lem:intersectionprop}.
 
 We now move on to the identification of the scaling limit of the outermost loops of $B$ as CLE$_4(h)$ using Lemmas~\ref{lem:2lambda_intersect} and~\ref{lem:CLEiteration}. 
Our aim is to show that the continuum construction of Lemma~\ref{lem:CLEiteration} is mirrored in the discrete. Since we look at the outer boundaries of only the primal current $\n^\delta$,
the relevant auxiliary collections of loops will be $Q^\delta_{2k}$, $k=0,1,\ldots$.
Let $\ell^\delta\in Q^\delta_{2k}$, and recall that $B_{2k}(\ell^\delta)$ is the set of outer boundaries of $\n^\delta$ that touch $\ell^\delta$ from the inside, and let $B_{2k}(\ell)$ be the set of their scaling limits, where $\ell$ is the scaling limit of $\ell^\delta$.
 We claim that the restriction of the loops in 
 $B_{2k}(\ell)$ to $O(\ell)$ agrees with the restriction of $\mathcal L^-_{-\sqrt 2 \lambda, (2-\sqrt2) \lambda}(h^0|_{O(\ell)})$ to $O(\ell)$ if $c(\ell^\delta)=1$, and with the restriction of 
 $\mathcal L^+_{- (2-\sqrt2) \lambda,\sqrt 2 \lambda}(h^0|_{O(\ell)})$ if $c(\ell^\delta)=-1$. Without loss of generality let us assume that $c(\ell^\delta)=1$. Indeed by definition, the loops in $Q_{2k+1}^\delta(\ell^\delta) \setminus A_{2k}^\delta(\ell^\delta)$ restricted to the inside of $\ell^\delta$ follow pieces of loops from $B_{2k}(\ell^\delta)$. Reversely, the loops in $B_{2k}(\ell^\delta)$ follow pieces of loops in 
 $Q_{2k+1}^\delta(\ell^\delta) \setminus A_{2k}^\delta(\ell^\delta)$ unless the loops in $B_{2k}(\ell^\delta)$ come to distance one (see Fig.~\ref{fig:nestedMarkov} for an example). Since $B_{2k}(\ell)$ do not intersect each other, and by the paragraph above, inside $O(\ell)$ all loops from $B_{2k}(\ell)$ follow pieces of  $O(\ell)$ if $c(\ell^\delta)=1$. On the other hand, again by definition, the restriction of the loops in 
 $B_{2k}(\ell)$ to $\ell$ is the closure of the complement of the restriction of $\mathcal L^-_{-\sqrt 2 \lambda, (2-\sqrt2) \lambda}(h^0|_{O(\ell)})$ to $\ell$. Hence, by Lemma~\ref{lem:2lambda_intersect},
 $B_{2k}(\ell)$ is equal to $\mathcal L^+_{-\sqrt 2 \lambda, (2-\sqrt2) \lambda}(h^0|_{O(\ell)})$. This together with the construction of Lemma~\ref{lem:CLEiteration} that extracts the outermost loops from 
 $B$ proves that these outermost loops are CLE$_4(h)$.
 
 The fact that $A(\gamma)$ for every outermost loop $\gamma\in B$ is equal to $\lp_{-2\lambda, (2\sqrt{2}-2)\lambda}(h^0|_{O(\gamma)})$ if $\epsilon(\gamma)=1$, and to $\lp_{-(2\sqrt{2}-2)\lambda, 2\lambda}(h^0|_{O(\gamma)})$ 
 if $\epsilon(\gamma)=-1$ follows directly from the discussion above and the second part of Lemma~\ref{lem:loop_intersect}.
\end{proof}

\subsection{Asymptotic behavior of the number of clusters}\label{subsec:asym}
Let us now prove a lemma which leads to the asymptotic numbers of clusters in the double random current models that surround the origin. 

\begin{lemma}\label{lem:drc_radii}
In the scaling limit of the double random current model in the unit disk (with either the free or wired boundary conditions), let $N(\eps)$ be the number of clusters surrounding the origin such that their outer boundaries have a conformal radius w.r.t.\ the origin at least $\eps$. Then
\begin{align*}
N(\eps)/ \log(\eps^{-1}) \underset{\eps\to0}{\longrightarrow} 1/ (\sqrt{2}\pi^2).
\end{align*}
\end{lemma}

\begin{proof}
By Theorems~\ref{thm:cvg_free_drc_v1} and~\ref{thm:cvg_wired_drc_v1} and \cite[Proposition 20]{MR3936643}, we know that the difference of log conformal radii between the outer boundaries of two successive double random current clusters that encircle the origin is given by $R:=T_1+T_2$, where $T_1$ is the first time that a standard Brownian motion exits $[-\pi, (\sqrt{2}-1)\pi]$ and $T_2$ is the first time that a standard Brownian motion exits $[-\pi, \pi]$. We have
$$\Eb(T_1+ T_2) = (\sqrt{2}-1) \pi^2 + \pi^2= \sqrt{2}\pi^2. $$
The $n$-th cluster which encircles the origin has log conformal radius equal to $-S_n$ where $S_n:= -(R_1 + \cdots + R_n)$ and $R_i$ are i.i.d.\ random variables distributed like $R$. 
Then $N(\eps)$ is the smallest $n\ge 1$ such that $S_{n+1}\ge \log(\eps^{-1})$.
By the law of large numbers, we know that $S_n/n$ converges to $\Eb(R)$ a.s.\ as $n\to\infty$.
Since $N(\eps)\to\infty$ as $\eps\to 0$, we also have that 
$$S_{N(\eps)+1}/(N(\eps)+1) \to \Eb(R) \text{ a.s.  as } \eps\to 0.$$
Note that $\log(\eps^{-1}) \le S_{N(\eps)+1} \le \log(\eps^{-1}) + R_{N(\eps)}$. It follows that 
$$\lim_{\eps\to 0}  \log(\eps^{-1})/N(\eps) = \Eb(R) =\sqrt{2}\pi^2.$$
The inverse of the above equation proves the lemma.
\end{proof}

\bibliographystyle{amsplain}

\end{document}